\documentclass[12pt]{amsart}

\date{October 7, 2026}
\title{Braidings of Self-Equivalences and Bordism}

\author{Ian Hambleton}

\address{Department of Mathematics, McMaster University	L8S 4K1, Hamilton, Ontario, Canada}
\email{hambleton@mcmaster.ca }

\author{K\"ur\c sat S\"ozer}
\address{Universit\' e de Lille, CNRS, UMR 8524, Laboratoire Paul Painlev\' e, Lille, F-59000, France}
\email{kursat.sozer@univ-lille.fr}

\thanks{The authors were partially supported at McMaster by an NSERC Discovery Grant. KS was supported at University of Lille by CDP C2EMPI and the CPER WaveTech@HdF}

\usepackage[T1]{fontenc}
\usepackage{lmodern}
\usepackage[final]{microtype}
\usepackage[mathscr]{euscript}

\usepackage{color}
\usepackage[all]{xy}
\usepackage{amscd,amsmath,latexsym,amssymb}
\usepackage[title]{appendix}
\usepackage{mathrsfs}
\usepackage{tikz}
\usepackage{tikz-cd}
\usepackage{enumerate}
\usepackage{mathtools}

\usepackage{quiver}
\usepackage[normalem]{ulem}

\usepackage{xparse}
\usepackage{etoolbox}
\usepackage{aliascnt}
\usepackage{hyperref}
\hypersetup{
	colorlinks=true,
	linkcolor=blue,
	filecolor=blue,
	citecolor = blue,
	urlcolor=blue,
}

\usepackage[textsize = footnotesize]{todonotes}
\usepackage{comment}

\calclayout
\newtheorem{dummy}{anything}[section]
\newtheorem{theorem}[dummy]{Theorem}
\newtheorem*{thma}{Theorem A}
\newtheorem*{thmb}{Theorem B}
\newtheorem{lemma}[dummy]{Lemma}
\newtheorem{proposition}[dummy]{Proposition}
\newtheorem*{proposition*}{Proposition}
\newtheorem{corollary}[dummy]{Corollary}

\theoremstyle{definition}
\newtheorem{definition}[dummy]{Definition}
  \newtheorem{example}[dummy]{Example}
  
  \newtheorem{remark}[dummy]{Remark}
   \newtheorem*{nonum}{Theorem}
   
   \newtheorem*{assumption}{Assumption}
   \newtheorem*{convention}{Convention}
    \newtheorem*{question}{Question}
  \newtheorem*{acknowledgement}{Acknowledgement}

\newcommand
{\eqncount}{\setcounter{equation}{\value{dummy}}%
\addtocounter{dummy}{1}}

\newcommand{\cB}{\mathcal B}
\newcommand{\cC}{\mathcal C}

\newcommand{\cE}{\mathcal E}

\newcommand{\cH}{\mathcal H}

\newcommand{\cK}{\mathscr K}

\newcommand{\cW}{\mathscr W}

\newcommand{\bZ}{\mathbb Z}

\newcommand{\bbQ}{\mathbb Q}
\newcommand{\bbR}{\mathbb R}
\newcommand{\bR}{\mathbb R}

\newcommand{\bbZ}{\mathbb Z}

\newcommand{\wG}{\widehat  G}
\newcommand{\wb}{\widehat b}
\newcommand{\wm}{\widehat m}
\newcommand{\La}{\Lambda}

\newcommand{\wOmega}{\widehat{\Omega}}

\newcommand{\wF}{\widehat F}
\newcommand{\wH}{\widehat H}

\DeclareMathOperator{\Aut}{Aut}

\DeclareMathOperator{\Map}{Map}

\DeclareMathOperator{\Image}{Im}

\newcommand{\cy}[1]{\bZ/{#1}}

\newcommand{\la}{\langle}
\newcommand{\ra}{\rangle}

\newcommand{\vv}{\, | \,}
\newcommand{\bd}{\partial}
\DeclareMathOperator{\id}{id}
\DeclareMathOperator{\hatid}{{\widehat\id}}

\def\O{\Omega}

\DeclareMathOperator{\hepta}{Aut}

\newcommand{\hept}[1]{\hepta_{\bullet}(#1)}

\newcommand{\heqpt}[1]{\cE_{\bullet} (#1)}

\newcommand{\htildeMk}[1]{\widetilde \cH_{#1}(\Ml)}
\newcommand{\htildeBk}[1]{\widetilde \cH_{#1}(\Bl)}
\newcommand{\htilde}[1]{\widetilde \cH(#1)}
\newcommand{\htildek}[2]{\widetilde{\cH}_{#1}(#2)}

\newcommand{\wM}{\widetilde M}

\newcommand{\heqptk}[2]{\cE_{\bullet} (#1\la#2\ra)}
\newcommand{\heqptkc}[2]{\cE_{\bullet}^c (#1\la#2\ra)}
\newcommand{\post}[2]{P_{#2}({#1})}

\newcommand{\Bw}{B\langle w_2\rangle}
\newcommand{\Mw}{M\negthinspace\langle w_2\rangle}
\newcommand{\Mk}{M\la \ell \ra}
\newcommand{\Bk}{B\la \ell \ra}
\newcommand{\Ml}{M\la \ell \ra}
\newcommand{\Bl}{B\la \ell \ra}
\newcommand{\MlT}{M\la \ell \ra_{TOP}}
\newcommand{\BlT}{B\la \ell \ra_{TOP}}

\newcommand{\Fl}{\on{Fib}\la \ell \ra}

\newcommand{\wxi}{\tilde\xi}
\newcommand{\wpp}{\tilde p}
\newcommand{\hatpp}{\hat p}
\newcommand{\wqq}{\tilde q}

\newcommand{\lmin}{\ell_{\text{min}}}

\newcommand{\trf}{tr{\hskip-1.4pt}f}

\newcommand\on[1]{\operatorname{#1}}
\newcommand\nr[1]{\medskip\noindent \textbf{#1}}

\newcommand{\bX}{X_\bullet}

\newcommand{\sDelta}{\Delta_{\on{inj}}}
\newcommand{\sDeltaop}{\sDelta^{\on{op}}}
\newcommand{\Deltaop}{\Delta^{\on{op}}}
\newcommand{\holim}{\underset{\longleftarrow}{\operatorname{holim}}\hspace{1.5pt}}

\DeclareMathOperator{\Mock}{Mock}
\DeclareMathOperator{\wMock}{\widehat\Mock}

\begin{document}

\begin{abstract}
Let $M$ be a closed, smooth or topological $n$-manifold, with $n \geq 4$. We construct a homotopy highly cartesian  square relating the space  $\heqpt{\Ml}$ of homotopy self-equivalences of $M$  (in a suitable range) over the Postnikov $\ell$-sections of its stable normal microbundle,
	and an  $(\infty + n)$-fold loop space representing an associated (normal) bordism theory. This implies the existence of braids of interlocking exact sequences involving the homotopy
	groups of $\heqpt{\Ml}$ and certain Lashof bordism groups, leading to a conceptual explanation and broad generalization of earlier work of Hambleton--Kreck for closed, oriented $4$-manifolds.
\end{abstract}

\maketitle

\section{Introduction}\label{section:introduction}
Understanding the space of homotopy self-equivalences of a manifold is a central topic in algebraic topology, deeply connected to both the classification of manifolds and homotopy theory.
 For a topological space $X$, the space of pointed homotopy self-equivalences, denoted by $\heqpt{X}$, is a (group-like) topological monoid  representing the homotopically invertible maps from $X$ to itself. 
For a smooth or topological $n$-manifold, 
these spaces of self-equivalences  play a key role in understanding the intrinsic symmetries of the manifold through surgery theory, mapping class groups, and homological stability for moduli spaces (see \cite{Berglund:2020}, \cite{Bustamante:2024}, \cite{Galatius:2020}).

Previous work by Hambleton--Kreck \cite{Hambleton:2004, Hambleton:2009}  developed a framework to investigate the group $\hept{M}\coloneqq\pi_0(\heqpt{M})$ of connected components of the space of orientation and base point preserving self-equivalences for a closed oriented $4$-manifold. Their approach involved the hands-on construction and analysis of a braid of interlocking exact sequences relating the group $\hept{M}$ to certain bordism groups associated with $M$. In this paper, we construct a  homotopy cartesian square of spaces that leads to a conceptual explanation for this braid, a broad generalization beyond the $\pi_0$-level, and some  new applications.

\subsection{Self-equivalences and bordism} 
Let $\ell \geq 0$ be a parameter, and let $M$ be a closed, connected, topological $n$-manifold, with $n \geq 4$. We obtain a map $\gamma_{\ell}\colon M \to P_\ell(BTOP)$  
 from the stable normal microbundle $\nu_M\colon M \to BTOP$ by composition with the $\ell$-th Postnikov truncation $w_\ell\colon BTOP \to P_\ell(BTOP)$. Let $\Ml$ denote the homotopy pullback given by these maps:
$$\Ml \coloneqq \left (BTOP \xrightarrow{w_\ell} P_\ell(BTOP) \xleftarrow{\gamma_\ell} M\right )$$
 and let $\xi_\ell \colon \Ml \to BTOP$ denote the fibration induced by the pullback diagram.

 In this work, we will be studying the space $\heqpt{\Ml}$ of all \emph{based homotopy self-equivalences} of $M$ preserving the Postnikov $\ell$-section  of $\nu_M$ up to prescribed homotopies. This space consists of homotopy commutative diagrams:
\begin{center}
	\begin{tikzcd}[ampersand replacement=\&]
		M \&\& M \\
		{} \& {P_{\ell}(BTOP)}
		\arrow["f", from=1-1, to=1-3]
		\arrow["{\gamma_{\ell}}"', from=1-1, to=2-2]
		\arrow["{\gamma_{ \ell}}", from=1-3, to=2-2]
	\end{tikzcd}
\end{center}
where $f\colon M \to M$ is a base point and local orientation preserving homotopy self-equivalence  (see  Definition \ref{def:orientation-preserving-maps}), equipped with a homotopy $H_f\colon M \times I \to P_\ell(BTOP)$ between $ \gamma_{\ell}\circ f$ and  $\gamma_{\ell}$ (see Section \ref{sec:three}).  For $\ell = \infty$,  we obtain the space $\heqpt{M, \nu_M}$ of based
 homotopy self-equivalences preserving the stable normal bundle of $M$.
 
 Similarly, we let $\Bl$ denote the homotopy pullback given by the maps 
 $$\Bl \coloneqq \left (BTOP \xrightarrow{w_\ell} P_\ell(BTOP) \xleftarrow{\nu_\ell} P_\ell(M)\right ),$$
 where $\nu_\ell = P_\ell(\nu_M)$, and let $\wxi_\ell \colon \Bl \to BTOP$ denote the fibration induced by the pullback diagram.
We consider the space $\heqpt{\Bl}$ of all base point preserving  homotopy self-equivalences.  For $\ell \geq 2$, the space $\heqpt{\Bl}$ admits a weakly equivalent model (Theorem \ref{thm:threeeight}) consisting of base point preserving homotopy self-equivalences
\begin{center}
 \begin{tikzcd}[ampersand replacement=\&]
		P_\ell(M) \&\& P_\ell(M) \\
		{} \& {P_{\ell}(BTOP)}
		\arrow["\phi", from=1-1, to=1-3]
		\arrow["{\nu_{\ell}}"', from=1-1, to=2-2]
		\arrow["{\nu_{ \ell}}", from=1-3, to=2-2]
	\end{tikzcd}
\end{center} 
 of the $\ell$-th Postnikov stage $P_\ell(M)$, preserving  $\nu_\ell = P_\ell(\nu_M)$  up to prescribed homotopies.

A point $(f, H_f) \in \heqpt{\Ml}$  can equivalently be described by a map $\hat{f} \colon M \to \Ml$  into the total space 
  of the fibration $\xi_\ell \colon \Ml \to BTOP$, that satisfies $\xi_\ell \circ \hat{f} = \nu_M$  (see diagram \eqref{diag:threezero} and  Definition \ref{def:thickaut}).
A similar description holds for points in $\heqpt{\Bl}$ in terms of maps $\hat{f}\colon M \to \Bl$.

The fibrations $\xi_\ell \colon \Ml \to BTOP$ and  $\wxi_\ell \colon \Bl \to BTOP$  define Lashof bordism theories $\Omega^\infty Q(\xi_\ell)$ and $\Omega^\infty Q(\wxi_\ell)$, where a bordism element in $\Omega_n(\xi_\ell)$ is represented by a pair $(N, \hat g)$, where $N$ is a closed $n$-manifold and $\hat g\colon N \to \Ml$ is a map such that $\xi_\ell \circ \hat g = \nu_N$, and similarly for elements in $\Omega_n(\wxi_\ell)$  (see Section \ref{sec:four} for more details, or  \cite[Chapter 2]{Stong:1968}, \cite{Lashof:1963}).

\subsection{Highly cartesian squares and braidings}
The last two paragraphs in the previous subsection suggest the existence of two maps of spaces $$\alpha\colon \heqpt{\Ml} \to \Omega^{\infty+n} Q(\xi_\ell) \text{ and } \beta\colon \heqpt{\Bl} \to \Omega^{\infty+n} Q(\wxi_\ell)$$
relating the spaces of self-equivalences and bordism. In Section \ref{sec:ss_models}, we use the $\Delta$-set theory of Rourke and Sanderson
\cite{Buoncristiano:1976} to model the involved bordism theories as $\Omega$-spectra and extract semi-simplicial models for their associated  $(\infty+n)$-fold loop spaces $\Omega^{\infty+n} Q(\xi_\ell)$ and $\Omega^{\infty+n} Q(\tilde{\xi}_\ell)$. This allows us to define $\alpha$ and $\beta$ in Section \ref{subsection:semi-simplicial-set-and-construction-of-square}.

There are, furthermore, two forgetful maps $$\wpp_\ell\colon \heqpt{\Ml} \to  \heqpt{\Bl} \text{ and } \wqq_\ell \colon \Omega^{\infty+n} Q(\xi_\ell) \to \Omega^{\infty+n} Q(\wxi_\ell)$$
leading to a commutative square of spaces.  The following is our main result (in which $[a]$ denotes the greatest integer function).

\begin{thma}
	Let $M$ be a closed topological manifold of dimension $n \geq 4$. For $k \geq 0$ and any $\ell \geq  \lmin\coloneqq [(n+k+1)/2]$, there exists a homotopy $(k+1)$-cartesian square:
	\begin{center}
		\begin{tikzcd}
			\heqpt{\Ml} \arrow[d, "\alpha"] \arrow[rr,"\wpp_\ell"] && \heqpt{\Bl} \arrow[d, "\beta"]\\
			\wOmega^{\infty+n} Q(\xi_\ell) \arrow[rr, "\wqq_\ell"] & & \Omega^{\infty+n} Q(\tilde{\xi}_\ell)
		\end{tikzcd}
	\end{center}
	where $\wOmega^{\infty+n} Q(\xi_\ell) \subseteq \Omega^{\infty+n} Q(\xi_\ell)$
	is the union of all connected components  consisting of elements of degree one.
\end{thma}

From the defining properties of a homotopy $(k+1)$-cartesian square,  we obtain a (sign-commutative) braid of interlocking  exact sequences, in a range of dimensions depending on the parameter $\ell\geq 2$. The terms appearing in the braid are:
\begin{itemize}
	\item the homotopy groups of $\heqpt{\Ml}$ and $\heqpt{\Bl}$;
	\item the bordism groups $\widehat{\Omega}_{n+k}(\Ml) = \pi_k (\widehat{\Omega}^{\infty+n} Q(\xi_\ell))$, $\Omega_{n+k}(\Bl) = \pi_k (\Omega^{\infty+n} Q(\tilde{\xi}_\ell))$ and $\widehat{\Omega}_{n+k+1}(\Bl,\Ml) $ in a Lashof bordism exact sequence	
	   (see Section \ref{sec:four} for the details and, in particular, Definition \ref{def:restricted-infty-plus-n-bordism-spaces} and Remark \ref{rem:sixfourteen} for an explanation of the  ``degree one'' hat-condition, which only applies if $k =0$);  
	\item the fibre bordism groups $\widetilde{ \cH}_k(\Ml)$ modelling the homotopy fibre of the map $\alpha$ in the square of Theorem A
	(see Definition \ref{def:the-bordism-group-htildemk} for an explicit description). 
\end{itemize}
For $n = 4$, $k = 0$ and $\ell = 2$, our next result recovers work of Hambleton--Kreck in the setting of oriented $4$-manifolds \cite{Hambleton:2004, Hambleton:2009}, which was the starting point of our investigation  (see Section \ref{sec:HKthmB}, and compare \cite{Pamuk:2026} for a related braid under some additional assumptions).

\begin{thmb}
	Let $M$ be a closed, smooth or topological manifold of dimension $n \geq 4$. For  $k \geq 0$ and $\ell \geq \lmin\coloneqq [(n+k+1)/2]$,	
  there exists an interlocking braid of exact sequences: 
	\vskip .2cm
	\begin{center}
	$$\begin{matrix}
		\resizebox{\displaywidth}{!}{
			\xymatrix@!C@C-60pt{
			\widehat\Omega_{n+k+1}(\Ml)    \ar[dr] \ar@/^2pc/[rr]   &&
					\widetilde{ \cH}_k(\Ml)  \ar[dr] \ar@/^2pc/[rr] &&
					\pi_k(\heqpt{\Bl})\ar[dr]^{\beta_k}  \\
					& \Omega_{n+k+1}(\Bl) \ar[dr] \ar[ur]  &&
		\pi_k(\heqpt{\Ml})  \ar[dr]^{\alpha_k}\ar[ur] &&
		 \Omega_{n+k}(\Bl) \\
		\pi_{k+1}(\heqpt{\Bl}) \ar[ur] \ar@/_2pc/[rr] &&
		 \widehat \Omega_{n+k+1}(\Bl,\Ml)
					\ar[ur]^{\gamma_k} \ar@/_2pc/[rr]
					&&
					\widehat \Omega_{n+k}(\Ml)\ar[ur] }
		}\end{matrix}
		$$
	\end{center}
	\vspace{0.5cm}
	For $k \geq 1$ all the maps in the diagram are group homomorphisms and the braid continues to the right, but for $k=0$ the maps $\alpha_0$, $\beta_0$ and $\gamma_0$ are only maps of pointed sets.
\end{thmb}

For topological manifolds this is an immediate application of Theorem A. The smooth case  follows from it by  
the results of Section \ref{subsec:corb}.
 In that case, we use the stable normal bundle $\nu_M \colon M \to BO$ and replace $BTOP$ by $BO$ in all constructions.

\begin{remark}\label{rem:smooth_model}
	We expect that there is an  analogue of Theorem A for smooth manifolds, but the techniques for modelling smooth bordism theories by using, for example, semi-simplicial sets, are still under development  (see \cite{Ayala:2025,Laures:2024} for recent work on the Quinn spectra \cite[\S 6]{Quinn:1995}).
\end{remark}

\subsection{Applications}

In Section \ref{sec:nine} we survey explicit calculations for $\hept{M}$ in the literature using the braid,  and indicate new applications of the braids to the topology of manifolds. Theorem \ref{thm:sixone} extends a result of Dror, Dwyer and Kan \cite[Theorem 1.1]{Dror:1979} to our setting, and 
Theorem \ref{thm:sixtwo} is a generalization of the work of Barcus and Barratt
\cite[\S 6]{Barcus:1958} and Baues \cite[Theorem 8.14]{Baues:1996a} to non-simply connected manifolds. In Theorem \ref{thm:sevensixteen}, we obtain a ``commensurability'' result  (under certain conditions) for $\hept{\Ml} = \pi_0(\heqpt{\Ml})$, generalizing the work of Triantafillou  \cite[Theorem 2]{Triantafillou:1995} and Bustamante, Krannich and Kupers \cite[Theorem D]{Bustamante:2024}.

\begin{acknowledgement} 
The authors would like to thank Alexander Kupers, Larry Taylor, and Oscar Randal-Williams for helpful comments and correspondence. We also want to express our particular thanks to Robin J.~Sroka, who contributed very significantly to this project since its start at McMaster in 2023, but declined authorship.
\end{acknowledgement}

\setcounter{tocdepth}{1}
\tableofcontents

\section{Topological preliminaries}\label{sec:two}

In this section we collect general homotopy theoretic tools and fix  some notation. 
In particular, we explain
\begin{itemize}
	\item the connection between Theorem A and Theorem B in Lemma \ref{lem:twoseven};
	\item the construction of the fibrations $\xi_\ell$ and $\wxi_\ell$ which are the input for our two Lashof bordism theories and the key Lemma \ref{lem:normalk-type} that this work is based on;
	\item our conventions for working with homotopy groups of Kan semi-simplicial sets.
\end{itemize}

\begin{convention}
	 Throughout this work, a space $X$ is called \emph{$n$-connected} if for all $-1 \leq k \leq n$ every map $S^k \to X$ can be extended to a map $D^{k+1} \to X$ (where $S^{-1} = \emptyset$). Similarly, $X$ is \emph{$n$-co-connected} if for all $k \geq n$ every map $S^k \to X$ can be extended to a map $D^{k+1} \to X$. 
	
	A map $f\colon X \to Y$ is called \emph{$n$-connected} if its homotopy fibre   $\on{hofib}_y(f)$ is $(n-1)$-connected for every choice of $y \in Y$. A map $f\colon X \to Y$ is called \emph{$n$-co-connected} if its homotopy fibre $\on{hofib}_y(f)$ is $n$-co-connected for every choice of $y \in Y$. 
\end{convention}

\begin{remark} In \cite[Proposition 2.6.9]{MunsonVolic:2015} there is a useful alternate formulation: a map $f\colon X\to Y$ is $n$-connected if and only if the pair $(M_f, X)$ is $n$-connected, where  $M_f$ is the mapping cylinder of $f$. If $X \neq \emptyset$, either condition is equivalent to assuming that for any $x \in X$, the induced map  $f_*\colon \pi_k(X, x) \to \pi_k(Y, f(x))$ is an isomorphism for $0 \leq k <n$ and a surjection for $k =n$. If $X \neq \emptyset$ and $n \geq 1$, we say that  $f\colon X\to Y$ \emph{is an $n$-equivalence} if any one of these three conditions holds.
\end{remark}

\subsection{Relation between highly cartesian squares and braidings}

In this subsection, we summarize the homotopy theoretic tools that we will use to prove Theorem A and Theorem B. We refer the reader to \cite{MunsonVolic:2015} for related background.

To prove Theorem A we will check the following characterization of $(k+1)$-cartesian squares.

\begin{lemma}[{\cite[Prop 3.3.18]{MunsonVolic:2015}}]\label{lem:p-cartesian-square}
	Let $k \geq 0$. A commutative square of spaces
	$$
	\xymatrix{W \ar[r]^p\ar[d]^\alpha & Y\ar[d]^\beta\\
		X \ar[r]^q & Z}
	$$
	is homotopy $(k+1)$-cartesian if and only if for all $x \in X$, the map
	$$
	\on{hofib}_x(\alpha) \to \on{hofib}_{q(x)}(\beta)
	$$
	is 	$(k+1)$-connected. 
\end{lemma}

 \begin{remark} More explicitly, the condition we need later for Proposition \ref{proposition:k-hml-iso-k-hbl} is that the induced map on homotopy fibres should induce a bijection on $\pi_0$, an isomorphism on $\pi_i$, for $i \leq k$, and an epimorphism on $\pi_{k+1}$. Since $k\geq 0$, these conditions hold even in the case of empty fibres; for  example, the unique map $\emptyset \to \emptyset$ is $n$-connected for every $n$ since there are no points in the target over which the homotopy fibre conditions must be checked. 
 \end{remark}

To construct the braid for topological manifolds displayed in Theorem B, we apply the following lemma once we have completed the proof of Theorem A.  

\begin{lemma}
	\label{lem:twoseven}
	A homotopy $(k+1)$-cartesian  square as in \textup{Lemma \ref{lem:p-cartesian-square}}, give rise to an interlocking braid of (sign) commutative diagrams up degree $i \leq k$ of the following form.
	
	$$\xymatrix@R-8pt@C-8pt{
					&\pi_{i+1}(X)    \ar[dr] \ar@/^2pc/[rr]   &&
					\pi_i(\on{hofib}_x(\alpha))  \ar[dr] \ar@/^2pc/[rr] &&
					\pi_i(Y)\ar[dr]  \\
					\pi_{i+1}(W)\ar[ur]\ar[dr]&& \pi_{i+1}(Z) \ar[dr] \ar[ur]  &&
					\pi_i(W)  \ar[dr]\ar[ur] &&  \pi_{i}(Z) \\
					&\pi_{i+1}(Y) \ar[ur] \ar@/_2pc/[rr] && \pi_i(\on{hofib}_z(q))
					\ar[ur]^{\gamma_i} \ar@/_2pc/[rr]&&
					\pi_i(X)\ar[ur] }
		$$
	\vskip 0.5 cm
	For $i \geq 1$, the braid continues to the right. 
\end{lemma}

\begin{proof} After picking compatible base points, $q(x) = z = \beta(y)$, a homotopy $p$-cartesian square of spaces provides four exact sequences of homotopy groups given by the horizontal and vertical maps in the diagram. In a range of dimensions, depending on $k$, a formal argument shows that these are contained in an interlocking braid of (sign) commutative diagrams  (see Wall \cite{Wall:1966}). The triangular sub-diagrams are all strictly commutative, since the maps involved are induced by the maps in
Lemma \ref{lem:p-cartesian-square}. 
    The map labelled $\gamma_k$ in the Theorem B braid is given by the composite  (for $i \leq k$)
$$\gamma_k \colon \pi_{k+1}(Z, X)\cong  \pi_k(\on{hofib}_z(q))  \xleftarrow{\cong}  \pi_k(\on{hofib}_y(p)) \cong \pi_{k+1}(Y, W) \to \pi_k(W).$$ 
  In the range $i\leq k$,  the ``boundary'' square sub-diagrams have the form: $$\xymatrix{\pi_{i+1}(Z )\ar[rr]\ar[d] \ar@{-->}[ddrrr]&&\pi_{i}(\on{hofib}_{z}(\beta))&\\
 \pi_{i}(\on{hofib}_z(q)) &&&\pi_{i}(\on{hofib}_x(\alpha))\ar[ul]_\cong\ar[d]^\bd\\
 &\pi_{i}(\on{hofib}_y(p))\ar[ul]^{\cong} \ar[rr]^\bd && \pi_i(W)
}$$
and they are  only sign-commutative. They contain the two different ways of computing the boundary map $\pi_{i+1}(Z) \to \pi_i(W)$, as in the induced Mayer-Vietoris sequence (see \cite[Lemma 15.1]{Eilenberg:1952} for the ``hexagonal lemma''). 
\end{proof}

\subsection{Construction of a Moore--Postnikov factorization}\label{subsec:moore_post}

In this subsection, we record properties of the fibrations $\xi_\ell\colon \Ml \to BCAT$ and $\wxi_\ell\colon \Bl \to BCAT$  obtained from $\nu_M\colon M \to BCAT$, for $CAT = O$ or $TOP$.   

For our applications, we will let  $Y = BO$ or $Y=BTOP$ in the following general statement. We refer the reader to \cite{Hatcher:2002, MunsonVolic:2015} for related background. 

 \begin{lemma}\label{lem:normalk-type}
	Let $\nu\colon M \to Y$ be a base point preserving map between path-connected spaces and $\ell \geq 0$. Then there exists a commutative diagram of base point preserving maps
	\eqncount
	\begin{equation}\label{eq:normalk-diag}
		\vcenter{\xymatrix{M \ar[dd]^{\nu} \ar[rrr]^{\hatid_M}\ar@{-->}[drr]^{\hat\nu_\ell}
				&&&\Ml\ar[dd]^{\xi_\ell}\ar[dl]^{\wpp_\ell}\\
				&&\Bl\ar[dr]^{\wxi_\ell}&\\
				Y \ar@{=}[rrr]&&&Y
		}}
	\end{equation}
	where
	\begin{itemize}
		\item $\xi_\ell$ and $\wxi_\ell$ are fibrations;
		\item $\Ml$ and $\Bl$ are path-connected; and
		\item $\Ml\xrightarrow{\ \wpp _\ell\ } \Bl \xrightarrow{\ \wxi_\ell \ } Y$ is an  $(\ell+1)$-th stage of a Moore-Postnikov factorization of the map $\xi_\ell \colon \Ml \to Y$. In other words, the map $\wpp_\ell$ is $(\ell+1)$-connected and the map $\wxi_\ell$ is $(\ell+1)$-co-connected (see \ \cite[p.414]{Hatcher:2002}).
	\end{itemize}
\end{lemma}

We start by recalling that the Postnikov tower $\{\post{X}{\ell}\}_{\ell \geq 0}$ of a space $X$ can be constructed functorially (see e.g.\ \cite[Section 1.1.3]{Galatius:2014}), that the space $\post{X}{\ell}$ is obtained from $X$ by attaching cells of dimension $\geq \ell+2$, and that the functors $\{\post{-}{\ell}\}_{\ell \geq 0}$ come with natural transformations $(-) \Longrightarrow \post{-}{\ell}$ and $\post{-}{\ell+1} \Longrightarrow \post{-}{\ell}$ that are given by inclusion maps.

We now consider the map $\nu \colon M \to Y$, let $\ell \geq 0$ and begin the construction of the commutative diagram in Lemma \ref{lem:normalk-type} .
Using the functor $\post{-}{\ell}$ and the natural transformation $(-) \Longrightarrow \post{-}{\ell}$, we obtain a strictly commutative diagram of base point preserving maps
\eqncount
\begin{equation}
	\label{diag:twoone}
	\begin{tikzcd}
		M \arrow[r, "c_\ell"] \arrow[d, "\nu"] \arrow[dr, "\gamma_{\ell}", dashed] & \post{M}{\ell} \arrow[d, "\nu_\ell"]\\
		Y \arrow[r, "w_\ell"] & \post{Y}{\ell}
	\end{tikzcd}
\end{equation}
where $\nu_\ell\coloneqq P_\ell(\nu)$ and the dashed diagonal map $\gamma_\ell$ is defined as the composition $\nu_\ell \circ c_\ell = w_{\ell} \circ \nu$.

\begin{definition}\label{def:hpull}
	We use the construction in \cite[Definition 3.2.4]{MunsonVolic:2015} to define the based spaces $\Ml$ and $\Bl$ as models for the  homotopy pullbacks in the following two homotopy commutative diagrams of pointed maps:
	\begin{center}
		\begin{tikzcd}[ampersand replacement=\&]
			\Ml \& M\\
			Y \& P_\ell(Y)
			\arrow["j", from=1-1, to=1-2]
			\arrow["\xi_\ell"', from=1-1, to=2-1]
			\arrow["w_\ell", from=2-1, to=2-2]
			\arrow["\gamma_\ell", from=1-2, to=2-2]
		\end{tikzcd}
		\hspace{1cm}
		\begin{tikzcd}[ampersand replacement=\&]
			\Bl \& P_\ell(M)\\
			Y \& P_\ell(Y)
			\arrow["j", from=1-1, to=1-2]
			\arrow["\wxi_\ell"', from=1-1, to=2-1]
			\arrow["w_\ell", from=2-1, to=2-2]
			\arrow["\nu_\ell", from=1-2, to=2-2]
		\end{tikzcd}
	\end{center}
	The spaces are given by
	\begin{eqnarray*}
		& \Ml \coloneqq \{(y, x, \theta) \in Y \times M \times P_\ell(Y)^I\vv \theta(0) = w_\ell(y), \theta(1) = \gamma_\ell(x)\}
				\\[.2cm]
		& \Bl \coloneqq \{(y, x, \delta) \in Y \times  P_\ell(M)  \times P_\ell(Y)^I\vv \delta(0) = w_\ell(y), \delta(1) = \nu_\ell(x)\}
	\end{eqnarray*}
	with base points $m_{\ell}\coloneqq(y_0, x_0, C_{z_0})$ and $b_{\ell}\coloneqq(y_0, c_{\ell}(x_0),  C_{z_0})$, where  $y_0 = \nu(x_0) \in Y$, $x_0 \in M$, $z_0 = \gamma_\ell(x_0) \in  P_\ell(Y)$ are the base points and $C_{z_0}$ denotes the constant path at $z_0 \in  P_\ell(Y)$. The \emph{structure maps}  ($\xi_\ell$ or $\wxi_\ell$, and the two $j$-labeled maps) in this model are defined by the obvious projections. The canonical homotopies (between $\gamma_\ell \circ j$ and $w_\ell \circ \xi_\ell$, and between $\nu_\ell \circ j$ and $w_\ell \circ \wxi_\ell$) are given by the formula $H((y,x,\gamma),t) = \gamma(t)$.
\end{definition}

These two homotopy pullbacks fit into the following homotopy commutative diagram, in which all three squares are homotopy pullback squares.
\eqncount
\begin{equation}\label{diag:twotwo}
	\vcenter{
		\xymatrix@C+5pt@R+5pt{
			M\la \ell\ra \ar[r]^{j}\ar[d]^{\wpp_\ell} \ar@/_2pc/[dd]_(0.3){\xi_\ell}
			& M\ar[d]^{c_\ell}\ar@/^2pc/[dd]^(0.3){\gamma_\ell}\\
			B\la \ell\ra \ar[r]^(0.5){j}\ar[d]^{\wxi_\ell} & \post{M}{\ell}\ar[d]^{\nu_\ell}\\
			Y\ar[r]^{w_\ell}& \post{Y}{\ell}
		}
	}
\end{equation}
Indeed, the map $c_\ell$ induces the map $\wpp_\ell \colon \Ml \to \Bl$ by setting $\wpp_\ell((y,x,\theta))= (y, c_\ell (x), \theta)$. The upper square is strictly commutative, and hence homotopy commutative by the constant homotopy.   Then \cite[Prop 3.3.20]{MunsonVolic:2015} implies that it is also a homotopy pullback square.

For the final step of the construction, we observe that a base point preserving map $W \to \Ml$ is equivalent to the data of  base point preserving maps $\psi_M \colon W \to M$ and $\psi_Y \colon W \to Y$ along with a  based homotopy from $w_\ell \circ \psi_Y$ to $\gamma_\ell \circ \psi_M$.
   A similar observation applies to maps $W \to \Bl$.  We then define $\hatid_{M} \colon M \to \Ml$ as the map defined by the homotopy commutative diagram
\begin{center}
	\begin{tikzcd}
		M \arrow[r, "\on{id}_M"] \arrow[d, "\nu"] & M \arrow[d, "\gamma_\ell"]\\
		Y \arrow[r, "w_\ell"] & \post{Y}{\ell}
	\end{tikzcd}
\end{center}
with prescribed homotopy constant at $\gamma_\ell$. We obtain the desired commutative diagram displayed in diagram \eqref{eq:normalk-diag} by defining the map $\hat\nu_\ell$ as the composition of $\hatid_M$ and $\wpp_\ell$.

\begin{proof}[Proof of Lemma \ref{lem:normalk-type} ]
	For the first part, note  that by \cite[Prop 3.2.5]{MunsonVolic:2015} the structure maps   ($\xi_\ell$ or $\wxi_\ell$, and the two $j$-labeled maps), are fibrations in the model for the homotopy pullback used in Definition \ref{def:hpull}. 
	
	For the second part, we observe that the two $j$-labeled fibrations in Definition \ref{def:hpull} have path-connected base space (since $M$ is path-connected) and are surjective (since $Y$ is path-connected). To see that their total spaces are path-connected, it hence suffices to check that their fibres are path-connected. This holds because their fibres are homotopy equivalent to the homotopy fibre of $w_\ell\colon Y \to \post{Y}{\ell}$, which is $\ell$-connected with $\ell \geq 0$ and hence path-connected.
	
	For the third part, we start by arguing that the map $\wpp _\ell\colon \Ml \to \Bl$ is an $(\ell+1)$-equivalence. This holds because the homotopy fibre of $\wpp_\ell$ is homotopy equivalent to that of $c_\ell$ and hence $\ell$-connected.  Next, we argue that the map $\wxi_\ell\colon \Bl \to Y$ is $(\ell+1)$-co-connected. This holds because the homotopy fibre of $\wxi_\ell$ is homotopy equivalent to that of $\nu_\ell$ and hence its homotopy groups vanish for $i \geq \ell+1$.
\end{proof}

\subsection{Homotopy groups of Kan semi-simplicial sets} \label{subsec:semi-simplicial-sets-and-kan-condition} 
In this subsection, we briefly explain our conventions for working with homotopy groups of Kan semi-simplicial sets. We refer readers to \cite{Ebert:2019}, \cite{May:1992} and standard references such as \cite{Rourke:1971}, \cite{Rourke:1982}, and \cite{Nicas:1982} for related background.

We recall the basic setup from the first paragraphs of \cite[\S 1]{Ebert:2019}. Let $\Delta$ denote the category with objects the linearly order sets $[n] = \{0 < 1< \dots < n\}$ of integers, and with morphisms $[n] \mapsto [m]$ the monotone functions, with composition given by composition of functions. Let $\sDelta \subset \Delta$ denote the subcategory containing all objects, but only the injective monotone maps. 

\begin{definition} A \emph{simplicial object} in a category $\cC$ is a functor $ \Deltaop \to \cC$. A \emph{semi-simplicial object} in  $\cC$ is a functor $\sDeltaop \to \cC$. We denote such a (semi-)simplicial object by $X_\bullet$ , and write 
$X_p = X_\bullet([p])$.
\end{definition}

The simplicial objects in a category $\cC$, together with natural transformations, form a category $s\cC$, and the semi-simplicial objects form a category $ss\cC$. There is a functor $F\colon s\cC \to ss\cC$ defined by restricting functors along $\sDeltaop \subset \Deltaop$.

\begin{remark} These notions were introduced by Eilenberg and Zilber \cite{Eilenberg:1950}, and developed by Kan \cite{Kan:1957} under the name ``complete semi-simplicial complexes'' or \emph{css complexes}. Kan \cite[Definition 1.1]{Kan:1957}  identified the ``extension condition'', now called the \emph{Kan condition}, as the key to defining homotopy groups of css complexes. In the work of Rourke and Sanderson \cite{Rourke:1971}, they used the term Kan $\Delta$-set to mean a css complex satisfying the extension condition.
\end{remark}

Throughout this work, we let $\Delta^n_{\on{inj}, \bullet}$ denote the semi-simplicial set  $\Delta^n_{\on{inj}, p} \coloneqq \sDelta([p],[n])$. Note that $\Delta^n_{\on{inj}, n}$ contains a unique $n$-simplex. For every $0 \leq k \leq n$, we consider the subobject
  $\Lambda^{n,k}_{\on{inj}, \bullet} \subset \Delta^n_{\on{inj}, \bullet}$, called the $k$-th horn of $\Delta^n_{\on{inj}, \bullet}$, generated by all faces of the unique $n$-simplex except its $k$-th face. A semi-simplicial set $*_\bullet$ is terminal if $*_p$ is a singleton for all $p \geq 0$.

\begin{definition}
A morphism of semi-simplicial sets $p_\bullet \colon E_\bullet\longrightarrow B_\bullet$ is a \emph{Kan fibration} if every commutative square
	$$
	\xymatrix{\Lambda^{n,k}_{\on{inj}, \bullet}  \ar[r]^{\sigma_0} \ar[d] & E_\bullet \ar[d]^{p_\bullet}\\
		\Delta^{n}_{\on{inj}, \bullet}  \ar[r]^{\bar\sigma} \ar@{-->}[ur]^{\sigma}& B_\bullet
	}
	$$
	of morphisms admits a lift $\sigma\colon \Delta^{n}_{\on{inj}, \bullet}  \to E_\bullet$ for  $n \geq 1$ and $0 \leq k \leq n$.
	\begin{enumerate}
	\item A \emph{semi-simplicial set $\bX$ is Kan} if its terminal map $\bX \to \ast_\bullet$ is a Kan fibration (equivalently, if for every   $n \geq 1$ and $0 \le k \leq n$, every map $\Lambda^{n,k}_{\on{inj}, \bullet} \to \bX$ extends to a map $\Delta^{n}_{\on{inj}, \bullet}  \to \bX$). 
	
	\item A \emph{base point in a Kan semi-simplicial set} $\bX$ is a point sub-complex
	  $x_{\bullet} \colon *_\bullet \to \bX$, meaning a sequence $x_\bullet =(x_n)_{n \geq 0}$ of simplices $x_n \in X_n$ such that $d_i x_n=x_{n-1}$ for all $0 \leq i \leq n$ and $n \geq 1$ (see \cite[Section 1.3]{Nicas:1982}).
	  \end{enumerate}
	  \end{definition}

	We will now define the homotopy groups of a pointed Kan semi-simplicial set $(\bX, x_\bullet)$. For this, we use a relative version (fixing $x_\bullet$) of a theorem of Rourke--Sanderson \cite[Theorem 5.7]{Rourke:1971} to equip $(\bX, x_\bullet)$ with degeneracy operators that promote it to a pointed simplicial set $(s\bX, sx_\bullet)\colon \Delta^{\on{op}} \to \on{Set}$, without adding or removing any simplices.

	\begin{remark}  Since we working relative to the point subcomplex in constructing these degeneracies, this condition must be added in the proof of \cite[Theorem 5.7]{Rourke:1971}. The same argument based on  \cite[Theorem 5.3]{Rourke:1971} then holds, as that result is already stated in a form relative to a given subcomplex $J \subset X$.
\end{remark}

\begin{definition}	
	For any $m \geq 1$, we define the homotopy groups of $\bX$ as the simplicial homotopy groups of $s\bX$,
$$\pi_m(\bX,x) \coloneqq \pi_m(s\bX,s x_\bullet) = \left\{a \in X_m  \vv  d_i a= x_{m-1} \text{ for all } 0 \leq i \leq m \right\}/ \sim
$$
where $a \sim a'$ if there exists $b \in X_{m+1}$ such that $d_m b = a, \ d_{m+1}b=a'$, and $d_i b=x_m$ for all $0 \leq i \leq m-1$. 

\begin{enumerate}
\item For $m=0$, define $\pi_0(\bX)$ as the set of vertices modulo path relation, so that $a \sim a' \Leftrightarrow$ if there exists $b \in X_1$ with $d_0 b=a$, and $d_1 b=a'$)
\item $\pi_m(X,x)$  is a group for $m \geq 1$  and a pointed set for $m=0$, when equipped with the operation $[a] \cdot [b] = [d_m c]$, where $c \in X_{m+1}$ is an extension of the $m$-horn $\Lambda^{m+1,m}_{\on{inj}, \bullet}$ defined by the face vector $(x_m, \dots, x_m, a, -, b)$,
\end{enumerate}
\end{definition}

We will follow the setup of \cite[Section 1.2]{Ebert:2019}, whenever we talk about the \emph{geometric realization} $\|\bX\|$ of a semi-simplicial set $\bX$. By \cite[Lemma 1.7]{Ebert:2019},  the quotient map $\|\bX\| \to |s\bX|$ is a homotopy equivalence and the simplicial homotopy groups are isomorphic under the induced map with those of $\|\bX\|$.  In particular, the definitions above do not depend on the degeneracy operators that we picked by the virtue of \cite[Theorem 5.7 and Remark 6.5]{Rourke:1971}.

\begin{definition}
A map $h \colon A_\bullet \to B_\bullet$ of Kan semi-simplicial sets is  $n$-connected if its homotopy fibre $\on{hofib}_{b_\bullet}(h)$ is $(n-1)$-connected at all base points $b_\bullet \in B_\bullet$. If $n > 0$, this is equivalent to asking that $h_*\colon \pi_k(A_\bullet) \to \pi_k(B_\bullet)$ 
 is a bijection for $k = 0$ and, for all choices of base points, an isomorphism for $1\leq k < n$ and surjection for $k = n$.
\end{definition}

\begin{convention}
	From now on we will denote semi-simplicial sets (or spaces) by $X$ and not $\bX$, to avoid confusion with our base point notation for spaces.
\end{convention}

\section{Fundamental classes and non-orientable manifolds} \label{sec:one}
In this short section, we review Poincar\'e duality and the definition of (twisted) fundamental classes for smooth manifolds. A similar discussion applies to topological manifolds if one replaces the stable normal bundle by the stable normal microbundle over (see Milnor \cite{Milnor:1964}, Kister \cite{Kister:1964}, Rourke and Sanderson \cite{Rourke:1970} for foundational material about topological microbundles). 
For the rest of the paper, we will work in the following general setup.

\begin{convention} We will assume that all spaces are equipped with a base point,  and that all maps are based point preserving. All manifolds considered are compact, 
smooth (or locally flat topological) embedded submanifolds of $\bR^{\infty}$.
We fix an orientation of the canonical bundle over $BO$ (or $BTOP$) at the base point, and equip all manifolds with a local orientation at the base point by pullback over the classifying map of the stable normal bundle (or microbundle). 
\end{convention}

Let  $\ell \geq 1$ and consider the Moore-Postnikov factorization $\wxi_\ell \colon \Bl \to BO$ associated to $\nu_M\colon  M \to BO$  (see Section \ref{subsec:moore_post}). Let $\hat\nu_\ell \colon M \to  \Bl$ denote the classifying map of the $\wxi_\ell$-structure on $M$ with $\wxi_\ell \circ \hat\nu_\ell = \nu_M$, and let  $\lambda \colon \Bl \to {K( \pi_1(M, x_0),1)}$  
be a map inducing an isomorphism of fundamental groups so that  $c\coloneqq\lambda\circ \hat\nu_\ell$ classifies the universal covering of $M$. 

Note that $\hat\nu_\ell \colon M \to  \Bl$ induces an isomorphism on fundamental groups, because $\hat\nu_\ell = \wpp_\ell \circ \hatid_M$ (see diagram \eqref{eq:normalk-diag}), $\wpp_\ell$ is an $(\ell+1)$-equivalence with $\ell \geq 1$ (by Lemma \ref{lem:normalk-type} ), and $\hatid_M$ is an isomorphism on fundamental groups (because $\on{id}_M = j \circ \hatid_M$ and $j$ is $(\ell+1)$-connected as discussed in the proof of Lemma \ref{lem:normalk-type}).

The orientation character $w\colon \pi_1(M, x_0) \to \{\pm 1\}$ of $M$, defined by the first Stiefel-Whitney class $w_1(M) \in H^1(M; \cy 2)$, now gives rise to a local 
coefficient system $\bZ^w$ on $M$ by pulling back along $c\colon M \to K(\pi_1(M, x_0), 1)$. A choice of generator 
  $\mu_M \in H_n(M;\bZ^w) \cong \bZ$ is called a (twisted) fundamental class for $M$ associated to the local coefficient system   $\bZ^w$ 
  (see \cite[\S 6]{Taylor:2008}). 
   A choice of fundamental class can be regarded as an ``orientation'' on $M$, since cap product with this class induces Poincar\'e duality for $M$ as a non-orientable manifold\footnote{Note Larry Taylor's remark ``non-orientable manifolds can not be oriented'' \cite[\S 5]{Taylor:2008}.}{}. Note that a \emph{preferred} choice of generator   $\mu_M\in H_n(M; \bZ^w)$ is determined by pulling back a (fixed) local orientation at a base point for the canonical bundle over $BO$.

\begin{definition}\label{def:orientation-preserving-maps}
	A pair $(M, \mu_M)$ consisting of a closed $n$-manifold $M$ and a choice of fundamental class $\mu_M \in H_n(M;\bZ^w)$ will be called a \emph{(locally) oriented $n$-manifold}. A map $f\colon (N, \mu_N) \to (M, \mu_M)$ is \emph{(locally) orientation-preserving} if  $f_*(\mu_N) = \mu_M$ and $f^*(w_1(M)) = w_1(N)$.
\end{definition}

The preferred local coefficient system $\bZ^w$
pulled back by $\lambda$ from $K( \pi_1(M, x_0),1)$ gives a  preferred local coefficient system on $\Bl$. 
Under the Thom isomorphism induced by the collapse map  $\varphi\colon S^{n+k}\to T(\nu_M)$,  for large $k$, the cap product $\varphi_*[S^{n+k}] \cap U(\nu_M)$ with a Thom class gives a generator of $H_n(M;\bZ^w)$.  Hence a choice of fundamental class $[M] \in H_n(M;\bZ^w)$ determines  a preferred  generator  $U(\nu_M)\in H^k(T(\nu_M); \bZ^w) \cong \bZ$, and conversely. 

\begin{remark}\label{rem:fourfive}  Let $X$ be an $n$-dimensional Poincar\'e complex and set $\Lambda\coloneqq \bZ\pi_1(X, x_0)$. We will follow the conventions of \cite[Chap.~2]{Wall:1999}, and use the notation $H_*^t(X; \cB) = H(C_*(X)\otimes_{\Lambda} \cB)$ for homology with coefficients $\cB$ in a right $\Lambda$-module. To define the tensor product,  $\cB$ is given a left $\Lambda$-module structure by the involution induced by $g\mapsto w(g) g^{-1}$, for $g \in \pi_1(X)$. 
Then $[X] \in H_n^t(X;\bZ) \equiv H_n(X;\bZ^w)$, and duality for a Poincar\'e  complex $X$ is an isomorphism
$\cap [X] : H^r(X; \cB) \to H_{n-r}^t(X; \cB)$ induced by cap product with the fundamental class. 
\end{remark}

\section{Self-equivalences}
\label{sec:three}

Let $M$ be a closed, smooth, embedded $M \subset \bR^\infty$, connected, based $n$-manifold with a preferred orientation $(M,\mu_M)$ as defined in Definition \ref{def:orientation-preserving-maps}. In this section, we introduce two variants $\heqpt{\Ml}$ and $\heqpt{\Bl}$ of the classical spaces of base point-preserving homotopy self-equivalences  $\heqpt{M}$ and $\heqpt{P_\ell(M)}$, respectively. These variants encode information about the normal Gauss map  $\nu_M \colon M \to BO$  up to the $\ell$-th Postnikov stage. Their construction is based on Lemma \ref{lem:normalk-type} and the pullback diagram \eqref{diag:twotwo} with $M$ the manifold  and $Y=BO$.  For convenience, we present the version for smooth manifolds over $BO$ and work in the general setup explained in Section \ref{sec:one}.  In particular, all spaces and maps are pointed. We leave the changes needed for topological manifolds over $BTOP$ to the reader.

\subsection{Self-equivalences of $M$ over $P_{\ell}(BO)$}

First, observe that a map $\hat{f} \colon M \to \Ml$ can be described by a triple $(\nu_f, f, H_f)$ where
\begin{enumerate}[(i)]
	\item $\nu_f \colon M \to BO$ is a  map such that $\nu_f = \xi_{\ell} \circ \hat{f}$,
	\item $f \colon (M,m_0) \to (M,m_0)$ satisfies $f= j \circ \hat{f}$, and
	\item $H_f \colon M \times I \to P_\ell(BO)$ is a homotopy with boundary conditions $H_f|_{M \times \{0\}}= w_{\ell} \circ \nu_f$ and $H_f|_{M \times \{1\}}= \gamma_{\ell} \circ f$.
\end{enumerate}
Equivalently, we have the following diagram:
\eqncount
\begin{equation}\label{diag:threezero}
	\begin{tikzcd}[ampersand replacement=\&]
		M \\
		\& {M\langle \ell \rangle} \&\& M \\
		\\
		\& BO \&\& {P_{\ell}(BO)}
		\arrow["{\hat{f}}"{description}, dotted, from=1-1, to=2-2]
		\arrow["f", curve={height=-12pt}, from=1-1, to=2-4]
		\arrow["{\nu_f}"', curve={height=18pt}, from=1-1, to=4-2]
		\arrow["\circlearrowright"{description}, draw=none, from=1-1, to=4-2]
		\arrow["j"', from=2-2, to=2-4]
		\arrow["\circlearrowright"{pos=0.1}, shift left=4, draw=none, from=2-2, to=2-4]
		\arrow["{\xi_\ell}"', from=2-2, to=4-2]
		\arrow["{H_f}", shift left=5, curve={height=18pt}, shorten <=30pt, shorten >=30pt, Rightarrow, 2tail reversed, no head, from=2-4, to=4-2]
		\arrow["{\gamma_{\ell}}", from=2-4, to=4-4]
		\arrow["{w_{\ell}}"', from=4-2, to=4-4]
	\end{tikzcd}
\end{equation}

The space $\Map_{\bullet}(M,\Ml)$ of pointed maps is equipped with the compact-open topology.

\begin{definition}\label{def:thickaut}
	For  $\ell \geq 1$, let $\heqpt{\Ml}$ denote the subspace of $\Map_{\bullet}(M,\Ml)$ containing  maps $\hat{f}=(f, H_f) \colon M \to \Ml$ where $\nu_M= \xi_\ell \circ \hat{f}$ is the classifying map of the stable normal bundle of $M$ and $f=j\circ \hat f \in \heqpt{M}$ is a  \emph{based homotopy self-equivalence}, meaning that it is  base point and local orientation preserving (see  Definition \ref{def:orientation-preserving-maps}). 
	\end{definition}

The base point of $\heqpt{\Ml}$ is $\widehat{\text{id}}_M = (\text{id}_M, C_{\gamma_\ell})$ where $C_{\gamma_{\ell}}=\gamma_{\ell} \circ p_M \colon M \times I \to P_{\ell}(BO)$ is the constant homotopy at $\gamma_\ell$. The next remark explains why this space is precisely the space of based homotopy self-equivalences of $M$ that preserve $\gamma_\ell$ up to prescribed homotopy, which is how we will think about it.

\begin{remark}\label{rem:threethree}
	Points of $\heqpt{\Ml}$ are in bijection with pairs $(f,H_f)$ where $f \in \heqpt{M}$ is a based homotopy self-equivalence and $H_f \colon M \times I \to P_\ell(BO)$ is a base point and local orientation preserving homotopy between $\gamma_\ell$ and $\gamma_\ell \circ f$ (note that we have $\gamma_{\ell}= w_{\ell} \circ \nu_M$, by diagram \ref{diag:twoone}). This homotopy ensures that the following diagram commutes up to homotopy:
	\[\begin{tikzcd}[ampersand replacement=\&]
		M \&\& M \\
		\\
		{} \& {P_{\ell}(BO)}
		\arrow["f", from=1-1, to=1-3]
		\arrow["\simeq"', from=1-1, to=1-3]
		\arrow["{\gamma_{\ell}}"', from=1-1, to=3-2]
		\arrow["{\gamma_{\ell}}", from=1-3, to=3-2]
		\arrow["{H_f}"{pos=0.6}, shorten <=62pt, shorten >=40pt, Rightarrow, from=3-1, to=1-3]
	\end{tikzcd}\]
\end{remark}

Consider two elements $\hat{f},\hat{g} \in \heqpt{\Ml}$ where $\hat{f}=(f,H_f)$ and $\hat{g}=(g,H_g)$. We define their product as
$$
(g,H_g) \bullet (f,H_f)\coloneqq(g \circ f, H_{g \circ f}) \in \heqpt{\Ml}
$$
where $g \circ f$ is the composition of homotopy self-equivalences, and $H_{g \circ f} \colon M \times I \to P_\ell(BO)$ is a homotopy connecting $\gamma_{\ell}$ to $\gamma_{\ell} \circ g \circ f$. This homotopy is explicitly defined by concatenation as
\begin{displaymath}
	H_{g \circ f}(m,t) = \left\{
	\begin{array}{lr}
		H_f(m,2t) & t \in [0,1/2]\ \\
		H_g (f(m), 2t-1) &  t \in [1/2,1].
	\end{array}
	\right.
\end{displaymath}
This operation equips $\heqpt{\Ml}$ with the structure of an $H$-space which is, in fact, group-like (meaning that each element has a homotopy inverse). Therefore, the set $\hept{\Ml} = \pi_0(\heqpt{\Ml})$ of path components inherits the structure of a group.

\begin{theorem}\label{thm-being-group}
	$(\heqpt{\Ml}, \bullet)$ is a group-like $H$-space with homotopy identity element $\hatid_M$.
\end{theorem}

\begin{proof}
Using Remark \ref{rem:threethree}, we identify a point of $\heqpt{\Ml}$ with a pair $\hat{f}=(f,H_f)$. The formula defining $\bullet$ is well-defined since the two homotopies being concatenated agree at their common endpoint. It is continuous with respect to the compact-open topology, since composition of maps and the usual concatenation of paths are continuous.

Homotopy associativity follows from the standard reparametrization homotopy between the two ways of concatenating three paths. Similarly, the left and right unit homotopies are obtained by continuously removing the constant path occurring in $\widehat{\id}_M \bullet \hat{f}$ and $\hat{f} \bullet \widehat{\id}_M$. All these homotopies are through based maps.

It remains to check that the resulting $H$-space is group-like. Let $\hat f=(f,H_f)$ and choose a based homotopy inverse $h$ of $f$, together with a based homotopy $A$ from $f \circ h$ to $\id_M$. The map $h$ is again locally orientation preserving. Writing bars for reversed homotopies, we define a based homotopy from $\gamma_{\ell}$ to $\gamma_{\ell} \circ h$ by concatenation
$$
H_h = (\gamma_{\ell} \circ \overline{A})* \overline{\,H_f \circ (h\times \id_I)\,}.
$$
Thus $\hat h=(h,H_h)$ is an element of $\heqpt{\Ml}$. The homotopy associated to $\hat{f} \bullet \hat{h}$ is
$$
(\gamma_{\ell} \circ \overline{A})* \overline{\,H_f \circ (h \times \id_I)\,} * (H_f \circ (h \times \id_I))
$$
which is homotopic relative to its endpoints to $\gamma_{\ell} \circ \overline{A}$, since the last two homotopies are reverses of each other. Moreover, writing $A_s= A(-,s)$, we define $K_s(m,t) = \gamma_{\ell} \bigl( A(m, 1-t(1-s)) \bigr)$. For every $s \in I$, the map $K_s$ is a homotopy from $\gamma_{\ell}$ to $\gamma_{\ell} \circ A_s$. Hence, the pairs $(A_s,K_s)$ define a path from $(f \circ h, \gamma_{\ell} \circ \overline{A})$ to $(\id_M,C_{\gamma_{\ell}})$. This gives $[\hat{f} \bullet \hat{h}] = [\widehat{\id}_M]$ in $\pi_0(\heqpt{\Ml})$ meaning that every element of $\pi_0(\heqpt{\Ml})$ has a right inverse. Since $\pi_0(\heqpt{\Ml})$ is an associative unital monoid, it is a group. Hence, $\heqpt{\Ml}$ is a group-like $H$-space.
\end{proof}

\subsection{Self-equivalences of $P_{\ell}(M)$ over $P_{\ell}(BO)$}

We now define an analogous space of self-equivalences $\heqpt{\Bl}$ for $\Bl$, and discuss how it encodes self-equivalences of $P_{\ell}(M)$ preserving $\nu_\ell$  over $P_{\ell}(BO)$. For the notation of the maps see diagram \eqref{diag:twotwo}.

\begin{definition}\label{def:thickautB}
	For  $\ell \geq 1$, let $\heqpt{\Bl}$ denote the subspace of $\Map_{\bullet}(M,\Bl)$ containing base point preserving maps $\hat{f}=(f, \widetilde{H}_f) \colon M \to \Bl$ where $\nu_M= \wxi_\ell \circ \hat{f}$ is the classifying map of the stable normal bundle of $M$, $f\coloneqq j\circ \hat{f}$ is a base point preserving  $(\ell + 1)$-equivalence, and $\widetilde{H}_f$ is a base point preserving homotopy from $\gamma_{\ell}= \nu_{\ell} \circ c_{\ell}$ to $\nu_{\ell} \circ f$.
\end{definition}

The base point of $\heqpt{\Bl}$ is given by 
 $\hat\nu_\ell = (c_\ell, \widetilde{C}_{\gamma_\ell})$  where $\widetilde{C}_{\gamma_\ell} \colon M \times I \to P_{\ell}(BO)$ is the constant homotopy at $\gamma_\ell$. Notice that the bundle $\wxi_\ell \colon \Bl \to BO$ inherits a local orientation by pullback.

The full space $\heqpt{\Bl}$ as defined above does not seem to admit a natural $H$-space structure, in contrast to $\heqpt{\Ml}$. However, we can define a closely related weakly equivalent space which does.
  In particular, $\Aut_{\bullet}(\Bl)=\pi_0(\heqpt{\Bl})$ does admit a natural group structure.

\begin{definition}\label{def:thickautBc}
	Let $\heqptkc{B}{\ell}$  denote the  space of homotopy commutative diagrams:
	\[\begin{tikzcd}[ampersand replacement=\&]
		\post{M}{\ell} \&\& \post{M}{\ell} \\
		\\
		{} \& {P_{\ell}(BO)}
		\arrow["\phi", from=1-1, to=1-3]
		\arrow["\simeq"', from=1-1, to=1-3]
		\arrow["{\nu_{\ell}}"', from=1-1, to=3-2]
		\arrow["{\nu_{\ell}}", from=1-3, to=3-2]
		\arrow["{H_\phi}"{pos=0.6}, shorten <=62pt, shorten >=40pt, Rightarrow, from=3-1, to=1-3]
	\end{tikzcd}\]
	given by pairs
	$(\phi, H_\phi)$ such that $\phi\colon P_\ell(M) \to P_\ell(M)$ is a base point preserving  homotopy equivalence, and
	$H_\phi \colon \post{M}{\ell} \times I \to \post{BO}{\ell}$ is a base point preserving homotopy $\nu_{\ell} \simeq \nu_{\ell} \circ \phi$.
\end{definition}

\begin{remark}\label{rem:foursix}
Notice that the existence of the prescribed homotopy $H_{\phi}$ implies that $\phi$ preserves the orientation character ($\phi^* w= w$) and so it induces an endomorphism of $H_n(P_{\ell}(M);\bZ^w)$.
\end{remark}
We can, once again, use composition to define an $H$-space structure $(\heqptkc{B}{\ell}, \bullet)$.

\begin{theorem}
	$(\heqptkc{B}{\ell}, \bullet)$ is a group-like $H$-space.
\end{theorem}
\begin{proof}
    This follows by exactly the same proof as Theorem \ref{thm-being-group}, replacing $M,\gamma_{\ell}$ by $P_{\ell}(M), \nu_{\ell}$ respectively.
\end{proof}
There is a natural map 
$$
c\colon \heqptkc{B}{\ell} \to \heqpt{\Bl}: (\phi, H_\phi) \mapsto (\phi\circ c_\ell, H_\phi\circ (c_\ell \times \id)),
$$
which is well-defined because $\phi\circ c_\ell\colon M \to P_\ell(M)$ is an $(\ell +1)$-equivalence.

\begin{theorem}\label{thm:threeeight}
	If $\ell \geq 2$, the map $c$ is a weak equivalence between $\heqptkc{B}{\ell}$ and $\heqpt{\Bl}$.
\end{theorem}

The following lemmas provide the necessary extension results for the proof.

\begin{definition}
	We say that a map $\phi\colon A \times Y \to Z$ is an \emph{$A$-family of $k$-equivalences} (respectively an \emph{$A$-family of homotopy self-equivalences}), if $\phi(a,-)\colon Y \to Z$ is a (based) $k$-equivalence (respectively, a (based) homotopy self-equivalence in case $Y=Z$)  for each $a \in A$.
\end{definition}

\begin{lemma}\label{lem:threesix}
	Let $(X,A)$ be a CW pair, and let $\phi_A\colon A \times P_\ell(M) \to P_\ell(M)$ be an $A$-family of homotopy self-equivalences,  for $\ell \geq 2$. Let  $f \colon X \times M \to P_\ell(M)$ be a map such that the restriction $f|_{A \times M} = \phi_A \circ (\id_A \times c_\ell)$. Then  there is a unique extension of $\phi_A$
	up to base point preserving homotopy, relative to $(X \times M)\cup (A \times P_{\ell}(M))$
$$
    \phi_f \colon X \times P_\ell(M) \to P_\ell(M)
    $$ 
    such that $\phi_f\circ (\id_X \times c_\ell )= f$.  If $f$ is  an $X$-family of $(\ell +1)$-equivalences, then $\phi_f$ is an $X$-family of homotopy self-equivalences.
	In addition,
	if $\nu_\ell \circ f \simeq \nu_\ell \circ c_\ell \circ p_M$ where $p_M\colon X \times M \to M$ is the projection, then $\nu_\ell \circ \phi_f \simeq \nu_\ell \circ p_{\post{M}{\ell}}$. 	
\end{lemma}

\begin{proof}
	The inclusion $c_{\ell} \colon M \to P_{\ell}(M)$ is a cofibration and $(\ell+1)$-connected. Consider the relative CW complex $(P_{\ell}(M),M)$, which is obtained by attaching cells of dimensions  $\geq \ell +2$ and hence has no cells in dimensions $\leq \ell +1$. The obstructions to extend $f \colon X \times M \to P_{\ell}(M)$ over $X \times P_{\ell}(M)$ relative to $(X\times M \cup (A \times P_\ell(M))$  lie in
	$$H^{k+1}(X \times P_{\ell}(M),X \times M \cup A \times P_{\ell}(M); \pi_k(P_{\ell}(M)))$$ for $k \geq \ell +1$.  Since $\pi_k (P_{\ell}(M))=0$ for $k \geq \ell+1$, all obstructions vanish and the map $f$ extends to $\phi_f \colon X \times P_{\ell}(M) \to P_{\ell}(M)$ (see \cite{Olum:1950} for obstruction theory in a  setting where the target is not a simple space).
	
	Let $\phi_f$ and $\phi_f'$ be two extensions of $f$. Then obstructions to homotope $\phi_f$ to $\phi_f'$ (relative to $A \times M$) lie in
	$H^{k}(X \times P_{\ell}(M),X \times M \cup A \times P_{\ell}(M); \pi_{k}(P_{\ell}(M)))$ for $k \geq \ell+1$, and they all vanish. Therefore, $\phi_f$ is unique up to base point preserving homotopy.
	
	Assume that $f(x,-)$ is an $(\ell+1)$-equivalence for all $x \in X$. Since $c_{\ell}$ is also $(\ell+1)$-equivalence, it follows that for all $x \in X$, the map $\phi_f(x,-)$ induces an isomorphism on $\pi_i$ for $i \leq \ell$ and a surjection on $\pi_{\ell+1}$. Since $\pi_k(P_{\ell}(M))=0$  
for	 $k \geq \ell+1$, it follows for any $x \in X$ that the extension $\phi_f(x,-)$ induces isomorphism on all homotopy groups of $P_{\ell}(M)$. By Whitehead's theorem, $\phi_f$ is an $X$-family of homotopy equivalences.
	
	Now assume that there is a homotopy $T\colon  X \times M \times I \to \post{BO}{\ell}$ for $\nu_\ell \circ f \simeq \nu_\ell \circ c_\ell \circ p_M$.   Using the property $f = \phi_f \circ (\on{id}_X \times c_{\ell})$ and that $c_\ell \circ p_M = p_{\post{M}{\ell}} \circ (\on{id}_X \times c_{\ell})$, this means that  $T$ gives a homotopy
	$$\nu_{\ell}\circ \phi_f \circ (\on{id}_X \times c_{\ell}) \simeq \nu_{\ell} \circ p_{\post{M}{\ell}} \circ (\on{id}_X \times c_{\ell}).$$
	Since $(\on{id}_X \times c_{\ell}) \colon X \times M \to X \times P_{\ell}(M)$ is a cofibration, by applying the homotopy extension property to this homotopy we extend the homotopy $T$ from  $X \times M \times I$ to $$\widehat T\colon X \times P_{\ell}(M)\times I \to\post{BO}{\ell},$$
	 starting with the extension 
	$\widehat T_0 =\nu_{\ell} \circ p_{\post{M}{\ell}}$ at $X \times \post{M}{\ell} \times 0$.

	  The resulting map $\widehat T_1\colon X \times  \post{M}{\ell} \to  \post{BO}{\ell}$ restricted to the subspace $X \times M$ is homotopic via $T$ to $\nu_{\ell}\circ \phi_f \circ (\on{id}_X \times c_{\ell})$. Since $(X\times P_{\ell}(M),X \times M)$ has no relative cells in dimensions $\leq \ell+1$ and $\pi_i(P_{\ell}(BO))=0$ for $i > \ell$, we conclude that $\widehat T_1 \simeq \nu_\ell \circ \phi_f$.
	 Thus, we obtain a homotopy
	$\nu_\ell \circ \phi_f \simeq \nu_\ell \circ p_{\post{M}{\ell}}$. \end{proof}

\begin{lemma}\label{lem:threeeleven}
	Suppose that $\wF_1, \wF_2\colon X \to \heqptkc{B}{\ell} $
	satisfy $c \circ \wF_1 \simeq c \circ \wF_2$. Then $\wF_1 \simeq \wF_2$.
\end{lemma}
\begin{proof}

	If we view the maps $\wF_i$ as pairs $\wF_i = (F_i, T_i)$ where $F_i\colon X \times \post{M}{\ell} \to \post{M}{\ell}$ and $T_i\colon X \times I \times \post{M}{\ell} \to \post{BO}{\ell}$, then $c \circ \wF_i = (F_i \circ (\on{id}_X \times c_\ell), T_i \circ (\on{id}_{X \times I} \times c_\ell))$ (using analogous notation to that in Definition \ref{def:thickautB}).	Now, let
	 $\widehat{Q} \colon X \times I \to \heqptk{B}{\ell}$ be a homotopy $c \circ \wF_1 \simeq c \circ \wF_2$. We view $\widehat{Q}$ as a pair $\widehat Q = (Q, T)$ where $Q \colon X \times M \times I \to P_\ell(M)$
	is a homotopy $F_1 \circ (\on{id}_X \times c_\ell) \simeq F_2 \circ (\on{id}_X \times c_\ell)$ and $T\colon X \times M \times I^2 \to P_\ell(BO)$ is a homotopy $T_1 \circ (\on{id}_{X \times I} \times c_\ell) \simeq T_2 \circ (\on{id}_{X \times I} \times c_\ell)$ (in the first direction) and a homotopy $\nu_\ell \circ Q \simeq w_\ell \circ \nu_M \circ p_M$ with $p_M \colon X \times M \times I \to M$ denoting the projection (in the second direction). Finally, we consider the following diagram:
	$$\xymatrix{ X \times M \times I \ar[dd]_{\nu_M \circ p_M} \ar[dr]\ar@/^2pc/[drr]^Q&&\\
		\save [] + <1.5cm, -0.55 cm>*{T}\restore &X \times P_\ell(M) \times I \ar@{-->}[r]^(0.6){Q'} \ar[d]\save [] + <1.0cm, -0.55 cm>*{T'}\restore& P_\ell(M)\ar[dl]^{\nu_\ell}\\
		BO \ar[r]_{w_\ell} & P_\ell(BO)&
	}$$
	We apply the first statement of Lemma \ref{lem:threesix} to show  that (i)
	that there is an extension $Q'\colon X \times P_\ell(M) \times I \to P_\ell(M)$ of the given homotopy $Q$, and (ii) that there is an extension $T' \colon X \times P_\ell(M) \times I^2 \to P_\ell(BO)$ of the given homotopy $T$. The pair $\widehat Q' = (Q', T')$ then defines the desired homotopy $\widehat Q' \colon X \times I \to \heqptkc{B}{\ell} $ for $\wF_1 \simeq \wF_2$ by the last two statements of Lemma \ref{lem:threesix} (note that $w_\ell \circ \nu_M \circ p_M = \nu_\ell \circ c_\ell \circ p_M$ by diagram (\ref{diag:twoone})). 
\end{proof}

\begin{proof}[The proof of Theorem \ref{thm:threeeight}] We apply Lemma \ref{lem:threesix} to show surjectivity on homotopy groups and Lemma \ref{lem:threeeleven} to show injectivity.
\end{proof}

\section{Bordism}
\label{sec:four}

We recall the definitions of the bordism groups that appear in the braid.  Let $M$ be a closed connected smooth $n$-dimensional manifold, and its normal bundle determines the classifying map $\nu_M \colon M \to BO$. The corresponding version for topological manifolds with $\nu_M \colon M \to BTOP$ is defined similarly. Recall that, by convention, all manifolds are
 smooth (or locally flat topological) embedded submanifolds of $\bR^{\infty}$, and  locally oriented via the classifying map of the stable normal bundle.

\subsection{Lashof bordism and stable Thom classes}\label{subsec:lashofone}
In the following sections, we will need a uniform way to assign preferred fundamental classes as the dimensions of the manifolds vary. Recall that Stong \cite[Chapter~2]{Stong:1968} and Lashof \cite{Lashof:1963} define a
$(B,f)$-structure on a smooth $n$-manifold $M$ using a sequence of fibrations $f_r \colon B_r \to BO_r $ and maps $g_r \colon B_r \to B_{r+1}$ satisfying $f_{r+1}\circ g_r=j_r\circ f_r$ where $j_r\colon BO_r\to BO_{r+1}$ is the usual inclusion.

By a $(B_r,f_r)$-structure on the normal bundle of $M \subset \bR^{n+r}$, we mean a lift of its unstable classifying map $\nu_M \colon M \to BO_r$ to $B_r$. The inclusion $\bR^{n+r} \subset \bR^{n+r+1}$, together with the stabilization map $g_r$, induces a $(B_{r+1},f_{r+1})$-structure. A \emph{$(B,f)$-structure on $M$} is an equivalence class of compatible sequences of such structures, where two sequences are equivalent if they agree for all sufficiently large $r$. A $(B,f)$-manifold is a manifold equipped with a $(B,f)$-structure.

One can form the \emph{cobordism category of $(B,f)$-manifolds}, whose objects are closed smooth manifolds with $(B,f)$-structure and whose morphisms are compact smooth bordisms with compatible $(B,f)$-structure. Bordism classes form graded abelian groups $\Omega_*(B,f)$, with addition given by disjoint union.

Let $\gamma^r \to BO_r$ be the universal $r$-plane bundle. We write $TBO_r = T(\gamma^r)$ and $TB_r=T(f_r^*\gamma^r)$ for the corresponding Thom spaces. In our applications, we set
$$
\Bl_r \coloneqq \Bl \times_{BO} BO_r
$$
and let $\widetilde{\xi}_\ell^{(r)} \colon \Bl_r \to BO_r$ denote the pullback of the fibration $\widetilde{\xi}_{\ell} \colon \Bl \to BO$.

\begin{proposition}\label{prop:stableThom}
The sequence $(B,f)$ admits a stable Thom class
$$
U=(U_r),\qquad U_r\in \widetilde{H}^r(TB_r;\bZ^w),
$$
such that, for every $b\in B_r$, the restriction of $U_r$ to the Thom space of the fibre over $b$ is a generator. Moreover, under the natural suspension identification, these classes satisfy
$$
(Tg_r)^*(U_{r+1})=\Sigma U_r .
$$
\end{proposition}

\begin{proof}
Let $u_r \in \widetilde{H}^r(TBO_r;\bZ^w)$ denote the universal twisted Thom class of the universal $r$-plane bundle, and define $U_r=(Tf_r)^*(u_r)$. By naturality of the Thom class, the restriction of $U_r$ to the Thom space of each fibre is a generator. The identity $f_{r+1}\circ g_r=j_r\circ f_r$ induces the commutative diagram \cite[p.~18]{Stong:1968}:
$$
\xymatrix@C+5pt{
\Sigma TB_r
   \ar[d]_{\Sigma Tf_r}
   \ar[r]^{Tg_r}
&
TB_{r+1}
   \ar[d]^{Tf_{r+1}}
\\
\Sigma TBO_r
   \ar[r]_{Tj_r}
&
TBO_{r+1}.
}
$$
The universal Thom classes are compatible under stabilization, so that $(Tj_r)^*(u_{r+1})=\Sigma u_r$. Consequently,
$$
\begin{aligned}
(Tg_r)^*(U_{r+1})
&=(Tg_r)^*(Tf_{r+1})^*(u_{r+1})  \\
&=(\Sigma Tf_r)^*(Tj_r)^*(u_{r+1}) \\
&=(\Sigma Tf_r)^*(\Sigma u_r)
 =\Sigma U_r ,
\end{aligned}
$$
which proves the required compatibility. 
\end{proof}

\subsection{The bordism exact sequence in the braid}\label{subsec:lashoftwo}
We use the Lashof bordism theory to define  the exact sequences in the braid of Theorem B.   Note that by  Proposition \ref{prop:stableThom}, there is a coherent way to assign Thom classes used to transport local orientations from $BO_r$  or $BTOP_r$ for the normal bundles of manifolds embedded in codimension $r$.

\begin{definition} For each $k \geq 0$,
\begin{enumerate}
    \item the bordism group $\Omega_{n+k}(\Ml)$ consists of bordism classes of pairs $(N,\hat{f})$ where
\begin{itemize}
    \item $N$ is a closed $(n+k)$-manifold with stable normal bundle $\nu_N$,
    \item $\hat{f} \colon N \to \Ml$ is a map satisfying $\xi_\ell\circ \hat{f} =\nu_N$. 
\end{itemize}
Two such pairs $(N,\hat{f})$ and $(N',\hat{f'})$ are \emph{bordant} if there exists a compact  $(n+k+1)$-manifold $R$ with boundary $\partial R = N \sqcup (-N')$
 and a map $\wF\colon R \to \Ml$ such that $\wF|_{\partial R}= \hat f \sqcup \hat f'$ and $\xi_\ell\circ \wF= \nu_R$  is the stable normal bundle of $R$.
    \item Similarly, for the space $\Bl$, the bordism group $\Omega_{n+k}(\Bl)$
    is defined using $\wxi_\ell$.
    \item The relative bordism group $\Omega_{n+k+1}(\Bl,\Ml)$
     consists of bordism classes of 
     triples $(W,\wF,\bd \wF)$  where
    \begin{itemize}
        \item $W$ is a compact $(n+k+1)$-manifold with boundary $\bd W$,
        \item $\wF \colon W \to \Bl$ is a map satisfying $\wxi_{\ell} \circ \wF=\nu_W$ where $\nu_W$ is the stable normal bundle of $W$,
        \item $\bd \wF\colon \bd W \to \Ml$ is a lift of $\wF|_{\bd W}$ to $\Ml$, i.e., $\wF|_{\bd W} = \tilde{p}_\ell \circ \bd \wF$ (see diagram \ref{diag:twotwo}), defining a $\xi_{\ell}$-structure on $\bd W$.
    \end{itemize}

Two such triples $(W,\wF, \bd \wF)$ and $(W',\wF',\bd \wF')$ are bordant if there is a compact $(n+k+2)$-manifold with corners $Y$ whose boundary faces are $\bd_0 Y = W \sqcup (-W')$ and $\bd_1 Y = V$, where $\bd V = \bd W \sqcup (-\bd W')$, together with maps $\widehat{T} \colon Y \to \Bl$, $\bd \widehat{T} \colon V \to \Ml$, which extend the given data and satisfy $\widehat{T}|_V  = \tilde{p}_{\ell} \circ \bd \widehat{T}$.

    \item The group operation on these bordism groups is given by disjoint union and the inverse is obtained by reversing the chosen local orientation.
    \end{enumerate}
\end{definition}

Since $\tilde{p}_{\ell} \colon \Ml \to \Bl$ is a map of normal structures $\xi_{\ell}$ and $\tilde{\xi}_{\ell}$ by $\xi_{\ell}= \tilde{\xi}_{\ell} \circ \tilde{p}_{\ell}$, these groups fit into the Lashof bordism exact sequence
$$
\cdots \to \O_{n+k+1}(\Ml) \xrightarrow{(\wqq_{\ell})_*}
\O_{n+k+1}(\Bl) \to \O_{n+k+1}(\Bl,\Ml) \xrightarrow{\bd} \O_{n+k}(\Ml) \to \cdots
$$
where  $(\wqq_{\ell})_*$ is induced by postcomposition with $\tilde{p}_{\ell}$, and the boundary homomorphism is given by $\bd (W,\wF,\bd \wF) = (\bd W,\bd \wF)$.

\begin{remark}\label{rem:fourone} If $W\subset \bbR^\infty$ is a (locally flat) embedding of a compact manifold with boundary, we assume that $(W, \bd W) \subset (\bbR^N \times \bbR_{+}, \bbR^{N} \times 0)$ is \emph{neatly embedded},  for some large $N > 0$, meaning that $W \cap \bbR^{N} \times 0= \bd W$ and the embedding is a product on a collar neighbourhood of the boundary. 
\end{remark}

\subsection{Fibre bordism groups} \label{sec:the-bordism-group-htildemk}

Let $M\subset \bbR^\infty$ be a closed, smooth, connected  $n$-dimensional manifold  (as above) and fix an element $\wm \in \heqpt{\Ml}$. We define bordism models for the homotopy fibres of the maps $\alpha$ and $\beta$ in the square of Theorem A.

\begin{convention}
	Let $k \geq 0$. We equip the standard smooth $k$-sphere $S^k$ with the cell structure  given by expressing $S^k = D^k_+ \cup D^k_-$ as the union of two $k$-disks  glued along their boundary. We choose the base point $s_0 = D^0_{-} \in S^0$, which determines a base point $s_0 \in S^k$ for all $k \geq 0$.  If $M$ has  the base point $m_0$, then $S^k \times M$ has the natural base point $(s_0, m_0)$. For each $k \geq 0$, $D^k_-$ has $s_0$ as its base point. For $k >0$, $D^k_+$ also has $s_0$ as its base point, and for $k=0$, $D^0_+$ is a point and serves as its own base point.
\end{convention}

\begin{definition}\label{def:the-bordism-group-htildemk}
For $k \geq 0$, define $\htildek{k}{\Mk, \wm}$ to be the set of bordism classes of pairs $(W, \wF)$, where
\begin{itemize}
	\item $W\subset \bbR^\infty$ is a compact, 
	 $(n+k+1)$-manifold equipped with a fixed identification $\bd W \cong S^k \times M$,
	\item the boundary $\bd W$ decomposes as $\bd W= \bd_- W \cup \bd_+ W$ where $\bd_{\pm} W= D^k_{\pm} \times M$,
	\item $\wF \colon W \to \Ml$ is a continuous map satisfying
	\begin{enumerate}
		\item $\wF|_{\{x_+\} \times M} \in \heqpt{\Ml}$ for every $x_+ \in D^k_+$,
		\item $\wF|_{D^k_- \times  M}= \wm \circ p_M$ where $p_M$ is the projection onto $M$.
		\item $\nu_W = \xi_{\ell} \circ \wF$  is the stable normal bundle of $W$.
	\end{enumerate}
\end{itemize}
Two such pairs $(W,\wF)$ and $(W',\wF')$ are \emph{bordant} if there exist
\begin{itemize}
	\item a  homotopy $\hat{h}\colon S^k \times M \times I \to \Ml$ between $\wF|_{\bd W}$ and $\wF'|_{\bd W'}$ satisfying
	\begin{enumerate}
		\item $\hat{h}|_{D^k_- \times M \times \{t\}}= \wm \circ p_M$ for all $t \in I$,
		\item $\hat{h}|_{\{x_+\} \times M \times \{t\}} \in \heqpt{\Ml}$ for all $x_+ \in D^k_+$ and $t \in I$,
	\end{enumerate}
    \item a compact, $(n+k+2)$-manifold with corners $Y$,  and boundary identified as $\bd Y= W \cup (S^k \times M \times I) \cup -W'$, 
	\item a continuous map $\wG \colon Y \to \Ml$ extending $\wF$, $\hat{h}$, and $\wF'$ and satisfying $\xi_{\ell} \circ \wG=\nu_Y$ where $\nu_Y \colon Y \to BO$ is the stable normal bundle of $Y$.
\end{itemize}
In other words, the bordism group element 
$$
[(W \cup_{\partial W= S^k \times M \times \{0\}} (S^k \times M \times I) \cup_{\partial W' = S^k \times M \times \{1\}} -W', \wF \cup \hat{h} \cup  \wF'] \in \Omega_{n+k+1}(\Ml )
$$
obtained by gluing $W$ and $W'$ to the boundary components of the cylinder $S^k \times M \times I$ represents zero.  If $\wm = \hatid_M$, then we write  $\htildeMk{k} = \htildek{k}{\Mk, \hatid_M}$. The elements of $\htildeMk{k}$ can be thought of as bordism classes of manifolds $W$ with boundary $S^k \times M$, equipped  with maps to $\Ml$ that are compatible with the identity of $M$ on one hemisphere and homotopy self-equivalences of $M$ on the other.  The reference map $\wF$ induces a preferred local orientation on $W$ compatible with the given local orientation on $M$.
We refer to Section \ref{subsec:groups} and Section \ref{subsec:groupstwo} for the definition of the group structure on $\htildeMk{k}$.
\end{definition}

Similarly, we define $\htildeBk{k}$ in the context of the space $\Bl$. Let $M\subset \bbR^\infty$ be a connected, locally oriented $n$-dimensional manifold and fix an element $\wb \in \heqpt{\Bl}$. 

\begin{definition}\label{def:the-bordism-group-htildebk}
For $k \geq 0$, we define $\htildek{k}{\Bl, \wb}$ to be the set of bordism classes of pairs $(Z,\wG)$, where  
\begin{itemize}
	\item $Z\subset \bbR^\infty$ is a compact,   
	$(n+k+1)$-manifold equipped with a fixed identification $\partial Z \cong S^k \times M$,
	\item the boundary decomposes as $\partial Z = \partial_- Z \cup \partial_+ Z$ where $\partial_{\pm} Z = D^k_{\pm} \times M$,
	\item $\wG \colon Z \to \Bl$ is a continuous map satisfying
	\begin{enumerate}
		\item $\wG|_{D^k_- \times  M} = \wb \circ p_M$ where $p_M \colon D^k_- \times M \to M$ is the projection,
		\item for each $x_+ \in D^k_+$, the restriction $\wG|_{x_+ \times M} \in \heqpt{\Bl}$.
		\item $\nu_Z = \tilde\xi_{\ell} \circ \wG$ where $\nu_Z \colon Z \to  BO$ is the stable normal bundle of $Z$.
	\end{enumerate}
\end{itemize}
Two such pairs $(Z,\wG)$ and $(Z',\widehat{G'})$ are \emph{bordant} if there exist
\begin{itemize}
	\item a  homotopy $\hat{h} \colon S^k \times M \times I \to \Bl$ between $\wG|_{\bd Z}$ and $\wG'|_{\bd Z'}$ satisfying
	\begin{enumerate}
		\item $\hat{h}|_{D^k_- \times M \times \{t\}}= \wb \circ p_M$ for all $t \in I$,
		\item $\hat{h}|_{\{x_+\} \times M \times \{t\}} \in \heqpt{\Bl}$ for each $x_+ \in D^k_+$ and $t \in [0,1]$,
	\end{enumerate}
	\item a compact,   $(n+k+2)$-manifold $V$ with boundary $\bd V= Z \cup (S^k \times M \times I) \cup -Z'$,
	\item a continuous map $\wH \colon V \to \Bl$ extending $\wG$, $\hat{h}$, and $\wG'$ satisfying $\wxi_{\ell} \circ \wH=\nu_V$ where $\nu_V$ is the stable normal bundle of $V$.
\end{itemize}
In other words, the bordism group element
$$
[(Z \cup_{\bd Z= S^k \times M \times \{0\}} (S^k \times M \times I) \cup_{\bd Z' = S^k \times M \times \{1\}} -Z', \wG \cup \hat{h} \cup -\widehat{G'})] \in \Omega_{n+k+1}(\Bl )
$$
obtained by gluing $Z$ and $Z'$ to the boundary components of the cylinder $S^k \times M \times I$ represents zero.
If $\wb = \wpp_\ell(\hatid_M)= \hat\nu_\ell$, then we write $\htildeBk{k} = \htildek{k}{\Bk, \wpp_\ell(\hatid_M)}$.
\end{definition}

\subsubsection{The group structures on $\htildeMk{k}$ and $\htildeBk{k}$ for $k \geq 1$}\label{subsec:groups}
Write $S^k = D_-^k \cup_{S^{k-1}} D_+^k$ and let $\Delta \colon S^k \to S^k\vee S^k$ be the standard pinch map. We choose $\Delta$ compatibly with the hemisphere decompositions, so that the negative hemisphere of the domain is mapped into the union of the negative hemispheres of the two wedge summands. A geometric realization of $\Delta$ is given by the $(k+1)$-dimensional pair-of-pants cobordism 
$$
P^{k+1} = S^{k+1} \setminus \bigcup_{i=1}^3\operatorname{int}(B_i^{k+1}) \qquad \partial P^{k+1}
= (S^k\sqcup S^k)\sqcup(-S^k).
$$ 
Choose a map $\rho \colon P^{k+1} \to S^k\vee S^k$ whose restrictions to the two incoming boundary components are the inclusions of the first and second wedge summands, and whose restriction to the outgoing boundary is $\Delta$. We furthermore choose $\rho$ so that the negative hemisphere of the outgoing boundary is mapped into the wedge of the negative hemispheres of the incoming boundary components.

Let $[(W,\wF)],[(W',\wF')] \in \htildeMk{k}$, and let $u,u'\colon S^k \to \heqpt{\Ml}$ be the based maps adjoint to the boundary restrictions of $\wF$ and $\wF'$, respectively. Thus $u(x)(m)=\wF(x,m)$ and $u'(x)(m)=\wF'(x,m)$. We form 
$$
W'' = W \cup_{\partial W} (P^{k+1}\times M) \cup_{\partial W'}W'
$$
where the boundaries of $W$ and $W'$ are identified with the two incoming boundary components of $P^{k+1}\times M$. On the middle piece, define $\wF_P(p,m) = \bigl((u\vee u')(\rho(p))\bigr)(m)$. This agrees with $\wF$ and $\wF'$ on the two incoming boundary components, and hence determines a reference map $\wF''\colon W''\to \Ml$. On the outgoing boundary, the adjoint of $\wF''$ is the composite 
$$
S^k \xrightarrow{\Delta} S^k \vee S^k \xrightarrow{u\vee u'} \heqpt{\Ml}
$$
which represents the sum of $[u]$ and $[u']$ in $\pi_k(\heqpt{\Ml})$. The condition imposed on the negative hemisphere ensures that this outgoing family is again based in the sense of Definition \ref{def:the-bordism-group-htildemk}. We define $[(W,\wF)]+[(W',\wF')]:= [(W'',\wF'')]$. The same construction, with $\Ml$ replaced by $\Bl$, defines the group operation on $\htildeBk{k}$. Standard pair-of-pants bordisms show that these operations are well-defined and associative. The identity elements are represented by $[(D^{k+1} \times M, \widehat{\id}_M \circ p_M)]$ and $[ (D^{k+1} \times M,\hat{\nu}_{\ell} \circ p_M)]$, respectively.

\begin{lemma}\label{lem:fiveseven}
	For $k \geq 1$, the set $\htildeMk{k}$ with the operation defined above is a group.
\end{lemma}
\begin{proof} 
	The details are left to the reader.
\end{proof}

The group operation on $\htildeBk{k}$ for $k \geq 1$ is similarly  defined using the pair of pants manifold and appropriate gluing constructions. The identity element is represented by $[D^{k+1} \times M, \hat\nu_{\ell} \circ p_M]$.
\begin{lemma}\label{lem:fiveeight}
	For $k \geq 1$, the set $\htildeBk{k}$ with the operation described above is a group.
\end{lemma} 
\begin{proof}
	The details are left to the reader.
\end{proof}

\subsubsection{The group structures on   $\htildeMk{0}$ and  $\htildeBk{0}$}
\label{subsec:groupstwo}
When $k=0$, the sphere $S^0$ consists of two points, $s_0$ and $s_1$, with $s_0$ serving as the base point. In this case, elements of $\htildeMk{0}$ are represented by pairs $(W,\wF)$ where $W$ is a compact,   $(n+1)$-manifold with boundary $\bd W= -M \sqcup M$, decomposed as $\bd_- W= \{s_0\} \times M$ and $\bd_+ W= \{s_1\} \times M$, and $\wF \colon W \to \Ml$ satisfies $\wF|_{\bd_-W} =\widehat{\text{id}}_M$ and $\wF|_{\bd_+ W}=\hat{f} \in \heqpt{\Ml}$. 

The group structure in this case differs in that it also involves the monoid structure on the space of self-equivalences. 
Given two such pairs $(W,\wF)$ and $(W',\wF')$, we define their composition 
  $(W \bullet W',\wF \bullet \wF')$ where 
  $W \bullet W'$ is the manifold obtained by gluing $W$ and $W'$ along $\bd_+W$ and $\bd_-W'$ via the identity map on $M$. The map 
   $\wF\bullet \wF' \colon W \bullet W' \to \Ml$ is defined to be $\wF$ on $W$ and $\hat{f} \bullet \wF'$ on $W'$ where $\hat{f}= \wF|_{\bd_+ W}$ and $\hat{f} \bullet \wF'=(f \circ F', \nu_{W'}, H_{f \circ F'})\colon W' \to \Ml$ with $f = j \circ \hat{f}$, $F'= j \circ \wF'$, $\nu_{W'} \colon W' \to  BO$ is the stable normal bundle of $W'$, and $H_{f \circ F'} \colon W' \times I \to P_{\ell}(BO)$ is the homotopy from $w_{\ell}\circ \nu_{W'}$ to $\gamma_{\ell} \circ f \circ F'$ defined by
\begin{displaymath}
	H_{f \circ F'}(w,t) = \left\{
	\begin{array}{lr}
		H_{F'}(w,2t) & t \in [0,1/2]\ \\
		H_f (F'(w), 2t-1) &  t \in [1/2,1]
	\end{array}
	\right..
\end{displaymath}

Note that the restriction of 
 $\wF \bullet \wF'$ to $\bd_+W'$ is $\hat{f} \bullet \hat{f}'$ where $\hat{f}'=\wF'|_{\bd_+ W'}$. The identity element for this operation is $(M \times I, \widehat{\text{id}}_M \circ p_M)$ where $p_M \colon M \times I \to M$ is the projection. 

To show that each element in $\htildeMk{0}$ has an inverse, consider $[(W,\wF)] \in \htildeMk{0}$. Let $\hat{f}^{-1} \in \heqpt{\Ml}$ be a homotopy inverse of $\hat{f}= F|_{\bd_+ W}$. Define a representative
$(W^{-1}, \wF^{-1})$ of $[(W,\wF)]^{-1}$ as
$$(-W, \hat{f}^{-1} \bullet\wF) \cup (M \times I, \wH)$$
where $\wH$ is a homotopy between
$\hat{f}^{-1} \bullet \hat{f}$ and $\hatid_M$ (compare \cite[2.2]{Hambleton:2004}).
The identification of boundary components of the composition 
$(W^{-1},\wF^{-1}) \bullet (W,\wF)$ yields the double of the manifold 
$(W^{-1}, \wF^{-1}$) which bounds.  Therefore, this composition is equivalent to the identity element $(M \times I, \widehat{\text{id}}_M \circ p_M)$, showing that $\htildeMk{0}$ is indeed a group.

There is a morphism $\Omega_{n+1}(\Ml,\xi_{\ell}) \to \htildeMk{0}$ sending a closed manifold $N$ equipped with a map $\wF\colon N \to \Ml$ to the class $[N \sqcup (M \times I), \wF \sqcup (\hatid_M \circ p_M)]$, or equivalently to the internal connected sum $[N \# (M \times I), \wF \# (\hatid_M\circ p_M)] $. We also have a morphism $\htildeMk{0} \to \hept{\Ml} = \pi_{0}(\heqpt{\Ml})$ sending $[(W,\wF)]$ to $[\wF|_{\bd_+ W}=\hat{f}] \in \hept{\Ml}$. 

\begin{lemma}\label{lem:fivenine}
	The two maps $\Omega_{n+1}(\Ml ) \to \htildeMk{0}$ and $\htildeMk{0} \to \hept{\Ml}$ are group homomorphisms.
\end{lemma}
\begin{proof}
	The details are left to the reader.
\end{proof}

As in the case of $\htildeMk{0}$, we can define a group structure on $\htildeBk{0}$. The output of the group operation 
  $(W \bullet W',\wF \bullet \wF')$ is again given by the manifold 
   $W \bullet W'$ obtained by gluing $W$ and $W'$ along $\bd_+W$ and $\bd_-W'$ via the identity map on $M$, and the map 
    $\wF \bullet \wF' \colon W \bullet W' \to \Bl$ which is $\wF$ on $W$ and $\hat{f} \bullet \wF'$ on $W'$. However, this time we replace $\hat{f}= \wF|_{\bd_+ W}$ by a map $\hat{f}^c \in \heqptkc{B}{\ell}$ with the property that its image under $c\colon \heqptkc{B}{\ell} \to \heqpt{\Bl}$ is in the connected component of $\hat{f}$ and define $\hat{f} \bullet \wF'=(f^c \circ F', \nu_{W'}, H_{f^c \circ F'})\colon W' \to \Bl$.  
The identity element in this group is represented by  $[(M \times I, \hat\nu_{\ell} \circ p_M)]$. There are similarly defined morphisms $\Omega_{n+1}(\Bl ) \to \htildeBk{0}$ and $\htildeBk{0} \to \hept{\Bl} = \pi_{0}(\heqpt{\Bl})$.

\begin{lemma}\label{lem:fiveten}
	The two maps $\Omega_{n+1}(\Bl) \to \htildeBk{0}$ and $\htildeBk{0} \to \hept{\Bl}$ are group homomorphisms.
\end{lemma}
\begin{proof}
	The details are left to the reader.
\end{proof}

\section{Semi-simplicial models and the square}\label{sec:ss_models}

\subsection{Semi-simplicial models for spaces of self-equivalences}\label{sec:quinnad}

\begin{definition}  We will abuse notation and let  $\heqpt{\Ml}$ also denote the Kan semi-simplicial set whose $d$-simplices are  the  singular $d$-simplices  $x_d\colon  \Delta^d \to \heqpt{\Ml}$, such that
	 there exists $\varepsilon >0$ such that  
	$x_d|_{\Delta^d_i(\varepsilon)} = x_d|_{\Delta^{d-1}_i} \circ \pi_i(\varepsilon)$, for every $0 \leq i \leq d$,
		   where $\Delta^d_i(\varepsilon)$ is a thickening of
 the $i$-th face $\Delta_i^{d-1}$ in $\Delta^d$ and $\pi_i(\varepsilon)\colon \Delta^d_i(\varepsilon) \to \Delta_i^{d-1}$ is the projection  (see the notation below).  Similar conventions apply for $\heqpt{\Bl}$.
	These are weakly equivalent \emph{face-collared} sub-semi-simplicial sets of the usual  singular complex of the spaces in Definition \ref{def:thickaut} and Definition \ref{def:thickautB}.
\end{definition}

In the following, we will mainly work with these face-collared models and rely on the context to distinguish them from the spaces in Definition \ref{def:thickaut} and Definition \ref{def:thickautB}.

\subsection{Semi-simplicial models for Lashof bordism}\label{sec:quinnadtwo}
In this section, we use ``mock bundles'' \cite{Buoncristiano:1976} to generalize the setting of \cite[Section 4.2]{Randal-Williams:2024} to that of Lashof-type bordism theories, which are based on topological manifolds with extra structure on the normal microbundle (see \cite{Lashof:1963} and \cite[Chapter II]{Stong:1968}). Our goal is to describe a Kan semi-simplicial model for the $n$-fold loop spaces of these bordism theories.

Our notation and setup is chosen close to that of \cite[Section 4.2]{Randal-Williams:2024}. We let $\Delta^p$ denote the standard $p$-simplex, let $\Delta^{p-1}_i \subset \Delta^p$ denote its $i$-th face, and for $\varepsilon  > 0$ write $\Delta^p_i(\varepsilon )$ for the thickened $i$-th face
$\Delta^p_i(\varepsilon ) = \{(t_0, \dots, t_p) \in \Delta^p \vv 0 \leq t_i < \varepsilon \}$. There is a projection map
$$\pi_i(\varepsilon )\colon  \Delta^p_i(\varepsilon ) \to \Delta^{p-1}_i.$$

\begin{definition}\label{def:fiveone}
Let $\xi \colon X \to \on{BTOP}$ be a fibration classifying the stable normal bundle data of a Lashof bordism theory. For $n \in \mathbb{Z}$, we define $\on{Mock}(\xi, n)$ to be the semi-simplicial set with $p$-simplices given by pairs $(W, \hat f)$ consisting of
	\begin{itemize}
		\item a locally flat compact topological $(n+p)$-submanifold $W \subset \Delta^p \times \mathbb{R}^\infty$
		\item \emph{a $\xi$-structure on $W$} given by a map $W \to X$ and a stable microbundle trivialization of 
		 $\tau_W \oplus \hat{f}^*(\xi)$,
		where $\tau_W$ is the tangent microbundle of $W$.  In other words, we have a triangle of continuous maps
		$$
		\begin{tikzcd}
			W \arrow[r, "\hat f"] \arrow[dr, "\nu_W", swap] & X \arrow[d, "\xi"]\\
			& \on{BTOP}
		\end{tikzcd}
		$$
		which commutes  	where $\nu_W$ is the classifying map of the stable normal microbundle of $W \subset \Delta^p \times \mathbb{R}^\infty$ (see Kirby and Siebenmann \cite[IV.8]{Kirby:1977}, Milnor \cite{Milnor:1964}, Kister \cite{Kister:1964}, Rourke and Sanderson \cite{Rourke:1970}).
	\end{itemize}
	For each $(W, \hat f)$  there  exists   some $k \in \mathbb{N}$ such that the following conditions hold for every $0 \leq i \leq p$:
	\begin{enumerate}
		\item $W$ is topologically transverse to each $\Delta_i^{p-1} \times \mathbb{R}^k$.
		\item It holds that $\partial W = W \cap (\partial \Delta^p \times \mathbb{R}^k)$.
		We write $\partial_i W = W \cap (\Delta_i^{p-1} \times \mathbb{R}^k)$ and
		$\partial_i^\varepsilon  W = W \cap (\Delta_i^{p}(\varepsilon ) \times \mathbb{R}^k)$.
		\item For each face, there exists an $\varepsilon  > 0$
		such that $\bd_i^\varepsilon  W= (\pi_i(\varepsilon) \times \id_{\mathbb{R}^k})^{-1} \partial_i W$  the composite
		$$
        \partial_i^\epsilon(\hat f)\colon \partial_i^\varepsilon  W \xrightarrow{\pi_i(\varepsilon )}
		\partial_i W \xrightarrow{\partial_i(\hat f)} X
        $$
		is compatible with the restriction of the $\xi$-structure on $W$ to the (thickened) faces $\bd_i^\varepsilon W$.
		The $i$-th face map $d_i\colon  \on{Mock}(\xi,n)_p \to \on{Mock}(\xi,n)_{p-1}$
		is given by restriction to $\partial_i W = W \cap (\Delta_i^{p-1} \times \mathbb{R}^\infty)$.	
	\end{enumerate}
\end{definition}
\begin{remark} If $N^n \subseteq \bR^{n+k}$ is a compact (locally flat) topological $n$-submanifold, with $k \gg n$, then the tangent microbundle $\tau_N$ of $N$ has a complementary normal microbundle $\nu_N$ with $\tau_N \oplus \nu_N \cong \epsilon^{n+k}$. By stability, this normal microbundle admits a unique topological $\bR^k$-bundle structure, and (by \cite[Corollary 5.5]{Rourke:1970}) an associated $D^k$-bundle with a Thom collapse map $S^{n+k} \to T(\nu_N)$. If there exists a map $f\colon N \to X$, as above, such that $f^*(\xi) \cong \nu_N$ (after stabilization), topological transversality provides an element in the Lashof bordism group $\Omega_n(X, \xi) \cong \pi^s_n(T(\xi))$. This construction provides an alternate description of the bundle data condition in Definition \ref{def:fiveone}.
\end{remark}

\begin{lemma}
	$\on{Mock}(\xi,n)$ is a Kan semi-simplicial set.
\end{lemma}
\begin{proof}
	This argument is analogous to the one in \cite[Remark 2.2.2]{Hebestreit:2021} (see also the discussion after \cite[Definition 2.3.1]{Hebestreit:2021}), which is why we will be brief.
	
	Let $\Lambda^{p,j} \to \on{Mock}(\xi,n)$ be some horn. By definition,
	this  is equivalent  to the existence of some $k \in \mathbb{N}$ and a $((p-1)+n)$ dimensional submanifold $W' \subset \Lambda^{p,j} \times \mathbb{R}^k$ that is equipped with a $\xi$-structure $\hat f' \colon  W' \to X$ satisfying the conditions above. Note that  $\Lambda^{p,j}$ is a union of $(p-1)$-simplices (the union of codimension one faces of $\Delta^p$ except the $j$-th face) and each satisfies the condition (iii) above. 
	
	As in \cite[Remark 2.2.2]{Hebestreit:2021}, we let $\varepsilon > 0$ denote a minimum value, so that condition (iii) holds for all these $(p-1)$-simplices. We then consider the neighborhood $\Lambda^{p,j}(\varepsilon ) = \bigcup_{i \neq j} \Delta^p_i(\varepsilon )$ of $\Lambda^{p,j}$ in $\Delta^p$ and the $(p+n)$-dimensional submanifold $W'(\varepsilon) \subset \Lambda^{p,j}(\varepsilon ) \times \mathbb{R}^k$ given by 
	 $\bigcup_{i \neq j} (\pi_i(\varepsilon) \times \id_{\mathbb{R}^k})^{-1} \partial_i W'.$ 
	 		These thickenings are compatible under overlaps by definition of  the horn $\Lambda^{p,j}(\varepsilon )$ and condition (iii) above. 
	
	We note that the map $\hat{f}'$ extends over the thickening to a  $\xi$-structure $\hat{f}'(\varepsilon)\colon W'(\varepsilon) \to X$ as a post-composition with projections $\pi_i(\varepsilon)$, and that this extension continues to satisfy the conditions above by its definition.
	As in \cite[Remark 2.2.2]{Hebestreit:2021}, we may now pick an isotopy between the identity on $\Delta^p$ and a suitable embedding $\Delta^p \hookrightarrow \Lambda^{p,j}(\varepsilon )$ which restricts to the identity on a neighborhood of $\Lambda^{p,j}$ to obtain a $(p+n)$-submanifold $W \subset \Delta^p \times \bbR^k$ and a $\xi$-structure $f\colon  W \to X$ that defines a lift $\Delta^{p} \to \on{Mock}(\xi,n)$ of $\Lambda^{p,j} \to \on{Mock}(\xi,n)$.
\end{proof}

\begin{corollary}\label{cor:fivefour}  For $i \geq 0$, $\pi_{i}(\on{Mock}(\xi,n))\cong  \Omega_{n+i}(X,\xi)$.
\end{corollary}
\begin{proof} This example is explained in \cite[Chapter IV.2, p.~83]{Buoncristiano:1976}. In particular, the base point is the empty manifold in every simplicial degree, an $i$-simplex whose faces are all empty is a closed $(n+i)$-manifold with a $\xi$-structure, and an $(i+1)$-simplex realizing a semi-simplicial homotopy is precisely a $\xi$-bordism between such representatives.
\end{proof}

We may similarly define the following ``thickened'' version of $\on{Mock}(\xi, n)$.

\begin{definition}
	Let $\xi\colon  X \to \on{BTOP}$ be a fibration. For $n \in \mathbb{Z}$, we define $\on{Mock}_\partial(\xi, n+1)$ to be the semi-simplicial set with $p$-simplices given by pairs $(W, \hat f)$ consisting of
	\begin{itemize}
		\item a locally flat compact topological $(n+p+1)$-submanifold $W \subset \Delta^p \times \mathbb{R}^\infty \times [0, \infty)$
		\item \emph{a $\xi$-structure on $W$} given by a map $W \to X$ and a stable microbundle trivialization of $\tau_W \oplus \hat f^*(\xi)$.
		In other words, we have a commutative triangle of continuous maps, 		
		$$
		\begin{tikzcd}
			W \arrow[r, "\hat f"] \arrow[dr, "\nu_{W}", swap] & X \arrow[d, "\xi"]\\
			& \on{BTOP}
		\end{tikzcd}
		$$
		where $\nu_{W}$ is the   classifying map of the stable normal microbundle of $W \subset \Delta^p \times \mathbb{R}^\infty \times [0, \infty)$.
	\end{itemize}
	For each $(W, \hat f)$ there exists some $k \in \mathbb{N}$ such that the conditions (i) and (iii) in the previous definition hold for every $0 \leq i \leq p$. In addition we have that:
	\begin{enumerate}
		\item[(iv)] $W$ is topologically transverse to $\Delta^p \times \mathbb{R}^k \times 0$, and writing $\partial_v W = W \cap (\Delta^p \times \mathbb{R}^k \times 0)$, there exists a $\delta > 0$ such that $W \cap (\Delta^p \times \mathbb{R}^k \times [0, \delta)) = \partial_v W \times [0, \delta)$ $= ( \id_{\Delta^p} \times \id_{\mathbb{R}^k} \times \pi_v^{\delta})^{-1} \bd_{v}W$. If we  denote this collar by $\partial_v^\delta W$, then $\partial W = (W \cap (\partial \Delta^p \times \mathbb{R}^k \times [0, \infty)) \cup \partial_v W$ and the composite
		$$\partial_v^\delta (\hat f)\colon \partial^\delta_v W  \xrightarrow{\ \pi_v(\delta )\ }
		\partial_v W \xrightarrow{\ \partial_v(\hat f)\ } X$$
		is compatible with the restriction of the $\xi$-structure on $W$ to the (thickened) faces $\partial_v^\delta W$. 
	\end{enumerate}
\end{definition}

These constructions lead to the following description of the Lashof bordism theory:
\begin{theorem}\label{thm:lashof_space} For any fibration $\xi\colon X \to BTOP$, there is an $\Omega$-spectrum $\on{Mock}(\xi)$ with the property that
	$$\Omega^{\infty + n} \on{Mock}(\xi) \simeq \|\on{Mock}(\xi, n)\|, \text{ for all } n \in \mathbb{Z},$$
	such that on the level of homotopy groups, it holds that
	$$\pi_i(\Omega^{\infty + n} \on{Mock}(\xi)) = \Omega_{n+i}(X, \xi) \text{ for all } n \in \mathbb{Z}.$$
\end{theorem}

\begin{proof} The following argument is analogous to that in \cite[Section 4]{Randal-Williams:2024}.
	One can first check as above that $\on{Mock}_\partial(\xi, n+1)$ is a Kan semi-simplicial set. Restricting along the inclusion $\Delta^p \times \mathbb{R}^\infty \times 0 \hookrightarrow \Delta^p \times \mathbb{R}^\infty \times [0, \infty)$ defines a map of semi-simplicial sets
	$$\on{Mock}_\partial (\xi,n+1) \to \on{Mock}(\xi, n)$$
	which is a Kan fibration with fibre $\on{Mock}(\xi, n+1)$. To see this, fix a horn $\Lambda^{p,j} \to \on{Mock}_{\partial}(\xi,n+1)$ whose restriction in $\on{Mock}(\xi,n)$ extends to a $p$-simplex $\Delta^{p} \to \on{Mock}(\xi,n)$ which we denote by a pair $(W,\hat{f})$.

    Write $(Z_i,\wF_i)$, for $i\neq j$, for the compatible $(p-1)$-simplices defining the horn. Thus $Z_i \subset \Delta_i^{p-1} \times \mathbb{R}^\infty \times[0,\infty)$, and their vertical boundaries agree with the corresponding faces of the prescribed simplex $(W,\hat{f})$; that is, $\bd_v Z_i=\bd_i W$. Choose $\varepsilon,\delta>0$ sufficiently small so that all horizontal and vertical collar conditions hold simultaneously. We first thicken the horn in the simplex direction by setting
    $$
    Z_{\Lambda}(\varepsilon) = \bigcup_{i\neq j} \bigl( \pi_i(\varepsilon) \times \id_{\mathbb{R}^{\infty}} \times \id_{[0,\infty)} \bigr)^{-1}(Z_i).
    $$
    We then combine this horizontal thickening with the vertical collar $W \times [0,\delta) \subset \Delta^p \times \mathbb{R}^{\infty} \times [0,\delta)$. These manifolds agree on their overlaps. Indeed, for each $i\neq j$, the vertical collar condition on $Z_i$ gives 
    $$
    Z_i \cap \bigl( \Delta_i^{p-1} \times \mathbb{R}^{\infty} \times [0,\delta) \bigr) = \bd_v Z_i \times [0,\delta) = \bd_i W \times [0,\delta).
    $$
    It follows that
    $$
    \bigl( \pi_i(\varepsilon) \times \id_{\mathbb{R}^{\infty}} \times \id_{[0,\infty)} \bigr)^{-1}(Z_i) \cap
    \bigl( \Delta_i^p (\varepsilon) \times \mathbb{R}^{\infty} \times [0,\delta) \bigr) = \bigl( (\pi_i(\varepsilon)\times\id_{\mathbb{R}^{\infty}})^{-1} (\bd_i W) \bigr) \times [0,\delta).
    $$
    By the horizontal collar condition on $W$, this is equal to $(W \cap (\Delta_i^p (\varepsilon) \times \mathbb{R}^{\infty}) \bigr) \times [0,\delta)$ which is precisely the corresponding part of the vertical collar $W \times [0,\delta)$.
    Consequently, they fit together to give a $\xi$-structured manifold over $(\Lambda^{p,j}(\varepsilon) \times \mathbb{R}^{\infty} \times [0,\infty)) \cup ( \Delta^p \times \mathbb{R}^{\infty} \times [0,\delta))$. The reference map is obtained by pulling back the maps $\wF_i$ on the horn part and by using $\hat{f} \circ \pi_W$ on $W \times [0,\delta)$. The collar conditions ensure that these definitions agree on the overlaps. 
    
    Choose an embedding 
    $$
    r=(r_1,r_2) \colon \Delta^p \times[0,\infty) \longrightarrow \bigl(\Lambda^{p,j}(\varepsilon) \times [0,\infty) \bigr) \cup \bigl(\Delta^p \times [0,\delta) \bigr)
    $$
which is the identity on a neighborhood of $\Lambda^{p,j} \times [0,\infty)$ and of $\Delta^p \times \{0\}$. Pulling the structured manifold back along $(r_1,\id_{\mathbb{R}^\infty},r_2)$ produces the required horn filler. Thus, every horn lifts, and the restriction map is a Kan fibration.

	Observe that the fibre over the empty manifold is isomorphic to $\on{Mock}(\xi, n+1)$. In the long exact sequence of this fibration, the connecting map
	$$\pi_i(\|\on{Mock}(\xi, n)\|) \to \pi_{i-1}(\|\on{Mock}(\xi, n+1)\|)$$ is an isomorphism. This uses the domain and target identification with the Lashof bordism groups $\Omega_{i+n}(X, \xi)$ in Corollary \ref{cor:fivefour}, and that the induced map between these identifications is the identity. In other words, the connecting map takes a closed  $\xi$-structured manifold to the same $\xi$-structured manifold, viewed as the vertical boundary of a collar. Thus, it is identity.  It follows that
		$\on{Mock}_\bd(\xi,n+1)$ is a contractible Kan semi-simplicial set and hence that there is a (weak) homotopy equivalence
	$$\Omega \|\on{Mock}(\xi, n)\| \xrightarrow{\simeq} \|\on{Mock}(\xi, n+1)\|.$$
	The result follows by passing to homotopy groups. In other words, these equivalences define the adjoint structure maps of an $\Omega$-spectrum whose $(\infty+n)$-space is weakly equivalent to $\|\on{Mock}(\xi, n)\|$. Corollary \ref{cor:fivefour} identifies its homotopy groups with the Lashof bordism groups.
\end{proof}

\begin{remark}
	An alternate approach to constructing the space-level models for the Lashof-type bordism theory can be based on \cite{Quinn:1995} or  \cite{Laures:2014}.
\end{remark}

\subsection{Construction of the square}
\label{subsection:semi-simplicial-set-and-construction-of-square}

Let  $M^n \subset \bbR^\infty$ be a closed, connected, topological $n$-manifold, $n \geq 4$,  with a preferred (local) orientation and stable normal microbundle $\nu_M \colon M \to BTOP$.
We have constructed two fibrations
$\xi_\ell\colon  \Ml \to BTOP$ and  $ \tilde{\xi}_\ell\colon  \Bl \to BTOP $, for $\ell \geq 2$,
by adapting the setting in Section \ref{sec:three} from  smooth to topological manifolds, and setting $Y = BTOP$ in diagram \eqref{diag:twotwo}.

\begin{definition}[The vertical maps]\label{def:kan_maps_two}
	The map of semi-simplicial sets
	$$\alpha\colon  \heqpt{\Ml} \to \on{Mock}(\xi_\ell, n)$$
	is defined by mapping a face-collared singular $d$-simplex $x_d\colon  \Delta^d \to \heqpt{\Ml}$ to its adjoint map $\hat x_d\colon  \Delta^d \times M \to \Ml$, which gives a $d$-simplex $(\Delta^d \times M,\hat{x}_d)$ in $\on{Mock}(\xi_\ell, n)$. 
	 To see that this map is well-defined, note first that $\Delta^d \times M$ is already a product near every face of $\Delta^d$. 
	
	Secondly, the face-collar condition on the simplex $x_d$ implies that the $\xi_{\ell}$ structure on $\bd_i^{\varepsilon}(\Delta^d \times M)$ given by $\hat{x}_d$ is pulled back from its restriction to $\bd_i(\Delta^d \times M)= \Delta^{d-1}_i \times M$ for each $i$. 
	
	Lastly, the stable normal microbundle of the embedding $\Delta^d \times M \subset \Delta^d \times \mathbb{R}^{\infty}$ is $p_M^* (\nu_M)$ which is in agreement with the fact that every $x_d(t) \in \heqpt{\Ml}$ satisfies $\xi_{\ell} \circ \hat{x}_d = \nu_M \circ p_M$ where $p_M \colon \Delta^d \times M \to M$ is the projection. 
	The map $$\beta\colon  \heqpt{\Bl} \to \on{Mock}(\wxi_\ell, n)$$ is defined similarly. 
\end{definition}

The induced map $\alpha_0 \colon \pi_0( \heqpt{\Ml}) \to \pi_0(\on{Mock}(\xi_\ell, n))$ is given by sending a self-equivalence $\hat f \colon M \to \Ml \in \heqpt{\Ml}$ to the bordism element $\alpha(\hat f) = (M, \hat f) \in \on{Mock}(\xi_\ell, n)$. We observe that there is a degree map  $\deg\colon \pi_0(\on{Mock}(\xi_\ell,n)) \to \bZ$,  defined for a $0$-simplex $\hat f \colon W^n \to \Ml$  by $\deg(\hat f) = \deg(f\colon W^n \to M)$, where $f= j \circ \hat f$ (see Diagram \ref{diag:threezero}).
Since any $\hat f \colon M \to \Ml \in \heqpt{\Ml}$ is locally orientation-preserving and degree is invariant under $\xi_{\ell}$-bordisms, the image of $\alpha_0$ is contained in components of degree one. This leads us to the following definition.

\begin{definition}
	\label{def:restricted-infty-plus-n-bordism-spaces}
	Let $\wMock(\xi_\ell, n)$ be the sub-semi-simplicial set of $\on{Mock}(\xi_\ell, n)$ consisting of all simplices whose vertices are of degree one. 
		\end{definition}

It follows that the map $\alpha$ restricts to a map
$$\alpha\colon  \heqpt{\Ml} \to \widehat{\on{Mock}}(\xi_\ell, n).$$
Since filling a horn cannot leave the connected component, $\widehat{\on{Mock}}(\xi_\ell, n)$ is a Kan semi-simplicial set.

\begin{definition}[The horizontal maps]\label{def:kan_maps_one}
	The maps of semi-simplicial sets
	$$\wpp_\ell\colon \heqpt{\Ml} \to \heqpt{\Bl} \text{ and }
	\wqq_\ell\colon \widehat{\on{Mock}}(\xi_\ell, n) \to \on{Mock}(\wxi_\ell, n)$$
	are defined as follows.	
	For the first map, we recall from Definition \ref{def:thickaut} and Definition \ref{def:thickautB} that the space $\heqpt{\Ml}$ is a subspace of $\on{Map}_\bullet(M, \Ml)$, and similarly $\heqpt{\Bl} \subseteq \on{Map}_\bullet(M, \Bl)$. Post-composition with $\tilde{p}_\ell\colon  \Ml \to \Bl$ yields a map of spaces $\heqpt{\Ml} \to \heqpt{\Bl}$.

	\begin{enumerate}
		\item The first map $\wpp_\ell$ is now given by post-composing a face-collared singular simplex $x_d\colon \Delta^d \to \heqpt{\Ml}$ with this map to obtain $\tilde{p}_\ell(x_d) \colon \Delta^d \to \heqpt{\Bl}$.
		\item The second map $\wqq_\ell$ is the map $\on{Mock}(\xi_\ell, n) \to \on{Mock}(\tilde \xi_\ell, n)$ induced by post-composing the $\xi_\ell$-structures $\hat{f}\colon W \to \Ml$ with $\tilde{p}_\ell\colon  \Ml \to \Bl$.
	\end{enumerate}
\end{definition}

\begin{lemma}
	\label{lem:semi-simplicial-model-main-square}
	The following square of Kan semi-simplicial sets commutes.
	\begin{center}
		\begin{tikzcd}
			\heqpt{\Ml} \arrow[d, "\alpha"] \arrow[rr,"\wpp_\ell"] && \heqpt{\Bl} \arrow[d, "\beta"]\\
			\widehat{\on{Mock}}(\xi_\ell, n) \arrow[rr, "\wqq_\ell"] & & \on{Mock}(\wxi_\ell, n)
		\end{tikzcd}
	\end{center}
\end{lemma}

\begin{proof}
	This follows directly from the definitions of the maps, as follows.  On a $d$-simplex $x_d$ of $\heqpt{\Ml}$ we have 
    $$
    \tilde{q}_{\ell} \circ \alpha (x_d) = (\Delta^d \times M, \tilde{p}_{\ell} \circ \hat{x}_d), \quad \text{and} \quad \beta \circ \tilde{p}_{\ell} (x_d) = (\Delta^d \times M, \widehat{\tilde{p}_{\ell}\circ x_d}).
    $$ 
    These reference maps agree because taking adjoints commutes with postcomposition.
\end{proof}

We base the lower semi-simplicial sets at the point subcomplexes
$$
(\Delta^d \times M,\hatid_M \circ p_M)_{d\geq0} \quad \text{and} \quad
(\Delta^d \times M,\hat{\nu}_\ell \circ p_M)_{d\geq0},
$$
respectively. With these choices, the square is a square of pointed Kan semi-simplicial sets. 

The square of spaces appearing in Theorem A is the commutative square of spaces obtained by geometrically realizing the square in the preceding lemma. In particular, we define the bordism spaces appearing in it as follows.

\begin{definition}\label{def:sixthirteen}
 Let $\Omega^{\infty + n}Q(\xi_\ell) \coloneqq  \|\on{Mock}(\xi_\ell,n)\|$, $\widehat{\Omega}^{\infty + n}Q(\xi_\ell) \coloneqq  \|\widehat{\on{Mock}}(\xi_\ell,n)\|$ and $\Omega^{\infty + n}Q(\tilde\xi_\ell) \coloneqq  \|\on{Mock}(\tilde\xi_\ell,n)\|$. \end{definition}

\begin{remark}\label{rem:sixfourteen}
	Since degree is additive under disjoint union, the set
	$\pi_0(\widehat{\Omega}^{\infty + n}Q(\xi_\ell)) \cong \wOmega_{n}(\Ml)$ of components consisting of elements of degree one   is not a subgroup. Hence $\alpha_0 \colon \pi_0( \heqpt{\Ml}) \to \wOmega_{n}(\Ml)$ is just a map of pointed sets and not a group homomorphism. For $k \geq 1$, the induced map
	$\alpha_k \colon \pi_k( \heqpt{\Ml}) \to \pi_k( \wOmega^{\infty+n} Q(\xi_\ell)) \cong \wOmega_{n+k}(\Ml)$  is a group homomorphism, which can be expressed directly by the formula
	$$\alpha_k([u \colon S^k \to \heqpt{\Ml}]) = [\hat u \colon S^k \times M  \to \Ml],$$
    where $\hat{u}(s,m) = u(s)(m)$. Here $u$ is based  at the element $\hatid_M$, and the stable parallelizability of $S^k$ gives
$\nu_{S^k \times M} \simeq p_M^* \nu_M$, so $\hat{u}$ carries the required $\xi_\ell$-structure. Since $\heqpt{\Ml}$ is a group-like $H$-space, translation identifies the homotopy types of their connected components based at $\alpha(\hat m)$, for $\hat m \in \heqpt{\Ml}$.  A similar discussion holds for the map $\beta_k \colon \pi_k( \heqpt{\Bl}) \to \Omega_{n+k}(\Bl)$, where there is no component restriction.
\end{remark}

 Note that the higher homotopy groups are \emph{independent of the base point}, and hence $\wOmega_{n+k}(\Ml) \cong \Omega_{n+k}(\Ml)$ and $\wOmega_{n+k+1}(\Bl,\Ml) \cong 
\Omega_{n+k+1}(\Bl, \Ml)$, for $k\geq1$.

\subsection{Semi-simplicial models of vertical homotopy fibres and their homotopy groups}
In this section, we describe Kan semi-simplicial set models for the homotopy fibres of the two vertical maps in the square in Lemma \ref{lem:semi-simplicial-model-main-square} and identify their homotopy groups with the bordism groups described in Section \ref{sec:the-bordism-group-htildemk}.

\begin{definition}\label{def:homfibre}
	A Kan semi-simplicial model for the homotopy fibre of $\alpha$ over a base point  $\alpha(\wm) = (M, \wm)$, where $\wm \colon  M \to \Ml \in \heqpt{\Ml}$, is the semi-simplicial set $\htilde{\alpha, \wm}$ which has as its set of $k$-simplices
	$$\htilde{\alpha, \wm}_k = \{(x_k,\hat{y}_{k+1}) \in \heqpt{\Ml}_k \times \widehat{\on{Mock}}(\xi_\ell, n)_{k+1}  \vv \alpha(x_k) = d_{k+1} \hat{y}_{k+1}; (d_0)^{k+1} \hat{y}_{k+1} = \alpha(\wm)\}.$$ 	
	In other words, a $k$-simplex consists of a (face-collared) singular $k$-simplex
	$ x_k\colon  \Delta^k \to \heqpt{\Ml}$ and a $(k+1)$-simplex $\hat y_{k+1} \in \on{Mock}(\xi_\ell, n)$,  such that the $(k+1)$-th face of $\hat{y}_{k+1}$ is $\alpha(x_k)$, and the $(k+1)$-th vertex of
	$\hat y_{k+1}$ is $\alpha(\wm)$.  The face maps are $d_i(x_k, \hat{y}_{k+1})=(d_i x_k, d_i \hat{y}_{k+1})$ for $0 \leq i \leq k$. The point subcomplex (base point) in degree $k$ is $(\wm_k, (\Delta^{k+1} \times M, \wm \circ p_M))$ where $\wm_k$ is the constant $k$-simplex at $\wm$.  
    Similarly we define $\htilde{\beta, \wb}$ for a base point
	$\beta(\wb) = (M, \wb)$, where $\wb \colon  M \to \Bl \in \heqpt{\Bl}$.
\end{definition}

We can now prove the following lemma, which gives a geometric interpretation of the homotopy fibres $\htilde{\alpha}$ and $\htilde{\beta}$ in terms of the bordism groups defined in Section \ref{sec:the-bordism-group-htildemk}.

\begin{lemma}\label{lemma:geometric-homotopy-fiber}
	For each $k \geq 1$, $\wm \in \heqpt{\Ml}$, $\wb = \wpp_\ell(\wm) \in \heqpt{\Bl}$ there are natural isomorphisms
	$$\pi_{k}(\htilde{\alpha, \wm}) \xrightarrow{\cong} \htildek{k}{\Ml, \wm} \text{ and } \pi_{k}(\htilde{\beta, \wb})  \xrightarrow{\cong} \htildek{k}{\Bl, \wb}$$
	given by mapping the element represented by the $k$-simplex $(x_k,\hat y_{k+1})$ of the domain to the element of the codomain represented by
	$\hat y_{k+1}$.
	
	For $k=0$ there are natural bijections of pointed sets $\pi_{0}(\htilde{\alpha, \wm}) \xrightarrow{\cong} \htildek{0}{\Ml,\wm}$  and similarly for $\beta$.  
		\end{lemma}

\begin{proof}  We present the argument for the first isomorphism. The proof for the  second is similar, and will be left to the reader.

    \smallskip
    \noindent \textbf{The map is well-defined.}     
    For the case $k \geq 1$, let $(x_k,\hat{y}_{k+1}) \in  \htilde{\alpha, \wm}_k$ represent an element of $\pi_k(\htilde{\alpha, \wm})$, so that all its faces are the base point. Write $\hat{y}_{k+1}=(W^{n+k+1},\wF)$. Then
    $$
    d_i \hat{y}_{k+1}  = (\Delta^k \times M,\wm \circ p_M) \quad 0 \leq i \leq k, \qquad \text{and} \quad d_{k+1} \hat{y}_{k+1}  = \alpha(x_k)  = (\Delta^k \times M,\hat{x}_k),
    $$
    where $\hat{x}_k$ is the adjoint of $x_k$. Let $D_+^k = d_{k+1} \Delta^{k+1}$ and $D_-^k = \cup_{i=0}^{k} d_i\Delta^{k+1}$. Under the standard identification $\bd \Delta^{k+1} = D_-^k \cup D_+^k \cong S^k$, the face data identify $\bd W \cong S^k \times M$. The restriction of $\wF$ to $D_-^k \times M$ is $\wm \circ p_M$, while its restriction to $D_+^k \times M$ is $\hat{x}_k$. Thus $(W,\wF)$ satisfies the boundary conditions of Definition \ref{def:the-bordism-group-htildemk} and defines an element of $\htildek{k}{\Ml, \wm}$. 
    
    Suppose that two representatives are equivalent. By the description of semi-simplicial homotopy in Section \ref{subsec:semi-simplicial-sets-and-kan-condition}, there is a $(k+1)$-simplex $(x_{k+1},\hat{y}_{k+2}) \in \htilde{\alpha, \wm}_{k+1}$ whose $k$-th and $(k+1)$-st faces are the two representatives and whose earlier faces are the base point. Writing $\hat{y}_{k+2}=(Y,\widehat{G})$, the $k$-th and $(k+1)$-st faces give $W$ and $-W'$. After the standard truncation of their common corner and rounding corners, the remaining faces form $S^k \times M \times I$. On $D_-^k \times M\times I$ the map $\widehat{G}$ is the constant family $\wm$, while on $D_+^k \times M \times I$ it is adjoint to $x_{k+1}$, since $d_{k+2} \hat{y}_{k+2} = \alpha(x_{k+1})$. Hence $(Y,\widehat G)$ gives a bordism in the sense of Definition \ref{def:the-bordism-group-htildemk}, and the assignment depends only on the homotopy class.
    
     For $k=0$, recall that $\pi_0(\htilde{\alpha, \wm})$ is the set of vertices modulo the relation generated by edges. A vertex $(x_0,\hat{y}_1)$ satisfies $d_0 \hat{y}_1= \alpha(\wm)$ and $d_1 \hat{y}_1 = \alpha(x_0)$. 
    Thus, if $\hat{y}_1 = (W^{n+1},\wF)$, the two boundary components of $W$ carry respectively the prescribed structure $\wm$ and the structure determined by $x_0$. Hence $(W,\wF)$ defines an element of $\htildek{0}{\Ml,\wm})$. An edge $(x_1, \hat{y}_2)$ between two vertices gives, after rounding
corners, a bordism between the corresponding geometric cycles. Therefore the assignment is well-defined on $\pi_0$.

\smallskip
	\noindent
   The following surjectivity and injectivity arguments described for $k \geq 1$ also apply when $k=0$.

	\smallskip
	\noindent
	\textbf{The map is surjective:}  Let $(W, \wF)$ be a representative of an element $[(W, \wF)] \in \htildek{k}{\Ml, \wm}$, where $W \subset \bR^{q} \subseteq \bR^\infty$ is a compact $(n+k +1)$-submanifold with boundary. We need to produce a pair $(W^+, \wF^+)$ satisfying the conditions in Definition \ref{def:fiveone} that represents the same element as $(W, \wF)$ in $\htildek{k}{\Ml, \wm}$. To this end, the collaring theorem for topological manifolds \cite[Theorem 2]{Brown:1962} implies that there is a neighbourhood of $\bd W$ in $W$ homeomorphic to $\bd W \times [0,1]$. By using the collar coordinate, a bump function which is zero outside the collar, and the Whitney embedding theorem, we may construct a new, locally flat embedding $W^+ \subset  \Delta^{k+1} \times \bR^{q'}$ (for some $q' \gg 0$) of the manifold $W$ with the property that $\bd W^+ = W^+ \cap (\bd \Delta^{k+1} \times \bR^{q'})$.

    Using the chosen collar of $\partial W$, we first arrange that $W^+$ has the required product form near each face of $\bd \Delta^{k+1}$. We then apply the relative topological transversality theorem
    (see \cite[\S~9.6C]{Freedman:1990}, \cite[Essay~III, Theorem~1.1]{Kirby:1977}) successively to the faces
    $\Delta_i^k \times \mathbb{R}^{q'}$ for $0\leq i\leq k+1$. At each stage, the isotopy is chosen relative to a smaller boundary collar and relative to neighbourhoods of the faces already treated. Since there are only finitely many faces, the resulting embedding is topologically transverse to every $\Delta_i^k \times \mathbb{R}^{q'}$, while the previously arranged face intersections and product-collar structures remain unchanged. 
    
        If we choose $q''$ sufficiently large, the two embeddings $W$ and $W^+$ are isotopic in $\Delta^{k+1} \times \bR^{q''}$   and we can apply the isotopy extension theorem 
	\cite[Corollary 1.4]{Edwards:1971} to conclude that they are ambiently  isotopic in $\Delta^{k+1} \times \bR^{q''}$. This ambient isotopy produces a manifold $Y' \cong W \times [0,1]$ sitting in $\Delta^{k+1} \times \bR^{q''} \times [0,1]$ with boundary piece $W \times 0$ given by the original $W$, boundary piece $W \times 1$ given by the new embedding $W^+$, and remaining boundary $\partial W \times [0,1] \cong S^k \times M \times I$. 
	
	Furthermore, the map $\wF$ in the original pair $(W, \wF)$ extends along this ambient isotopy to a map $\wG'$ on $Y'$, and hence also induces a map $\wF'$ on $W^+$ as well as a homotopy $\hat h'$ on the $\partial W \times [0,1]$-part. Lastly, let $\wF^+$ denote a homotoped version of $\wF'$  (through $\xi_\ell$-structures) such that
	\begin{itemize}
		\item $\partial_i(\wF^+) = \wm \circ p_M|_{\Delta_i^k \times M}$ for all $i \neq k+1$;
		\item $\partial_{k+1}(\wF^+) = \hat{x}_k$ where $\hat x_k$ the adjoint of a face-collared singular simplex $x_k$ homotopic through $\xi_{\ell}$-structured families relative to boundary to $\wF|_{D^k_+\times M}$, where $D^k_+\times M \subset \partial W$ is the ``nontrivial'' boundary part in Definition \ref{def:the-bordism-group-htildemk};
		\item for any $i$ there exists some $\varepsilon > 0$ such that $\partial_i^{\varepsilon}(\wF^+)= \partial_i(\wF^+) \circ \pi_i(\varepsilon)$.
	\end{itemize}
	Let $h''$ denote the homotopy between $\wF'$ and $\wF^+$.
	By gluing
	 $Y' \cong W \times [0,1] \subseteq \Delta^{k+1} \times \bR^{q''} \times [0,1]$
and the embedded cylinder  $Y'' = W^+ \times [1,2] \subseteq
  \Delta^{k+1} \times \bR^{q''} \times [0,2]$ along $W^+ \times 1$, we obtain a manifold $Y$. By extending the function $\wG'$ on $Y'$ with $h''$ on $Y''$, we obtain a function $\wG$ on $Y$. By restricting $\wG$ to the $\partial W \times [0,1]$-part of $Y$, we obtain a homotopy $\hat h$. By construction the pair $(Y, \wG)$ now bounds the element
	$$
	[(W \cup_{\partial W= S^k \times M \times \{0\}} (S^k \times M \times I) \cup_{\partial W' = S^k \times M \times \{1\}}
	   -W^+, \wF \cup \hat{h} \cup \wF^+)] \in \Omega_{n+k+1}(\Ml ),
	$$
	proving that $(W, \wF)$ and $(W^+, \wF^+)$ represent the same element in $\htildek{k}{\Ml, \wm}$. Moreover, $(W^+, \wF^+)$ defines a $k$-simplex $(x_k, y_{k+1})$ in $\htilde{\alpha,\wm}$; $x_k$ is the face-collared singular $k$-simplex appearing in the construction of $\wF^+$ and $y_{k+1}$ is $(k+1)$-simplex of $\on{Mock}(\xi_\ell, n)$ defined by $(W^+, \wF^+)$. This simplex $(x_k, y_{k+1})$ represents an element in $\pi_{k}(\htilde{\alpha, \wm})$ that is mapped to $[(W, \wF)] = [(W^+, \wF^+)] \in \htildek{k}{\Ml, \wm}$, proving surjectivity.
	
	\smallskip
	\noindent
	\textbf{The map is injective:} If $(x_k, \hat y_{k+1})$ and $(x_k', \hat y_{k+1}')$ are $k$-simplices representing elements in $\pi_k(\htilde{\alpha,\wm})$
	and if $\hat y_{k+1}$ and $\hat y_{k+1}'$ are bordant in the sense of Definition \ref{def:the-bordism-group-htildemk} via $(\hat h, Y, \wG)$, then $(\hat h, Y, \wG)$ gives rise to a $(k+1)$-simplex $(x_{k+1},\hat y_{k+2})$. We have that
	$$\bd Y= W \cup (S^k \times M \times I) \cup -W' = W \cup ((D^k_{-} \cup D^k_{+}) \times M \times I) \cup -W'.$$
	under the following identifications,
	\begin{align*}
		(W, \wG|_{W}) &\cong d_k (x_{k+1}, \hat{y}_{k+2}) =(x_{k}, \hat{y}_{k+1}) \\
		(W', \wG|_{W'}) &\cong 
		d_{k+1} (x_{k+1}, \hat{y}_{k+2})   = (x_k', \hat{y}_{k+1}') \\
		(S^k \times M \times I, \hat{h}) &\cong d_{k+2}(\hat{y}_{k+2}) \bigcup \, \cup_{i=0}^{k-1} d_i(\hat{y}_{k+2})= \hat{x}_{k+1} \bigcup \, \cup_{i=0}^{k-1} d_i(\hat{y}_{k+2})
	\end{align*}
	where the restriction $h|_{D^k_{+} \times M \times I}$ is adjoint to $x_{k+1}$ and $h|_{D^k_{-} \times M \times I}$ corresponds to $\cup_{i=0}^{k-1} d_i(\hat{y}_{k+2})$. Following similar arguments as in the surjectivity part above, the manifold $Y$ together with $\wG$ can be embedded into 
	 $\Delta^{k+2} \times \bR^q$ for large $q$ satisfying boundary and collar conditions and hence, defines $\hat y_{k+2}$,
	proving that $(x_k, \hat y_{k+1})$ and $(x_k', \hat y_{k+1}')$ are equivalent in $\pi_k(\htilde{\alpha,\wm})$.
\end{proof}

\begin{remark} We note that the identifications provided by Lemma \ref{lemma:geometric-homotopy-fiber} for $k \geq 1$ are compatible with natural group operations on the geometric models $ \htildeMk{k} $ and $ \htildeBk{k}$ that we described in Section \ref{subsec:groups} and Section \ref{subsec:groupstwo}. 
\end{remark}

\section{The proof of Theorem A}\label{sec:five_a}
 Recall that the square of spaces appearing in Theorem A is the commutative square of spaces obtained by geometrically realizing the square of semi-simplicial sets in Lemma \ref{lem:semi-simplicial-model-main-square}. 
To prove Theorem A, we will check that the condition for being homotopy $(k+1)$-cartesian   in Lemma \ref{lem:p-cartesian-square} holds on the semi-simplicial level. 

More precisely, we will show that map $\htilde{\alpha, \wm} \to \htilde{\beta, \wpp_\ell(\wm)}$ between the Kan semi-simplicial set models $\htilde{\alpha, \wm}$ and $\htilde{\beta, \wpp_\ell(\wm)}$ of the vertical homotopy fibres described in Definition \ref{def:homfibre} is $(k+1)$-connected for all choices of base points $\wm \in \heqpt{\Ml}$. This will rely on the description of the homotopy groups of these homotopy fibres established in Lemma \ref{lemma:geometric-homotopy-fiber}. 
Note that by definition,  the image $\alpha(\wm)$ of $\wm \in \heqpt{\Ml}$ has degree one and hence is  contained in a connected component of $\widehat{\on{Mock}}(\xi_\ell,n)$. The following result shows that it is enough to verify the condition of Lemma \ref{lem:p-cartesian-square} over the points $\alpha(\wm)$. This uses  a generalization of \cite[Lemma 2.8]{Hambleton:2004}, which holds in both the smooth and topological settings.

\begin{lemma} For $\ell \geq [(n+1)/2]$, an element $x = [N, \hat g] \in \wOmega_n(\Ml)$ is in the image of $\alpha_0 = \pi_0(\alpha)$ if and only if the element 
$$\wqq_\ell(x) = [N, \wpp_\ell \circ \hat g]  \in \Omega_n(\Bl)$$ 
is in the image of $\beta_0= \pi_0(\beta)$.
\end{lemma}
\begin{proof} After applying $\pi_0$, the commutativity of the square in Lemma \ref{lem:semi-simplicial-model-main-square} gives the forward implication. Conversely, if there exists an element $y=[M, \hat b] \in \pi_0(\heqpt{\Bl})$ such that
$\beta_0(y) = [N, \wpp_\ell \circ \hat g] $, then there is a $\tilde{\xi}_{\ell}$-bordism $(W,\wG)$ from $(N,\tilde{p}_{\ell} \circ \hat{g})$ to $(M,\hat{b})$. By performing surgeries in the interior of $W$, we can ensure that $\wG \colon W \to \Bl$ is $[n/2]$-connected. The obstructions for lifting $\wG$ relative to $(N,\hat{g})$ lie in $H^{i+1}(W,N; \pi_i(\Fl))$ where $\Fl$ is the fiber of $\tilde{p}_{\ell}$. All obstructions vanish as in the ``surjectivity'' part of the proof for Proposition \ref{proposition:k-hml-iso-k-hbl} below. Thus, $\wG$ lifts to $\widehat{H} \colon W \to \Ml$ extending $\hat{g}$. Denote the restriction $\widehat{H}|_{M}$ by $\hat{f}$ and we have $\tilde{p}_{\ell}\circ \hat{f} = \hat{b}$. Then the definition of $\heqpt{\Bl}$ implies that $f \coloneqq j \circ \hat{f} \colon M \to M$ is $\ell$-equivalence. Moreover, $f$ is of degree one since $\on{deg}(j\circ \hat{g})=1$. Then using the same Poincar\'e duality argument, we conclude that $f$ is homotopy equivalence and so we obtain $\hat{f} \in \heqpt{\Ml}$. The lifted bordism $(W,\widehat{H})$ gives the equality $[(N,\hat{g})]=[(M,\hat{f})]= \alpha_0([\hat{f}])$.
\end{proof}

 We now proceed with the proof of  Theorem A.  The geometric realizations, $\|\htilde{\alpha, \wm}\|$ and $\|\htilde{\beta, \wpp_\ell(\wm)}\|$, are homotopy equivalent to the homotopy fibres of the maps
$\alpha \colon \|\heqpt{\Ml}\| \to \| \widehat{\on{Mock}}(\xi_\ell, n)\|$ and $\beta\colon  \|\heqpt{\Bl}\| \to  \| \on{Mock}(\tilde\xi_\ell, n)\|$,
at any base points $\alpha(\wm)$ and $\beta(\wpp_\ell(\wm))$, respectively (see \cite[p.~16]{Nicas:1982}). And since the semi-simplicial homotopy groups agree with those of the geometric realization (see Section \ref{subsec:semi-simplicial-sets-and-kan-condition}), Theorem A then follows  by showing that map $\htilde{\alpha, \wm} \to \htilde{\beta, \wpp_\ell(\wm)}$  defined in semi-simplicial degree $r$ by $\hat{p}_{\ell} (x_r, \hat{y}_{r+1}) = ( \tilde{p}_{\ell} \circ x_r, \, \tilde{q}_{\ell}(\hat{y}_{r+1}))$
is   $(k+1)$-connected. For this,  we give explicit calculations based on Lemma \ref{lemma:geometric-homotopy-fiber} and the bordism description of the homotopy groups of these two homotopy fibres given in Definition \ref{def:the-bordism-group-htildemk} and Definition \ref{def:the-bordism-group-htildebk}. 

\begin{proposition}\label{proposition:k-hml-iso-k-hbl}
	For all $k\geq 0$, $\ell \geq  \lmin\coloneqq [(n+k+1)/2]$ and $\wm \in \heqpt{\Ml}$, the map $\hatpp_\ell\colon \htilde{\alpha, \wm} \to \htilde{\beta, \wpp_\ell(\wm)}$ induces an isomorphism $\htildek{k}{\Ml, \wm} \cong \htildek{k}{\Bl, \wpp_\ell(\wm)}$ and an epimorphism
	$\htildek{k+1}{\Ml, \wm} \twoheadrightarrow \htildek{k+1}{\Bl, \wpp_\ell(\wm)}$. In particular, $\hatpp_\ell$ is   $(k+1)$-connected whenever $k \geq 0$ and $\ell \geq  \lmin\coloneqq [(n+k+1)/2]$. 
\end{proposition}
\begin{proof}
	The geometric models in  Lemma \ref{lemma:geometric-homotopy-fiber}  identify $\pi_k(\htilde{\alpha, \wm}) \xrightarrow{\ (\hatpp_\ell)_* \ } \pi_k(\htilde{\beta, \wpp_\ell(\wm)})$ with the map $\psi_k \colon \htildek{k}{\Ml, \wm} \to \htildek{k}{\Bl, \wpp_\ell(\wm)}$ given by sending an element $[(W,\wF)]$ to $[(W,\tilde{p}_\ell \circ \wF)]$ where $\tilde{p}_\ell \colon \Ml \to \Bl$.  We will show that $\psi_k$ is an isomorphism  and that $\psi_{k+1}$ is an epimorphism. Since the base points	$\wm$ and $\wpp_\ell(\wm)$ are fixed, we will mostly suppress them in the notation.

	\nr{Surjectivity of $\psi_k$:} \\
	Given $(Z,\wG) \in \htildek{k}{\Bl, \wpp_\ell(\wm)}$, our goal is to find $[(Z,\wF)] \in \htildek{k}{\Ml, \wm}$ such that
	$\tilde{p}_\ell \circ \wF = \wG$.
	 Since $\htildek{k}{\Bl, \wpp_\ell(\wm)}$ consists of bordism classes with maps to $\Bl$ which extend $\wpp_\ell(\wm) \circ p_M$ on $D^k_- \times M \subset S^k \times M$, we need to lift $\wG\colon Z \to \Bl$ to a map $\wF \colon Z \to \Ml$ such that $\tilde{p}_{\ell} \circ \wF = \wG$ and $\wF|_{D^k_- \times M}= \wm \circ p_M$.
	
	\nr{Step 1: Arranging the connectivity of $\wG$.} We first adjust $\wG\colon Z \to \Bl$ so that it induces isomorphism on $\pi_i$ for $i \leq \lmin -1$. In other words, we will replace $(Z,\wG)$ by a bordant representative relative to the boundary, such that $\wG$ is $\ell_{\min}$-connected.

    To see this consider the inclusion $\bd_-Z= D^k_- \times M \hookrightarrow Z$ and recall that $j \circ \wpp_\ell(\wm) \colon M \to P_{\ell}(M)$ is an $(\ell+1)$-equivalence. Since $\wG|_{\bd_-Z} = \wpp_\ell(\wm) \circ p_M$ induces isomorphism on $\pi_i$ for $i \leq \ell$, it follows that $\wG$ induces a surjection on $\pi_i$ for $i \leq \ell$.   Note that $\Bl$ can be modeled with a finite $\left\lfloor\frac{n+k}{2}\right\rfloor$-skeleton using the 
      $\ell$-connected map $\hat{\nu}_{\ell} \colon M \to \Bl$.
    
    By performing surgery below the middle dimension (see \cite[\S 3]{Kreck:1999} and \cite[\S~11.3]{Freedman:1990}, \cite[Essay~III, Appendix~C]{Kirby:1977} for the topological surgery version), 
     we can kill the kernels of these maps $\wG_{\ast} \colon \pi_i(Z) \to \pi_i(\Bl)$ for $i \leq \lmin -1$, making $\wG$ an $\lmin$-connected map. In particular, $\wG$ induces a surjection on $\pi_i$ for $i = \lmin$.
	
	\nr{Step 2: Obstructions for lifting $\wG$ to $\Ml$.} The obstructions to lifting $\wG$ to a map $\wF \colon Z \to \Ml$ relative to $\bd_-Z$ lie in the cohomology groups 
	 $H^{i+1}_t(Z,\bd_-Z; \pi_i(\Fl))$, where $\Fl$ is the fibre of $\Ml \to \Bl$ (see Lemma \ref{lem:normalk-type}). Since $\Fl$ is $\ell$-connected, $\pi_i(\Fl)=0$ for $i \leq \ell$, and there are no  obstructions in those degrees. By Lefschetz duality, the obstruction groups are isomorphic to the homology groups (see Remark \ref{rem:fourfive}): $$H^t_{n+k+1-(i+1)}(Z,\bd_+Z; \pi_i(\Fl))$$
	We need to show that these groups vanish for $i > \ell$. 
		
	\nr{Step 3: Computing the homology groups.} Consider the inclusion $\bd_+Z= D^k_+ \times M \hookrightarrow Z$. 
    Since $\wG$ is $\lmin$-connected and $\wG|_{\bd_+ Z}$ induces an isomorphism on $\pi_i$ for $i \leq \lmin$, the inclusion $\bd_+ Z \hookrightarrow Z$ is $(\lmin-1)$-connected. The relative Hurewicz theorem with local coefficients therefore gives $H^t_j(Z,\bd_+ Z; \pi_i(\Fl))=0$ for $j\leq \lmin-1$. 
     For $i > \ell$, the relevant degrees of the homology groups are of the form $n+k+1-(i+1)=n+k-i$. Since $i > \ell \geq \lmin$, we have:
	$$
	n+k-i < n+k -\lmin= \left\lceil\frac{n+k-1}{2}\right\rceil.
	$$
	This degree is less than or equal to $\lmin -1$, implying that $H^t_{n+k-i}(Z,\bd_+Z;\pi_i(\Fl))=0$ for $i > \ell$. Since all obstructions vanish, there exists a lift $\wF \colon Z \to \Ml$ such that $\tilde{p}_\ell \circ \wF = \wG$ and
	  $\wF|_{\bd_-Z}=\wm \circ p_M$. Therefore, $[(Z,\wF)] \in \htildeMk{k}$ and $\psi([(Z,\wF)])=[(Z,\wG)]$, proving surjectivity.
	
	\nr{Surjectivity of $\psi_{k+1}$:} \\
	In order to check that $\htildeMk{k+1} \to \htildeBk{k+1}$  is surjective as well,
 we need to show that $n+k+1 -i \leq \lmin$.
 To achieve this, we need to improve the connectivity of $\wG$ for a representative $\wG \colon Z^{n+k+2} \to \Bl$ in $\htildeBk{k+1}$, which is surjective on homotopy groups up to degree $\ell$. Note that in the case $n+k+2$ is even, we can perform surgery to make $\wG$ an $(\lmin+1)$-connected map. If $n+k+2$ is odd, $\wG$ can only be made  $\lmin$-connected. The obstruction groups still vanish for $i \leq \ell$. For $i > \ell$,
we have  $i >\ell \geq \lmin$ and  $
	n+k+1 -i  < n+k+1 - \lmin$. If $n+k+2$ is odd, we get $ n+k+1 - \lmin = \lmin$, and
hence the obstructions also vanish in this case.
If $n+k+2$ is even, we get $ n+k+1 - \lmin = \lmin+1$, and again the obstructions vanish. Hence in both cases we have  the claimed  surjection.

\nr{Additional details for surjectivity:}\\
We will focus on the surjectivity of $\psi_k$ here, the case $\psi_{k+1}$ is similar. Above, we constructed a lifting $\wF \colon Z \to \Ml$ such that $\tilde{p}_{\ell} \circ \wF = \wG$ and $\wF|_{D^k_- \times M}= \wm \circ p_M$. To check that the pair $(Z, \wF)$ actually defines an  element in $\htildeMk{k}$, we note that
for each $x_+ \in D^k_+$, the restriction $j \circ \wG|_{x_+ \times M}$ is an $(\ell+1)$-equivalence and $\ell \geq [(n+k+1)/2]$. After the lifting, the restriction
$j \circ \wF|_{x_+ \times M}$ is a based homotopy self-equivalence: If $k \geq 1$, the fact that $S^k$ is path-connected implies that $j \circ \wF|_{x_+ \times M}$ is homotopic to the base point $\wm$, which is a homotopy self-equivalence.

If $k=0$, let $f_+ =j \circ\wF|_{\bd_+ Z}\colon M \to M$. The identity $c_{\ell} \circ f_+ = j \circ \wG|_{\bd_+ Z}$ shows that $f_+$ is an $\ell$-equivalence. Moreover, $j \circ \wF \colon Z \to M$ is a bordism over $M$ between $j \circ \wm$ and $f_+$, so $f_+$ has degree one. Since $\ell \geq \left\lfloor\frac{n+1}{2}\right\rfloor$ Poincar\'e duality with local coefficients implies that $f_+$ is a homology equivalence, and hence a homotopy equivalence by Whitehead's theorem (compare the argument in  \cite[Lemma 2.6]{Hambleton:2004}).

	\nr{Injectivity of $\psi_k$:} \\
	Suppose $\psi_k([(W,\wF)])=\psi_k([(W',\wF')])$,  meaning that $[(W,\tilde{p}_{\ell} \circ \wF)]=[(W', \tilde{p}_{\ell}\circ \wF')]$ in $\htildeBk{k}$. Then there exists a base point preserving homotopy $\hat{h} \colon S^k \times M \times [0,1] \to \Bl$ between $(\wpp_\ell \circ \wF)|_{\bd W}$ and $(\wpp_\ell \circ \wF')|_{\bd W'}$ satisfying $\hat{h}|_{D^k_- \times M \times \{t\}}= \wpp_\ell(\wm) \circ p_M$ for all $t \in I$, and $\hat{h}|_{x_+ \times M \times t} \in \heqpt{\Bl}$ for each $x_+ \in D^k_+$ and $t \in [0,1]$, such that the manifold pair
	$$
	(W \cup_{\bd W} (S^k \times M \times [0,1]) \cup_{\bd W'} - W' , (\tilde{p}_\ell \circ \wF) \cup \hat{h} \cup(\tilde{p}_\ell \circ \wF'))
	$$
		represents zero in $\Omega_{n+k+1}(\Bl, \wxi_{\ell})$, i.e. there exists an $(n+k+2)$-manifold $T$ bounding this manifold and a map $\wG \colon T \to \Bl$ extending the described map above. Our goal is to lift $\wG$ relative to $W \cup (D^k_- \times M \times [0,1]) \cup -W'$  to a map $\wH \colon T \to \Ml$.
	
	\nr{Step 1: Arranging the connectivity of $\wG$.} Our first step is to reduce to the case that $\wG \colon T \to \Bl$ induces isomorphisms on $\pi_i$ for $i \leq \left\lfloor\frac{n+k}{2}\right\rfloor$. 
    By surgery below the middle dimension in the interior of $T$, relative to $\bd T$, we may replace $(T,\wG)$ by a bordant representative such that $\wG$ induces isomorphisms on $\pi_i$ for $i \leq \left\lfloor\frac{n+k}{2}\right\rfloor$.

    Since $D^k_+ \times I$ is contractible, the map $\hat{h}|_{D^k_+ \times M \times I}$ is homotopic to any one of its slices. Each slice belongs to $\heqpt{\Bl}$, so it induces isomorphisms on $\pi_i$ for $i \leq \ell$.
    
    Therefore, $\wG$ induces surjection on $\pi_i$ for $i \leq \ell$. As before, we perform surgery below the middle dimension on $T$ to ensure that $\wG$ induces isomorphisms on $\pi_i$ for $i \leq \left\lfloor\frac{n+k}{2}\right\rfloor$. Note that $\left\lfloor\frac{n+k}{2}\right\rfloor= \lmin$ when $n+k$ is even and $\left\lfloor\frac{n+k}{2}\right\rfloor= \lmin-1$ when $n+k$ is odd. 	
		
	\nr{Step 2: Obstructions for lifting $\wG$ to $\Ml$.} The obstructions to lifting $\wG$ to $\wH \colon T \to \Ml$ relative to $K=W \cup (D^k_- \times M \times I) \cup - W'$ lie in the cohomology groups: 
	$$
	H^{i+1}_t(T, K; \pi_i(\Fl)) =
	\begin{cases}
		0 \ & \text{ if } i \leq \ell, \\
		H^t_{n+k+2 - (i+1)}(T, D^k_+ \times M \times [0,1]; \pi_i(\Fl)) & \text{ if } i \geq \ell+1.
	\end{cases}
	$$
	Since $\Fl$ is $\ell$-connected, the obstruction groups vanish for $i \leq \ell$. For $i \geq \ell+1$, we use Lefschetz duality and get the corresponding relative homology groups.
	
	\nr{Step 3: Computing the relative homology groups.} Since $\pi_i(\wG)$ is an isomorphism for $i \leq \left\lfloor\frac{n+k}{2}\right\rfloor$, it follows that the inclusion $D^k_+ \times M \times I \hookrightarrow T$ induces isomorphisms on $\pi_i$ and hence on $H_i$ for $i \leq \left\lfloor\frac{n+k}{2}\right\rfloor$. Therefore, we have $H^t_{j}(T,D^k_+ \times M \times I;\pi_i(\Fl))=0$ for $j \leq \left\lfloor\frac{n+k}{2}\right\rfloor$. Since $i > \ell \geq \lmin$, we have
	$$
	n+k+1 -i < n+k+1 -\left\lfloor\frac{n+k+1}{2}\right\rfloor = \left\lceil\frac{n+k+1}{2}\right\rceil.
	$$
	This degree is less than or equal to $\left\lfloor\frac{n+k}{2}\right\rfloor$, implying $H^t_{n+k+1-i}(T,D^k_+\times M \times I, \pi_i(\Fl))=0$ for $i > \ell$. Since all obstruction groups vanish, there exists a lift $\wH\colon T \to \Ml$ extending $\wF$, $\wm \circ p_M$, and $\wF'$. Thus, $(W,\wF)$ and $(W',\wF')$ are bordant in $\htildeMk{k}$, showing that $\psi_k$  is injective.

    To see that $(T, \wH)$ satisfies the conditions in
      Definition \ref{def:the-bordism-group-htildemk}, note that connectivity of parameter space $D^k_+ \times I$ implies that every slice of $\wH|_{D^k_+ \times M \times I}$ is homotopic, through based maps, to one of the endpoint slices $\wF$ or $\wF'$, which is a homotopy self-equivalence. Therefore, every slice belongs to $\heqpt{\Ml}$ and normal bundle condition follows using the equalities $\xi_{\ell} \circ \wH =  \tilde{\xi}_{\ell} \circ \tilde{p}_{\ell} \circ \wH = \tilde{\xi}_{\ell} \circ \wG = \nu_T$.
   \end{proof}

\begin{remark}\label{rem:smooth_psi} In preparation for the proof of  Theorem B, we note that the bordism groups $ \htildek{k}{\Ml}$ and $\htildek{k}{\Bl}$ are defined for both the smooth and topological versions of the spaces $\Ml$ and $\Bl$
 (see Definitions  \ref{def:the-bordism-group-htildemk} and \ref{def:the-bordism-group-htildebk}). Moreover, the map $\psi_k \colon \htildek{k}{\Ml} \to \htildek{k}{\Bl}$ given by sending an element $[(W,\wF)]$ to $[(W,\tilde{p}_\ell \circ \wF)]$ applies in both settings.

The proof of Proposition \ref{proposition:k-hml-iso-k-hbl} applies equally in the smooth category: one performs the surgeries in the interior smoothly and uses the same obstruction-theoretic lifting argument with $BO$ in place of $BTOP$. Consequently, $\psi_k$ is an isomorphism and $\psi_{k+1}$ is surjective for the smooth fibre-bordism groups as well.
\end{remark}

\section{The proof of Theorem B}\label{subsec:corb}
  For topological manifolds and bordisms,  Theorem A  provides a (sign-commutative) braid of interlocking exact sequences in a range of dimensions depending on the parameters $k \geq 0$ and $\ell \geq \lmin$ as displayed in 
  Theorem B  (see Lemma \ref{lem:twoseven}). We will show that an analogous braid exists for smooth manifolds and smooth bordisms. Some extra work is required in this setting due to the current lack of a suitable space-level model for bordism (see Remark \ref{rem:smooth_model}).
  
  \subsection{The setting}
 For $M$ a smooth, closed, connected $n$-manifold, with $n\geq 4$  and $\ell \geq 2$, we have defined the spaces $\heqpt{\Ml}$ and $\heqpt{\Bl}$ to study the based homotopy self-equivalences of $M$ over the Postnikov $\ell$-sections of the stable normal bundle $\nu_M\colon M \to BO$. As above, we assume that $M$ is locally oriented. 
 
 Let $\MlT$ or $\BlT$ denote the homotopy pullbacks of $BTOP$ and $M$ or $P_\ell(M)$ over $P_\ell(BTOP)$, as given in Definition \ref{def:hpull} by setting $Y = BTOP$ instead of $Y= BO$. The natural map $i\colon BO \to BTOP$  induces  maps of spaces  $i_{\Ml} \colon \Ml \to \MlT$ and 
     $i_{\Bl} \colon \Bl \to \BlT$, and induces a comparison  map between their related homotopy pullback squares.  In the explicit path-space model from Definition \ref{def:hpull}, the first map is defined as $i_{\Ml}(y,x,\theta)= (i(y),x, P_{\ell}(i)\circ \theta)$ and the analogous formula defines $i_{\Bl}$. Note that the comparison maps preserve the orientation character and the associated local coefficient systems.

 We can combine both the smooth and topological spaces into a diagram  (see \eqref{diag:twotwo}):
\eqncount
\begin{equation}\label{diag:Ml_smooth}
	\vcenter{
		\xymatrix@C+5pt@R+5pt{
			\Ml \ar[r]^{j}\ar[d]^{\wpp_\ell} 
			& M\ar[d]^{c_\ell}& \MlT \ar[l]_j\ar[d]_{\wpp_\ell}\\
			\Bl \ar[r]^(0.5){j} & \post{M}{\ell}& \BlT\ar[l]_j
		}
	}
\end{equation}
where the left and right squares are homotopy pullbacks. 

\begin{lemma}\label{lem:cocartsquare} If $\ell \geq [(n+k+1)/2]$,  the left and right squares in \textup{\eqref{diag:Ml_smooth}} are at least  $(n+k+3)$-cocartesian.
\end{lemma}
\begin{proof} The maps labelled $\wpp$, $c_\ell$ and $j$ in these diagrams are each $(\ell +1)$-connected, by Lemma \ref{lem:normalk-type} and the construction of Diagram \eqref{diag:twotwo}. By \cite[Theorem 4.2.2]{MunsonVolic:2015}, it follows that each square is $(2\ell +3)$-cocartesian, and  $2\ell +3 \geq n+k+3$.
\end{proof}

\begin{remark}\label{rem:ahss}  In the next section, we will use the Atiyah-Hirzebruch spectral sequence (AHSS) to  compute the Lashof bordism groups defined in Section \ref{subsec:lashoftwo}. The starting point is the isomorphism $\Omega_r(\Ml, \xi_\ell) \cong \pi^{S}_r(T(\xi_\ell))$ given by transversality, where $ \pi^{S}_*$ refers to the generalized homology theory ``stable homotopy'',  and $T(\xi_\ell)$ denotes the Thom space of the bundle $\xi_\ell$.

A similar expression holds for the bordism groups  $\Omega_r(\Bl, \wxi_\ell)$ and their analogues for $TOP$ bordism.  For simplicity,  we  use the same notation $\xi_\ell$ and $\wxi_\ell$ for the bundles over $BO$ or $BTOP$. Since the bundle $\xi_\ell = \wpp_\ell^*(\wxi_\ell)$ is pulled back from $\Bl$, we have an isomorphism $$\Omega_{n+k+1}(\Bl, \Ml) \cong \pi^S_{n+k+1}(T(\wxi_\ell), T(\xi_\ell)).$$
 The AHSS for a generalized homology theory $h_*$ computes $h_*(X,A)$ for a CW pair $(X,A)$ by starting with $E^2_{p,q} = H_p(X,A; h_q(pt))$ and converging to $h_{p+q}(X,A)$. The AHSS is natural with respect to maps of CW pairs (see \cite[Theorem 15.7]{Switzer:2002} for the construction of the spectral sequence and further details). 
 \end{remark}

\subsection{Input for the braid}\label{subsec:input_braid}
In the topological category, the braid of Theorem B follows from Theorem A by applying Lemma \ref{lem:twoseven} to the homotopy $(k+1)$-cartesian square. In the smooth category, we have not
constructed an analogous square of spaces.
  The braid in Theorem B in the smooth case follows formally from comparing  two ladder diagrams of long exact sequences in a commutative diagram, in which every third term is mapped isomorphically (as in the construction of the Mayer-Vietoris sequence).

  More precisely, we will proceed by showing that the maps labelled $\psi_k$ and $b_{k+1}$ in the following diagram are isomorphisms in our given range of dimensions  $\ell \geq \lmin = [(n+k+1)/2]$.

  \begin{definition} 
For each $k\geq0$, we define a map 
$$b_{k+1} \colon \pi_{k+1}( \heqpt{\Bl}, \heqpt{\Ml}) \to \wOmega_{n+k+1} (\Bl,\Ml)$$
as follows. Represent a relative class by a $D^{k+1}$-family in $\heqpt{\Bl}$ with a prescribed
$S^k$-family lift to $\heqpt{\Ml}$. If $\hat{u} \colon D^{k+1} \times M \to \Bl$, and $\bd \hat{u} \colon S^k \times M \to \Ml $ are the adjoint maps, we set $b_{k+1}[(u, \bd u)] = [(D^{k+1} \times M,\hat{u},\bd \hat{u})]$. This is well-defined under relative homotopy and is a homomorphism for $k\geq1$; for $k=0$ it is a map of pointed sets. All homotopy and relative homotopy groups below are taken at the compatible basepoints $\hatid_M$ and $\hat{\nu}_\ell=\tilde{p}_\ell(\hatid_M)$.
  \end{definition}
 \begin{proposition}\label{prop:two_ladders}
  	There exist two ladder diagrams of long exact sequences, which fit into a commutative diagram of the following form:
$$\xymatrix@R-9pt{&&\vdots\ar[d]&\vdots \ar[d]&\\
& & \htildeMk{k} \ar[r]^{\psi_k} \ar[d]& \htildeBk{k} \ar[d]& \\
\dots \ar[r]&\pi_{k+1}(\heqpt{\Bl}, \heqpt{\Ml}) \ar[r]^(0.6){\bd_{k+1}}\ar[d]^{b_{k+1}}&\pi_k(\heqpt{\Ml}) \ar[d]^{\alpha_k}\ar[r] & \pi_k(\heqpt{\Bl})\ar[r]\ar[d]^{\beta_k}& \dots\\
\dots \ar[r]& \wOmega_{n+k+1}(\Bl, \Ml) \ar[r]^(0.6){\bd_{k+1}}\ar@{-->}[ur]^{\gamma_k}& \wOmega_{n+k}(\Ml) \ar[r] \ar[d]& \Omega_{n+k}(\Bl) \ar[d] \ar[r]& \dots\\
&&\vdots&\vdots&
}$$ 
defined for all $k \geq 0$ and $n \geq 4$.  When $k=0$, the terms involving $\pi_0$ together with relative $\pi_1$, and the corresponding squares are interpreted in the category of pointed sets. For $k \geq 1$, all displayed terms are groups and all maps are homomorphisms. 

\end{proposition} 
\begin{proof} The two horizontal exact sequences are the maps induced by  the map $\Ml \to \Bl$ on the spaces of homotopy self-equivalences, or on the Lashof bordism groups, respectively. The two vertical exact sequences will be established below in Lemmas \ref{lem:eightone}  and \ref {lem:eighttwo}. The dashed map $\gamma_k$ is defined below under the hypothesis of the following proposition. \end{proof}

To complete the proof of Theorem B in the smooth case, we just need to verify the following proposition, and then we can argue exactly as in the proof of Lemma \ref{lem:twoseven}.

\begin{proposition}
	Suppose that  $k\geq 0$, $n \geq 4$, and $\ell \geq  \lmin\coloneqq [(n+k+1)/2]$. Then
	\begin{itemize}
		\item the map $\psi_k$ is an isomorphism and $\psi_{k+1}$ is surjective.
		\item the map $b_{k+1}$ is an isomorphism. 	\end{itemize}
\end{proposition}

\begin{proof}
The argument for the maps $\psi_k$ is completely analogous to the proof of Proposition \ref{proposition:k-hml-iso-k-hbl} (see Remark \ref{rem:smooth_psi}). Once this result is established, the maps $\gamma_k$ needed for the braid are defined by 
$\gamma_k \coloneqq \bd_{k+1}\circ (b_{k+1})^{-1}$  (see the proof of Lemma \ref{lem:twoseven} for the topological case).

The argument for the  maps $b_{k+1}$ is a reduction to topological case:   in Lemma \ref{lem:eight_iso_one}  and Lemma \ref{lem:eight_iso_two}, we will show an interesting result, namely that 
the relative groups involved,  $\pi_{k+1}(\heqpt{\Bl}, \heqpt{\Ml})$ and $\wOmega_{n+k+1}(\Bl, \Ml)$, are independent of the change of bundle data. 
More precisely, we will establish a commutative diagram
\eqncount
\begin{equation}\label{eq:b-iso}
\vcenter{\xymatrix{
\pi_{k+1}(\heqpt{\Bl},\heqpt{\Ml}) \ar[r]^{b_{k+1}} \ar[d]^{i_*}_{\cong} & \wOmega_{n+k+1}(\Bl,\Ml) \ar[d]_{i_*}^{\cong}  \\ 
\pi_{k+1}(\heqpt{\BlT},\heqpt{\MlT}) \ar[r]^{b_{k+1}^{TOP}} & \wOmega_{n+k+1}(\BlT,\MlT) 
}}
\end{equation}
showing that $b_{k+1}= (i_*)^{-1} \circ b_{k+1}^{TOP} \circ i_*$. 
Since we have already established in Theorem A  that $b_{k+1}^{TOP} $ is an isomorphism, 
 the result follows. 
\end{proof}

\subsection{Comparison of the smooth and topological relative groups} 
First, for the range of the map $b_{k+1}$, we compare the smooth and topological relative bordism groups.

\begin{lemma}\label{lem:eight_iso_one}  For $k \geq 0$, $n \geq 4$ and $\ell \geq [(n+k+1)/2]$, there is an isomorphism
$$i_*\colon \Omega_{n+k+1}(\Bl, \Ml) \to \Omega_{n+k+1}(\BlT, \MlT)$$
induced by the map of pairs  $(i_{\Bl},i_{\Ml})\colon (\Bl,\Ml) \to (\BlT, \MlT))$. 
\end{lemma}

\begin{proof}
By diagram \eqref{diag:Ml_smooth}, it follows that the homotopy fibre of the map $\Ml \to \Bl$ is weakly equivalent to the homotopy fibre of the map $M \to P_\ell(M)$, from the left-hand square. Similarly, 
the homotopy fibre of the map $\MlT \to \BlT$ is also weakly equivalent to the homotopy fibre of the map $M \to P_\ell(M)$, from the right-hand square.
Therefore  the map  $i\colon BO \to BTOP$   provides a map of pairs 
$(\Bl, \Ml)  \to (\BlT, \MlT)$, which induces isomorphisms $\pi_* (\Bl,\Ml) \cong \pi_*(\BlT, \MlT)$
 on all relative homotopy groups.

Now we can apply the Atiyah-Hirzebruch spectral sequence for stable homotopy
 \cite[Chapter XIII, Theorem 4.9]{Whitehead:1978}, and the relative Thom isomorphism to identify

\medskip
$$E^2_{p,q} = H_p(T(\wxi_\ell), T(\xi_\ell); \pi^S_q(pt)) \cong 
H_p(\Bl, \Ml;\pi^S_q(pt)\otimes \bZ^w)$$
in the AHSS for smooth bordism, and similarly
$$E^2_{p,q}(TOP) = H_p(T(\wxi_\ell), T(\xi_\ell); \pi^S_q(pt)) \cong 
H_p(\BlT, \MlT;\pi^S_q(pt)\otimes \bZ^w)$$
for the AHSS in topological bordism.

 The map $i \colon BO \to BTOP$ together with $i_{\Ml}$ and $i_{\Bl}$ induces a map of (smooth and topological) Thom pairs. Hence, by naturality of AHSS, we get a map of spectral sequences $E^2_{p,q}\to  E^2_{p,q}(TOP)$, which for each index $q$ has  the form
$$H_p(\Bl, \Ml;\cB) \longrightarrow H_p(\BlT, \MlT;\cB)$$ 
for the local coefficient system  $\cB =\pi^S_q(pt)\otimes \bZ^w$ induced from the action of $\pi_1(M)$.  Lemma \ref{lem:cocartsquare} gives that each of the squares in \eqref{diag:Ml_smooth} is at least $(n+k+3)$-cocartesian and by \cite[Proposition 3.7.24]{MunsonVolic:2015}, the induced map between homotopy cofibers is $(n+k+3)$-connected. Then Hurewicz theorem gives that $H_p(\Bl,\Ml;\cB) \to H_p(P_{\ell}(M), M;\cB)$  is an isomorphism for $p \leq n+k+2$ and surjection for $p =n+k+3$, and similarly for $(\BlT,\MlT)$. Therefore, we conclude that the spectral sequence map above is an isomorphism for $p \leq  n+k+2$, and surjective for $p = n+k+3$. Since  $\pi^S_q(pt)=0$ for $q<0$, both spectral sequences are first quadrant, and it follows that the induced maps
$$
\Omega_{i}(\Bl, \Ml) \xrightarrow{\cong} \Omega_{i}(\BlT, \MlT)
$$ 
are isomorphisms for $i \leq n+k+1$ and surjection for $i=n+k+2$.  Note that the skeleta filtration on the output $\Omega_*(K)$ of a bordism spectral sequence over a CW complex $K$ shows that
 there is an isomorphism
 $\Omega_{r}(K^{(r+1)}) \to \Omega_{r}(K)$.
Alternately, argue directly with the differentials $d_r\colon E^r_{p,q} \to E^r_{p-r, q +r -1}$ which decrease total degree $p +q \mapsto p+q -1$. Note that  if $d^r_{n-q+r, q-r+1}$ is non-zero, then $q+1 \geq r$. But the comparison map on homology in degree $n-q+r$ is an isomorphism unless $n-q+r \geq n+2$, implying that $r \geq q+2$ (contradiction). Therefore the only non-zero differentials hitting the $n$-line lie in the range where the homology comparison maps are isomorphisms.
\end{proof}

\begin{remark}
Note that Lemma \ref{lem:eight_iso_one} implies the ``hat'' version. In particular, for $k\geq 1$, the groups are already isomorphic to the unhatted versions. For $k=0$, the comparison map $i_{\Ml}$ preserves the degree, since $j_{TOP} \circ i_{\Ml} = j$, so $\on{deg}(j \circ \hat{f})= \on{deg}(j_{TOP} \circ i_{\Ml} \circ \hat{f})$  and similarly for the relative boundary map. 
\end{remark}

For the domain of $b_{k+1}$, we have: 

\begin{lemma}\label{lem:eight_iso_two} For $r \geq 0$ and $n \geq 4$, there is an isomorphism
$$ i_*\colon \pi_{r+1}(\heqpt{\Bl}, \heqpt{\Ml})  \to \pi_{r+1}(\heqpt{\BlT}, \heqpt{\MlT})$$
 induced by the natural map $i\colon BO \to BTOP$.
 \end{lemma}

\begin{proof}
By Definition \ref{def:thickaut}, the space $\heqpt{\Ml}$ is a subspace of  $\Map_{\bullet}(M,\Ml)$, and similarly $\heqpt{\Bl}$ is a subspace of  $\Map_{\bullet}(M,\Bl)$. 
Hence we obtain a diagram
 \eqncount
\begin{equation}\label{diag:Ml_E}
	\vcenter{
		\xymatrix@C+5pt@R+5pt{
			\heqpt{\Ml}\ar[r]^{j_*}\ar[d]^{(\wpp_\ell)_*} 
			& \heqpt{M}\ar[d]^{c_*}& \heqpt{\MlT} \ar[l]_{j_*}\ar[d]^{(\wpp_\ell)_*}\\
			\heqpt{\Bl}\ar[r]^(0.5){j_*} & \heqpt{\post{M}{\ell}}& \heqpt{ \BlT}\ar[l]_{j_*}
		}
	}
\end{equation}
in which both squares are cartesian (strict pullback squares), where we let  $\heqpt{P_\ell(M)}$ denote the subspace of $\Map_\bullet(M, P_\ell(M))$ consisting of the maps which are $(\ell+1)$-equivalences. The cartesian property follows directly from the definitions. 

Moreover, by \cite[Prop 3.2.5]{MunsonVolic:2015} the structure maps labelled $j$ in these diagrams are fibrations in the model for the homotopy pullbacks used in Definition \ref{def:hpull}. It then follows that the induced map $j_*\colon \heqpt{\Bl} \to \heqpt{P_\ell(M)}$ is a fibration (by verifying the adjoint form of the homotopy lifting property,  exactly as in  \cite[Prop 2.1.11]{MunsonVolic:2015}), and by 
 \cite[Prop 3.2.13]{MunsonVolic:2015} both
squares are homotopy pullbacks.

 Hence we conclude that  the homotopy fibre of the map $\heqpt{\Ml} \to \heqpt{\Bl}$ is weakly equivalent to the homotopy fibre of the map $\heqpt{M} \to \heqpt{\post{M}{\ell}}$. The same statement holds for the homotopy fibre of the map $\heqpt{\MlT} \to \heqpt{\BlT}$. Since the  homotopy fibre of the map $\heqpt{M} \to \heqpt{\post{M}{\ell}}$ does not depend on $BO$ or $BTOP$, the  resulting map  of pairs
$(\heqpt{\Bl}, \heqpt{\Ml}) \to (\heqpt{\BlT}, \heqpt{\MlT})$ induces an  isomorphism on all homotopy groups.
\end{proof}

\subsection{Homotopy groups and the maps $\alpha_*$ and $\beta_*$}\label{subsec:fiveseven}
We establish the exact sequences needed in Proposition \ref{prop:two_ladders},  relating the homotopy groups $\pi_k(\heqpt{\Ml})$ and $\pi_k(\heqpt{\Bl})$ with the modified bordism groups $\Omega_{n+k}(\Ml )$ and $\Omega_{n+k}(\Bl)$, respectively (see Lemma \ref{lem:eightone} and Lemma \ref{lem:eighttwo}).  For the topological case, these sequences are an immediate consequence of the existence of the vertical maps of spaces in the square in Theorem A. We currently do not know how to produce such maps of spaces in the smooth case, but nevertheless may produce the corresponding ``long exact sequence of homotopy groups'' needed for the proof of Theorem B in the smooth case by hand.

Elements of $\pi_k(\heqpt{\Ml})$ are represented by basepoint preserving maps $\hat{u} \colon (S^k,s_0) \to (\heqpt{\Ml},\hatid_M)$. Each such map has an adjoint $\hat{u}^{\#} \colon S^k \times M \to \Ml$ satisfying
\begin{itemize}
    \item $\hat{u}^{\#}(s_0,m)=\hatid_M(m)$ for all $m \in M$
    \item $\hat{u}^{\#}|_{\{s\}\times M} \in \heqpt{\Ml}$ for all $s \in S^k$.
\end{itemize}
Similarly, elements of $\pi_k(\heqpt{\Bl})$ are represented by basepoint preserving maps $\hat{v} \colon (S^k,s_0)\to (\heqpt{\Bl},\hat\nu_{\ell})$, with adjoint $\hat{v}^{\#}\colon S^k \times M \to \Bl$ satisfying $\hat{v}^{\#}(s_0,m)=\hat{\nu}_\ell(m)$ for all $m \in M$ and $\hat{v}^{\#}|_{\{s\}\times M} \in \heqpt{\Bl}$ for all $s \in S^k$.

\begin{definition}\label{def:alpha-and-beta-map}
In degree zero ($k=0$), we define
\begin{align*}
  \alpha_0 \colon \pi_0(\heqpt{\Ml}) &\longrightarrow \wOmega_{n}(\Ml) \quad \quad \quad &\beta_0 \colon \pi_0(\heqpt{\Bl}) &\longrightarrow \Omega_{n}(\Bl) \\
  [\hat{f} \colon M \to \Ml] &\longmapsto [(M, \hat{f})] &[\hat{g} \colon M \to \Bl] &\longmapsto [(M, \hat{g})].
\end{align*}
Here $\wOmega_{n}(\Ml)$ and $\Omega_n(\Bl)$ are regarded as pointed sets with base points $[(M,\hatid_M)]$ and 
$[(M,\widehat\nu_\ell)]$. For $k \geq 1$, we define  group homomorphisms
\begin{align*}
  \alpha_k \colon \pi_k(\heqpt{\Ml}) &\longrightarrow \wOmega_{n+k}(\Ml) \quad \quad \quad &\beta_k \colon \pi_k(\heqpt{\Bl}) &\longrightarrow \Omega_{n+k}(\Bl) \\
  [\hat{u} \colon S^k \to \heqpt{\Ml}] &\longmapsto [(S^k \times M, \hat{u}^{\#})] &[\hat{v} \colon S^k \to \heqpt{\Bl}] &\longmapsto [(S^k \times M, \hat{v}^{\#})].
\end{align*}
\end{definition}
The basepoint conditions and stable parallelizability of $S^k$ ensure that $\hat{u}^{\#}$ and $\hat{v}^{\#}$ are well-defined maps into $\Ml$ and $\Bl$, respectively, compatible with the normal data via $\xi_\ell$ and $\wxi_{\ell}$.  
Observe that a based homotopy of families produces the bordism $S^k \times M \times I$, so the assignment depends only on the homotopy class and hence the maps $\alpha_*$ and $\beta_*$ are well-defined.  For $k \geq 1$, compatibility of group operations follows from the pair of pants model for the pinch map, and the bijection $\wOmega_{n+k}(\Ml) \cong \O_{n+k}(\Ml)$  (see Remark \ref{rem:sixfourteen}). 

We need to define two additional maps:
\begin{enumerate}
\item There is a homomorphism $\wOmega_{n+k+1}(\Ml ) \to \htildeMk{k}$ defined by sending a closed $(n+k+1)$-manifold $N$ equipped with a map  $\wF\colon N \to \Ml$ to the class $[N \sqcup (D^{k+1}\times M), \wF \sqcup (\widehat{\text{id}}_M \circ p_M)]$. This map is well-defined since a $\xi_{\ell}$-bordism between two closed representatives extends by the product $\xi_{\ell}$-bordism on the boundary of $D^{k+1} \times M$. 
\item In addition, there is a natural map from $\htildeMk{k}$ to $\pi_k(\heqpt{\Ml})$, given by sending the class $[(W,\wF)]$ to the homotopy class  of the adjoint of the restriction map $(\wF|_{\bd W}) \colon S^k \times M \to \Ml$.  This map is well-defined since a boundary homotopy in the definition of bordism in $\htildeMk{k}$ gives a based homotopy of the adjoint boundary families.
\end{enumerate}

\begin{lemma}\label{lem:eightone} For $k \geq 1$, 
 there is a long exact sequence of groups  
\begin{align*}
 \pi_{k+1}(\heqpt{\Ml}) \xrightarrow{\ \alpha_*\ } \wOmega_{n+k+1}(\Ml) \to \htildeMk{k} \to \pi_k(\heqpt{\Ml})
  \xrightarrow{\ \alpha_*\ } \wOmega_{n+k}(\Ml)
\end{align*}
where all the maps are homomorphisms. For $k=0$, the same statement holds, except for the last term where the map $\alpha_0$ is a map of pointed sets. 
\end{lemma}
\begin{proof} We will outline the check of exactness at each term of the sequence,  and refer to Lemma \ref{lem:fiveseven} for the group structure on  $\htildeMk{k}$ for $k \geq 1$ (see also Lemma \ref{lem:fivenine}). Note that 
$ \wOmega_{n+k+1}(\Ml) =  \Omega_{n+k+1}(\Ml)$ if $k \geq 0$ (see Remark \ref{rem:sixfourteen}).

\medskip
\noindent
\textbf{Exactness at $\htildeMk{k}$:} The image of an element $(N,\wF) \in \O_{n+k+1}(\Ml)$ in $\htildeMk{k}$ is $[N \sqcup (D^{k+1} \times M),\wF \sqcup \hatid_M \circ p_M]$ which maps to the identity element $(\hatid_M \circ p_M|_{S^k \times M})^{\#} \in \pi_k(\heqpt{\Ml})$. 

Conversely,  let $[(W,\wF)] \in \htildeMk{k}$  and assume that  $[(W,\wF)]$ maps to the identity $[(\hatid_M \circ p_M)^{\#}] \in \pi_k(\heqpt{\Ml})$. Then, there is a homotopy $\hat{h} \colon S^k \times M \times I \to \Ml$ from $\wF|_{\bd W}$  to $\hatid_M \circ p_M$ (this based homotopy can be chosen relative to $D^k_-$ since $D^k_- \hookrightarrow S^k$ is a cofibration), and a closed $(n+k+1)$-manifold 
$$
N := W \cup (S^k \times M \times I) \cup (D^{k+1}\times M),$$ with reference map $\wG$ given by adjoining $\wF$, the homotopy $\hat{h} \colon S^k \times M \times I \to \Ml$, and $\hatid_M \circ p_M$ over $D^{k+1}\times M$. 
 Therefore $[(W,\wF)] = [(N, \wG) \sqcup (D^{k+1}\times M, \hatid_M \circ p_M) ] \in \htildeMk{k}$, by interpreting the manifold $(N \times I, \wG\circ p_N)$ as a relative bordism between $(W, \wF)$ and $(N, \wG) \sqcup (D^{k+1}\times M, \hatid_M \circ p_M)$.

\medskip
\noindent
\textbf{Exactness at $\pi_k(\heqpt{\Ml})$:} Note first that every based class in $\pi_k(\heqpt{\Ml})$ can be represented by a map which is constant on $D^k_-$. For $k \geq 1$, an element  $[\wF^{\#}] \in \pi_k(\heqpt{\Ml})$ lies in the kernel of $\alpha_*$ if and only if there exists an $(n+k+1)$-dimensional compact manifold $W$ with boundary $S^k \times M$ and a map $\wG\colon W \to \Ml$ extending $\wF$ and satisfying $\xi_{\ell} \circ \wG = \nu_W$. For $k=0$, a class $[\hat{f}]$ lies in 
 $\alpha_*^{-1}([M, \hatid_M])$ precisely when $[(M,\hat{f})]=[(M,\hatid_M)]$ in the pointed set $\wOmega_n(\Ml)$. Such a bordism is exactly a representative of $\htildeMk{0}$ whose nontrivial boundary is $\hat{f}$.

\medskip
\noindent
\textbf{Exactness at $\wOmega_{n+k+1}(\Ml)$:} Let $(N^{n+k+1}, \wF)$ be a representative of a bordism class in $\wOmega_{n+k+1}(\Ml)$ which maps to zero in  $\htildeMk{k}$. Then we have 
$$[(N \sqcup (D^{k+1} \times M), \wF \sqcup \widehat{\text{id}}_M \circ p_M)]= [(D^{k+1}\times M,\widehat{\text{id}}_M \circ p_M)]$$ which implies the existence of a homotopy $\hat{h} \colon S^{k} \times M \times I \to \Ml$ from $\widehat{\text{id}}_M \circ p_M$ to itself satisfying $\hat{h}|_{D^{k}_- \times M \times I}=\widehat{\text{id}}_M \circ p_M$ and $\hat{h}|_{\{x_+\} \times M \times \{t\}} \in \heqpt{\Ml}$ for all $x_+ \in D^{k}_+$ and $t \in [0,1]$,  where $S^{k}=D^{k}_- \cup_{S^{k-1}} D^{k}_+$.  In addition, we have  a relative $(n+k+2)$-bordism  $(Y, \wG)$ whose boundary  
$$\bd Y = N \sqcup (D^{k+1} \times M) \cup (S^k \times M \times I)\cup -(D^{k+1}\times M),$$
 and $\wG$ extends the maps to $\Ml$. 
The first condition implies that $\hat{h}$ preserves basepoints and these conditions allow us to define $\hat{h}\colon S^{k+1} \times M \to \Ml$ whose adjoint  represents an element in a homotopy class $[\hat{h}^{\#}] \in \pi_{k+1}(\heqpt{\Ml})$, using the fact that $S^{k+1}$ is reduced suspension of $S^{k}$. 
Moreover,  by considering $Y$ as a bordism between $N$ and 
$$S^{k+1} \times M =  (D^{k+1}_{+} \times M) \cup (S^k \times M \times I)\cup (D^{k+1}_{-}\times M),$$ 
 it follows that $\alpha_{k+1}([\hat{h}^{\#}])   \in \pi_{k+1}(\heqpt{\Ml})$ equals $[(N^{n+k+1}, \wF)]$. 
 
 Conversely, let $h \colon S^{k+1} \to \heqpt{\Ml}$ be a representative of an element in $\pi_{k+1}(\heqpt{\Ml})$ with image $[(S^{k+1} \times M,h^{\#})] \in \Omega_{n+k+1}(\Ml)$. Then we have 
 $$[(S^{k+1} \times M \sqcup D^{k+1}\times M, h^{\#} \sqcup \widehat{\text{id}}_M \circ p_M)]=[(D^{k+1} \times M, \widehat{\text{id}}_M \circ p_M)] \in  \htildeMk{k} $$ 
 by (i) trivializing $h \simeq h_1$ over the complement of a collar neighbourhood  of 
 $D^{k+1}_{+} \cap D^{k+1}_{-} \subset S^{k+1}$ to obtain a homotopy $h^\#_1 \colon S^{k}\times M \times I \to \Ml$ from $\widehat{\text{id}}_M \circ p_M$ to itself, and (ii) regarding the $(n+k+2)$-manifold $S^{k+1}\times M\times I$, 
 and reference map  $h_1^{\#} \circ p_{S^{k+1}\times M}\colon S^{k+1}\times M\times I \to \Ml$, as a relative bordism from 
 $S^{k+1} \times M \sqcup D^{k+1} \times M$ to $D^{k+1}\times M$.
\end{proof}

Similarly, there is a homomorphism $\O_{n+k+1}(\Bl ) \to \htildeBk{k}$ sending a closed manifold $N$ equipped with a map $\wG$ to $\Bl$ to the class $[N \sqcup (D^{k+1} \times M), \wG \sqcup (\hat{\nu}_\ell \circ p_M)]$. We also have a homomorphism $\pi_{k+1}(\heqpt{\Bl}) \to \O_{n+k+1}(\Bl, \wxi_{\ell})$ sending a homotopy class $[\hat{h} \colon S^{k+1} \to \heqpt{\Bl}]$ to the bordism class $[(S^{k+1} \times M, \hat{h}^{\#})]$. 
\begin{lemma}\label{lem:eighttwo}
For $k \geq 1$, 
 there is a long exact sequence of groups 
\begin{align*}
  \pi_{k+1}(\heqpt{\Bl}) \xrightarrow{\beta_*} \O_{n+k+1}(\Bl) \to \htildeBk{k} \to \pi_k(\heqpt{\Bl}) \xrightarrow{\beta_*} \O_{n+k}(\Bl ).
\end{align*} 
 For $k=0$, the same statement holds, except for the last term where the map $\beta_0$ is a map of pointed sets. 
\end{lemma}
\begin{proof} The details are similar to the proof of Lemma \ref{lem:eightone} and are left to the reader (see Lemma \ref{lem:fiveeight} and Lemma \ref{lem:fiveten}). 
\end{proof}

\section{The Hambleton-Kreck braid for $\hept{\Ml}$}\label{sec:HKthmB}
The braid of interlocking sequences constructed by Hambleton--Kreck  \cite[Theorem 3.15]{Hambleton:2004} to study $\hept{M}$  can be derived from  Theorem B, when we restrict $M$ to be a closed, oriented $4$-manifold and take $k=0$. 
 The only difference in the two braids is the ``hat'' condition on the Lashof bordism groups, but this has the consequence that our bordism ``sets'' are in bijection with bordism subgroups in the HK-version, and the bordism strand in that braid becomes a long exact sequence of groups and homomorphisms. This is useful for computations, so we outline the comparison in detail  and generalize the  HK-version to all manifolds  of dimension $n \geq 4$.
 
 To derive our braids from the square of spaces in Theorem A we imposed a ``degree one'' condition to define a space 
 $\wOmega^{\infty+n} Q(\xi_\ell)$, whose homotopy groups give the bordism groups 
 $\wOmega_{n+k}(\Ml) = \pi_k(\wOmega^{\infty+n} Q(\xi_\ell))$,  for $k \geq 0$. The homotopy groups are based at the element $\alpha(\hatid_M)$, and since 
  the spaces involved are group-like $H$-spaces, translation identifies the homotopy types of their connected components. In particular, the ``hat'' condition restricting the components only affects the bordism groups for $k=0$ (see Remark \ref{rem:sixfourteen}).

In the $4$-manifold setting of \cite[Theorem 3.15]{Hambleton:2004}, the three possible normal $2$-types for $M$ are distinguished by the $w_2$-types: (I) $w_2(\wM) \neq 0$,  (II) $w_2(M) = 0$, and (III) $w_2(M) \neq 0$, but   
 $w_2(\wM) =0$. In our notation, these $w_2$-types are encoded by  $\hept{\Mw} = \pi_0(\heqpt{\Ml})$ and $\hept{\Bw} = \pi_0(\heqpt{\Bl})$, for $\ell =2$. 
There is a map 
$$\alpha^{HK}\colon \hept{\Mw} \to \wOmega^{HK}_4(\Mw)$$
 analogous to $\alpha_0$, which has as its codomain the subgroup $\wOmega^{HK}_4(\Mw) \subseteq \Omega_4(\Mw)$  of the Lashof bordism group, given by imposing a \emph{degree zero} condition on the reference map. We generalize this map to all dimensions $n \geq 4$ by imposing the same condition on the codomain of $\alpha_0$.
 
 \begin{definition} For $n \geq 4$ and $\ell \geq 2$, let
  $\wOmega^{HK}_n(\Ml) \subseteq \Omega_n(\Ml)$ 
 denote the subgroup consisting of the elements 
   $[N,\hat{f}]$, where $N$ is a closed $n$-manifold and $\hat f\colon N \to \Ml$ is a map with $\on{deg}(j \circ \hat f) = 0$. Similarly, let $$\wOmega^{HK}_{n+1}(\Bl, \Ml) \subseteq \Omega_{n+1}(\Bl, \Ml)$$
 denote the subgroup defined on an element $[(W, \bd W), (\hat g, \bd \hat g)]$ by requiring $\on{deg}(j \circ  \bd \hat g) = 0$.
 
 Let 
 $$\alpha^{HK}\colon \hept{\Ml} \to \wOmega^{HK}_n(\Ml)$$
 be defined by $\alpha^{HK}(\hat f) = [M, \hat f] - [M, \hatid_M]$, for $[\hat f] \in \hept{\Ml}$. 
 \end{definition}
 
 Then $\alpha^{HK}(\hat f) = [M, \hat f] - [M, \hatid_M]  = \alpha_0([\hat f]) -  [M, \hatid_M] $ gives the map used in the braid and a similar definition gives the map $\beta^{HK} \colon \hept{\Bw} \to \wOmega^{HK}_4(\Bw)$.

	Even though there is no degree constraint on the codomain of $\beta_0$, there is a similar relationship between the induced $\beta_0$ map and the map  $\beta^{HK}$ map in \cite{Hambleton:2004},  relating the two different versions   $\wOmega_{n+1}(\Bl, \Ml)$ and   $\wOmega_{n+1}^{HK}(\Bw, \Mw)$ of the relative bordism group,
	given explicitly by 
	$[(W,\wF,\bd \wF)] \mapsto \left[\big(W \sqcup -(D^1 \times M), \wF \sqcup (\hat{\nu}_{\ell} \circ p_M),\bd \wF \sqcup \hatid_M \sqcup \hatid_M \big) \right]$.

\begin{corollary}
	\label{cor:theorem-C}
	Let $M$ be a closed, smooth or topological manifold of dimension $n \geq 4$. For   $\ell \geq  [(n+1)/2]$,	
  there exists an interlocking braid of exact sequences: 
	\vskip .2cm
		$$\vcenter{\xymatrix@R-15pt@C-28pt{
					\Omega_{n+1}(\Ml)    \ar[dr] \ar@/^2pc/[rr]   &&
					\widetilde{ \cH}_0(\Ml)  \ar[dr] \ar@/^2pc/[rr] &&
					\hept{\Bl}\ar[dr]^{\beta}  \\
					& \Omega_{n+1}(\Bl) \ar[dr] \ar[ur]  &&
					\hept{\Ml}  \ar[dr]^{\alpha}\ar[ur] &&  \Omega_{n}(\Bl) \\
					\pi_{1}(\heqpt{\Bl}) \ar[ur] \ar@/_2pc/[rr] &&  \wOmega^{HK}_{n+1}(\Bl,\Ml)
					\ar[ur]^{\gamma} \ar@/_2pc/[rr]&&
					 \wOmega^{HK}_{n}(\Ml)\ar[ur] 
		}}		$$
	\vskip .7cm
	All the maps in the diagram are group homomorphisms, except that the maps $\alpha\coloneqq  \alpha^{HK}$, $\beta\coloneqq \beta^{HK}$ and $\gamma\coloneqq\gamma^{HK}$ are only maps of pointed sets.
\end{corollary}
\begin{proof}
To relate the maps used in   \cite[Theorem 3.15]{Hambleton:2004} to those in Theorem B, observe that there is a difference map 
$$D\colon \wOmega_n(\Ml) \to \wOmega^{HK}_n(\Ml)$$
defined by mapping the connected component of ``degree one'' elements, represented by 
$[(N,  \hat g)]$,  to the bordism element 
	$[(N \sqcup -M, \hat g \sqcup \hatid_M)] \in \wOmega^{HK}_n(\Ml)$.    
	Since $\hat g$ has degree one, the difference element has degree zero, and we have the relation $\alpha^{HK} = D \circ \alpha_0$.
	 Note that the map $D$ is a bijection, with inverse map $E\colon \wOmega^{HK}_n(\Ml) \to 
	\wOmega_n(\Ml)$ defined by $[N', \hat g'] \mapsto [(N', \hat g') \sqcup (M, \hatid_M)]$. One can check that $E\circ D = \id$ and $D\circ E = \id$.  A similar discussion holds for the relative groups 
	$ \wOmega^{HK}_{n+1}(\Bl,\Ml)$. We can now use the maps $D$ and $E$ to replace the bordism groups in Theorem B by the HK-variants. It is easy to check that the resulting interlocking braid agrees with that in \cite[Theorem 3.15]{Hambleton:2004} when specialized to that setting. 
\end{proof}
\begin{remark}\label{rem:pamuk} Pamuk \cite[Theorem 2.17]{Pamuk:2026} gives a related braid for ordinary based self-equivalences $\hept{M}$ of closed, oriented $BO\langle r+1\rangle$-manifolds $M$ of dimension $2r$ or $2r-1$, under some additional assumptions, together with  some calculations for products and connected sums of spheres.
\end{remark}

\section{Applications and further questions}\label{sec:nine}
We provide three  new applications of the braid (Theorem \ref{thm:sixone}, Theorem \ref{thm:sixtwo} and Theorem \ref{thm:sevensixteen}), and  a short survey of the literature, together with some related questions. As above, we will give the statements for smooth manifolds, but the corresponding results also hold in the topological manifold setting. 

\subsection{Computations}
 So far the braid in Theorem B has only been computed for $\hept{M}$ (the case $k=0$)  under some special assumptions.
  For example, the structure of $\hept{M}$ is known if  $M$ is any closed, oriented, smooth (or topological) $4$-manifold with odd order finite fundamental group (see \cite[Theorem A and \S 4]{Hambleton:2004}). In addition, an approach to the classification of $4$-manifolds based on calculations of the braid was introduced by Pamuk \cite{Hegenbarth:2015a,Pamuk:2009,Pamuk:2010}, and further developed in \cite{Hambleton:2023}. We expect that the braid will lead to new information  in the future about the higher homotopy groups of the spaces of homotopy self-equivalences.

\subsection{Stability of $\heqpt{\Bl}$ as $\ell$ increases}
\label{subsection:space-of-self-equivalences-and-postnikov}
As $\ell$ increases, the Postnikov sections $P_\ell(M)$ converge to $M$ up to homotopy. 
The natural transformations $P_{\ell+1}(-) \to P_{\ell}$ induce compatible transition maps 
$$
M \langle \ell+1 \rangle \to \Ml \quad \text{and} \quad B \langle \ell+1 \rangle \to \Bl,
$$
and hence towers of the corresponding spaces of self-equivalences.

A result of Dror, Dwyer and Kan \cite[Theorem 1.1]{Dror:1979} shows that the spaces $B\heqpt{P_\ell(M)}$ converge to $B\heqpt{M}$ (by taking the ``identity component'' of  a suitable homotopy inverse limit).

Let $\heqpt{\Bl}_{c_*[M]}$ denote the subspace of $\heqpt{\Bl}$ consisting of all connected components that contain an ``orientation preserving'' homotopy self-equivalence,  meaning an element $(\phi, H_\phi) \in \heqptkc{B}{\ell}$  such that $\phi_*(c_\ell [M]) = c_\ell [M] \in H^t_n(P_\ell(M); \bZ)$.  Using the braid diagram, we obtain a result similar to \cite[Theorem 1.1]{Dror:1979} relating the homotopy inverse limit of $\heqpt{\Bl}_{c_*[M]}$ to the space $\heqpt{M, \nu_M}$ of homotopy self-equivalences of $M$ which preserve the stable normal bundle up to prescribed homotopy.

\begin{theorem}\label{thm:sixone}
The compatible maps induced by Postnikov truncation define a natural map
$$
\Theta \colon \heqpt{M,\nu_M} \longrightarrow \holim \heqpt{\Bl}_{c_{ *}[M]}.
$$
This map is a weak equivalence. 
\end{theorem}

\begin{proof}
All homotopy groups below are based at the identity elements. Since the spaces involved are group-like $H$-spaces, translation identifies the homotopy types of their connected components.

Since $\wpp_{\ell} \colon \Ml \to \Bl$ is $(\ell+1)$-connected, relative Hurewicz with local coefficients gives
$H_p(\Bl,\Ml;\mathcal{B})=0$ for $p \leq \ell+1$ for every local coefficient system $\mathcal{B}$. The relative
Atiyah--Hirzebruch spectral sequence (see Remark \ref{rem:ahss}) therefore gives $\Omega_r(\Bl,\Ml)=0$ for $r \leq \ell+1$ and hence $\wOmega_r(\Bl,\Ml)=0$ for $r \leq \ell+1$ (see Section \ref{subsec:input_braid}).

For $k>0$, the braid contains the exact sequence
$$
\widehat{\Omega}_{n+k+1}(\Bl,\Ml)  \longrightarrow \pi_k(\heqpt{\Ml})  \longrightarrow \pi_k(\heqpt{\Bl})
 \longrightarrow \wOmega_{n+k}(\Bl,\Ml).
$$
Thus the middle map is an isomorphism for $k \leq \ell-n$ and is surjective for $k = \ell - n+1$. It remains to identify its image on $\pi_0$. Let $(\phi,H_\phi) \in \heqptkc{B}{\ell}$ satisfy $\phi_*(c_*[M])=c_*[M]$. For $\ell \geq n= \dim(M)$, obstruction theory applied to the homotopy fibre of $c_{\ell} \colon M \to P_{\ell}(M)$ provides a based map $f \colon M \to M$ and a based homotopy $H_2 \colon c_{\ell} \circ f \simeq \phi \circ c_{\ell}$. This homotopy implies that $f$ is an $\ell$-equivalence. Moreover, since $c_*$ is an isomorphism in degree $n$,
$$
(c_* \circ f_*)[M]  = \phi_* (c_{ *}[M])  = c_{*}[M]
$$
implies $f_*[M]=[M]$. By Poincar\'e duality with local coefficients, $f$ is a homology equivalence, and hence a homotopy equivalence. Define a homotopy $H_f = (\nu_{\ell} \circ H_2) \ast \bigl(H_{\phi} \circ (c_{\ell} \times \id_I)\bigr)$ from $\gamma_{\ell} \circ f$ to $\nu_{\ell}\circ c_{\ell} =\gamma_{\ell}$. Then $\hat{f}= (f,H_f)$ is an element of $\heqpt{\Ml}$, and the homotopy-pullback description of $\heqpt{\Bl}$ shows that
$\wpp_{\ell}(\hat f)$ lies in the connected component represented by $(\phi,H_\phi)$. Consequently, $\heqpt{\Ml} \to \heqpt{\Bl}_{c_{\ell *}[M]}$ is an $(\ell-n+1)$-equivalence.

We use the standard consequence of Milnor's $\varprojlim^1$-sequence (see \cite[Lemma 2]{Milnor:1962} and \cite[Chap.~IX, Theorem 3.1]{Bousfield:1972}). In particular, if $X_{\ell} \to Y_{\ell}$ is a map of towers whose connectivity tends to infinity with $\ell$, then $\holim  X_{\ell} \to \holim  Y_\ell$ is a weak equivalence. Indeed, for each fixed $q$, the maps on $\pi_q$ and $\pi_{q+1}$ are isomorphisms on a cofinal tail of the towers, and hence induce isomorphisms on both the $\varprojlim$ and $\varprojlim^1$ terms in the Milnor exact sequences. It follows that
$$
\holim  \heqpt{\Ml}  \xrightarrow{\ \approx\ } \holim  \heqpt{\Bl}_{c_{\ell *}[M]}.
$$
Finally, the maps $\heqpt{M,\nu_M}\to \heqpt{\Ml}$ induced by $BO \to P_\ell(BO)$ have connectivity tending to infinity. Indeed, this follows from obstruction theory, since the fibre of $BO \to P_\ell(BO)$ is $\ell$-connected and $M$ is a finite $n$-dimensional CW complex. Applying the same homotopy inverse limit
argument gives
$$
\heqpt{M,\nu_M} \xrightarrow{\ \approx\ } \holim  \heqpt{\Ml}.
$$
Combining the last two weak equivalences proves the result.
\end{proof}

\subsection{The group of homotopy self-equivalences of manifolds}
The groups $\hept{M^n}$ have previously been described in terms of a short exact sequence of more familiar groups  in a meta-stable range of dimensions, where $M^n$ is $(k-1)$-connected, with $k \geq 2$, and $n \leq 3k-2$,
in some cases under the assumption that $M$ is stably parallelizable (see \cite{Barcus:1958,Baues:1996a,Crowley:2011,Kahn:1969,Krylov:2003}).

We would like to explain how to recover  some of these explicit results from the braid diagrams,  for example when $n = 2k$, and provide a generalization. The problem in the braid setting is that the computation of the bordism groups
$\Omega_{*}(M\la k \ra)$ appears to be difficult. On the other hand, it is clear from the answers obtained by other authors that the image of
$\gamma_0\colon \widehat{\Omega}_{2k+1}(B \la k \ra, M\la k\ra) \to \hept{M\la k \ra}$ appearing in Theorem B and \cite{Hambleton:2004}
 must account for the ``pinch map'' elements arising
from  the homotopy groups $\pi_{2k}(S^k)$ and $\pi_k(M)$.

\begin{example} 
If $M$ is a $(k-1)$-connected stably parallelizable $2k$-dimensional manifold, then $\nu_M \simeq \ast$ implies that   $[(f,H_f)] = [(f,H_f')] \in \hept{M \la k \ra}$ for any $f \in\hept{M}$ and for any two homotopies $H_f, H_f' \colon \gamma_k \simeq \gamma_k \circ f$,  since $[\Sigma M, P_k(BO)]=\{ \ast \}$. Therefore,  $\hept{M \la k \ra}=\hept{M}$ and $\hept{P_k(M)}=\text{Aut}(\pi_k(M))$, since $P_k(M)$ is a $K(\pi_k(M),k)$-space and each connected component of $\heqpt{K(\pi_k(M),k}$ is contractible (see \cite[19.1.3]{Rutter:1997}). 
\end{example}

The work of Baues \cite[Theorem 8.14]{Baues:1996a} suggests that we should consider the complement of a small embedded $n$-disk in $M^n$, for $n \geq 4$,  and consider the homotopy self-equivalences which restrict to the identity on the complement.
That choice is adequate for simply-connected manifolds, but in general we let  $M_0 \coloneqq M - \mathring H$, where
 $\mathring H$ denotes the interior of a thickening $H$ of a suitable embedded $1$-complex representing a  finite set  of generators for $\pi_1(M, x_0)$.

 More precisely, let $\phi\colon F \to \pi_1(M, x_0)$ be a surjection from a finitely generated free group $F$ with generators
$\{\alpha_1, \dots, \alpha_t\}$, such that their images $\{g_i = \phi(\alpha_i)\}$ give a  finite set of generators for $\pi_1(M, x_0)$. Let $K \subset M$ be an embedded (locally flat) wedge of circles at the base point $x_0 \in M$, representing  the given set of generators $S \coloneqq\{g_i \in \pi_1(M, x_0)\}$. Then we let $H$ denote a small regular neighbourhood of $K$ in $M$. 
 \begin{example}
If $M$ is simply-connected, we can  let $H = D^n$. If $M$ is smooth, we could also let $H$ be the union of the handles of index $\leq 1$ provided by a Morse function.
\end{example}
 In general, $H \cong \natural (S^1 \widetilde{\times} D^{n-1})_j$ is a boundary connected sum of possibly non-trivial bundles over $S^1$,  since the normal bundles of the core circles could be non-trivial if $M$ is non-orientable. We can \emph{stabilize} this construction by connected sum with  $S^n = (S^1 \times D^{n-1} ) \cup (D^2 \times S^{n-2})$, which adjoins an additional (null-homotopic) circle to $K$, and stabilizes the homotopy type of $M_0$ by wedge sum with $S^{n-2}$ while keeping that of $M$ unchanged.

\begin{lemma}\label{lem:seventhree} For $n \geq 4$, the inclusion $i\colon M_0 \subset M$ is an $(n-2)$-equivalence, and after stabilization as above (if necessary),  there is a homotopy equivalence of  finite CW pairs $(M, M_0)\simeq (P, Q)$, where $\dim Q \leq n-2$ and $\dim \sigma \geq n-1$, for all cells $\sigma \subseteq P - Q$.
\end{lemma}
\begin{proof}
 By general position, the homotopy groups $\pi_i(M, M_0) = 0$ for $i \leq n-2$, and  $\pi_1(M_0) \cong \pi_1(M)$ since $n \geq 4$. Hence the inclusion $i\colon M_0 \subset M$ is an $(n-2)$-equivalence.

 In the following, we let $\Lambda \coloneqq \bZ \pi_1(M)$ and $\Gamma \coloneqq\bZ \pi_1(H)$  denote the integral group rings, and note that we may regard a right $\La$-module $\cB$ as a right $\Gamma$-module by restriction along the surjection $\pi_1(H) \to  \pi_1(M)$.

  Since $M_0$ is a compact manifold with boundary, it is has the homotopy type of a finite complex. To show that $M_0$ is homotopy equivalent to a finite complex $Q$ with $\dim Q \leq n-2$, the first step is to check that  $M_0$ satisfies Wall's $D(n-2)$-condition \cite[\S 2]{Wall:1965}. In other words, we must verify that $H_i(M_0;\La) = 0$ for $i \geq n-1$ and that $H^{n-1}(M_0; \cB) = 0$ for any right $\Lambda$-module $\cB$.
 
  We start with an observation. Since $H$ and $M$ are connected, the surjectivity of $\pi_1(H) \to  \pi_1(M)$ implies that $\pi_i(M,H) = 0$ for $i = 0,1$ (i.e.\ the pair $(M,H)$ is $1$-connected). By \cite[Lemma 15, p.~46]{Whitehead:1950}, the pair $(M,H)$ is therefore homotopy equivalent, relative to its 1-dimensional core $K\simeq H$,  
  to a pair $(Z, K)$ with no cells in $Z-K$ of dimension $\leq 1$. 
  
  Now, to show that $H_i(M_0;\La) = 0$ for $i \geq n-1$, it is enough to consider $i \in \{n-1, n\}$ since the dimension of $M_0$ is $n$. Poincar\'e duality for the pair $(M_0, \partial M_0)$ and excision yields an isomorphism $H_{i}(M_0; \Lambda) \cong  H^{n-i}_t(M_0, \bd M_0; \Lambda)\cong H^{n-i}_t(M, H; \Lambda)$. The cohomology groups on the right side are trivial for $n-i \in \{0, 1\}$, because the pair $(M, H)$ has (up to homotopy equivalence) no cells in $M-H$ of dimension $\leq 1$ by the preceding observation. Similarly, to show that $H^{n-1}(M_0; \cB) = 0$ for any right $\La$-module $\cB$, we again apply Poincar\'e duality, excision and the observation to obtain $H^{n-1}(M_0; \cB)\cong H_1^t(M, H; \cB) = 0$.
 
 If $n \geq 5$, we can directly apply \cite[Theorem E]{Wall:1965} to conclude that $M_0$ is homotopy equivalent to a finite complex 
 $Q$ with $\dim Q \leq n-2$ (and hence its Wall finiteness obstruction vanishes). If $n=4$, Wall's $D2$-condition (verified in the previous paragraph) implies that $M_0$ is homotopy equivalent to a finite $3$-complex (see  \cite[p.~64]{Wall:1965}).
    Hence, we may assume (up to homotopy equivalence) that $M_0$ is a finite, connected $3$-complex satisfying Wall's $D2$-condition. This is precisely the setting of \cite[Section 2]{Hambleton:2019}, and therefore \cite[Lemma 2.1]{Hambleton:2019} applies to show that $M_0$ is (simple) homotopy equivalent to a finite $2$-complex after stabilizing with finitely many copies of $S^2$.

 The last statement about the cells in $P-Q$ follows from the first conclusion of \cite[Lemma 15, p.~46]{Whitehead:1950}, which applies with $r = n-2$.  \end{proof}

\begin{assumption}
For the rest of this subsection, we will assume that $M_0$ is stabilized in the sense above, so that it satisfies the conclusions of Lemma \ref{lem:seventhree}.
\end{assumption}

 Let $\hept{\Ml\vv M_0} $ be defined as
 $$\hept{\Ml\vv M_0} = \{[\hat f] \in \hept{\Ml}  \vv   \hat f|_{M_0}\simeq \hatid_{M_0} \},$$
 where $\hat f|_{M_0} = \hat f \circ i \colon M_0 \to \Ml$ and $\hatid_{M_0} = \hatid_M \circ \, i \colon M_0 \to \Ml$ denote the restrictions of the domain along the inclusion $i \colon M_0 \to M$, and
     the homotopy is constant over $BO$, so that  $\xi_{\ell} \circ h_t = \nu_M|_{M_0}$ for all $t \in I$, with no extra condition imposed on $\hat f|_{\bd M_0}$.
 Note that if $ \hat f_1|_{M_0}\simeq \hatid_{M_0}$ and  $\hat f_1 \simeq \hat f_2$
  then $ \hat f_2|_{M_0}\simeq \hatid_{M_0}$, so the  subset
 $\hept{\Ml\vv M_0}$ is well-defined.  We next point out an equivalent form of the definition (expressed  in the notation of Remark \ref{rem:threethree}): it then follows that $\hept{\Ml\vv M_0}$  is a subgroup of $\hept{\Ml}$.

 \begin{lemma}\label{lem:sevenfive}
  $\hept{\Ml\vv M_0} = \{[(f, H_f)]   \vv   f\circ i = i \text{\, and\, } H_f|_{M_0 \times I} = \gamma_\ell \circ i \circ p_{M_0}\}$. 
   \end{lemma}
   \begin{proof} If $\hat f \in \heqpt{\Ml}$ and $ \hat f|_{M_0} \simeq \hatid_{M_0}$ via a homotopy $\hat h_0 \colon M_0 \times I \to \Ml$, then by the homotopy extension theorem we get $\hat g \simeq \hat f$  in $\heqpt{\Ml}$ such that $\hat g|_{M_0} = \hatid_{M_0}$.  Similarly we get 
   $H_g \simeq H_f$, with $H_g|_{M_0\times I} = \gamma_\ell \circ i \circ p_{M_0}$.  The converse is obvious.
   \end{proof}
Our goal is to identify this subgroup with the image of
 $$\gamma_0\colon \wOmega_{n+1}(B \la \ell \ra, M\la \ell \ra) \to \hept{M\la \ell \ra}$$ 
  for $\ell = n-2$, by showing how it fits into the braid diagram of Corollary \ref{cor:theorem-C}.

\medskip
The next result is a generalization of the work of Barcus and Barratt
\cite[\S 6]{Barcus:1958} and Baues \cite[Theorem 8.14]{Baues:1996a}  from the simply-connected setting. This will involve the group $\hept{\Bl}_{c_*[M]} = \pi_0(\heqpt{\Bl}_{c_*[M]})$, where $\heqpt{\Bl}_{c_*[M]}$ is as in Section \ref{subsection:space-of-self-equivalences-and-postnikov}. More explicitly, it is the subgroup of $\hept{\Bl} \cong \pi_0(\heqptkc{B}{\ell})$ given by
$$\hept{\Bl}_{c_*[M]}=\left \{ [\hat{\phi} = (\phi, H_{\phi})]\colon   \phi_*(c_\ell[M]) = c_\ell[M]  \in H^t_n(P_\ell(M); \bZ)\right \}
$$
in the notation of Definition \ref{def:thickautBc}.

 \begin{theorem}\label{thm:sixtwo} If $\dim M = n \geq 4$ and $\ell = n-2$, there is an exact sequence of groups
 $$1 \to \hept{\Ml\vv M_0}  \to \hept{\Ml} \xrightarrow{(\wpp_\ell)_*} \hept{\Bl}_{c_*[M]}\to 1.$$
 \end{theorem}

\begin{proof}  Let $i \colon M_0 \to M$ denote the inclusion map, such that $\hatid_{M_0} = \hatid_M \circ i \colon M_0 \to \Ml$.
	 We first show that $\hept{\Ml\vv M_0}  \subseteq \ker (\wpp_\ell)_*$, and let $(\hat f \colon M \to \Ml) \in \hept{\Ml\vv M_0}$. We wish to construct a homotopy  $\hat h_2 = (\wpp_\ell(\hat f) \simeq  \hat \nu_\ell) \colon M \times I \to \Bl$. Since $\hat f \in \hept{\Ml\vv M_0}$, there exists a homotopy
	 $$\hat h_0 = (\hat f \circ i \simeq \hatid_M \circ i) \colon M_0 \times I \to \Ml.$$
	Post composition with $c_{\ell} \circ j \colon\Ml \to P_{\ell}(M)$ and unraveling the definition of mapping into $\Ml$ leads to a homotopy $h_1 = c_{\ell} \circ j \circ \hat h_0\colon M_0 \times I \to P_\ell(M)$ together with a homotopy $H_1\colon (M_0 \times I) \times I \to P_{\ell}(BO)$ between $\nu_{\ell} \circ h_1$ and $\gamma_{\ell}|_{M_0} \circ \on{pr_{M_0}} \colon M_0 \times I \to P_\ell(BO)$.
	
	We will now perform two obstruction theory arguments. First, we wish to extend the map $h_1\colon M_0 \times I \to P_\ell(M)$ along the inclusion $M_0 \times I \hookrightarrow M \times I$, relative to $M \times 0 \cup M_0 \times I \cup M \times 1$,  with  $c_{\ell} \circ f$ on $M \times 0$, $h_1$ on $M_0 \times I$ and $c_{\ell}$ on $M \times 1$. The obstructions to  extending $h_1$ to $h_1'=(c_{\ell} \circ f \simeq c_{\ell})$ lie in 
	 $$H^{i+1}_t(M \times I, M \times \bd I \cup M_0 \times I;\pi_i(P_{\ell}(M)))
	=H^i_t(M, M_0; \pi_i(P_{\ell}(M))), $$ 
	where $\pi_i(P_{\ell}(M)) = 0$ for $i \geq n-1$ (since $\ell = n-2$).  However, the pair $(M, M_0)$ is homotopy equivalent to a CW pair with additional cells only in dimensions $\geq n-1$ (by Lemma \ref{lem:seventhree}), so all the obstructions vanish.
    
    In the second step, we wish to use $h_1'=(c_{\ell} \circ f \simeq c_{\ell})$ as an input to extend the homotopy $H_1\colon (M_0 \times I) \times I \to P_{\ell}(BO)$ along the inclusion $(M_0 \times I) \times I \hookrightarrow (M \times I) \times I$ and relative to 
       $(M \times I) \times \bd I \cup (M_0 \times I) \times I$  with $\nu_{\ell} \circ h_1'$ on $(M \times I) \times 0$, $H_1$ on $(M_0 \times I) \times I$ and
      $\gamma_\ell \circ \on{pr}_M$ on $(M \times I) \times 1$. The obstructions to  extending $H_1$ to 
     $H_1'=(\nu_{\ell} \circ h_1' \simeq \gamma_\ell \circ \on{pr}_M)$
     lie in $H^i(M \times I, M_0 \times I; \pi_i(P_{\ell}(BO)))$, where $\pi_i(P_{\ell}(BO)) = 0$ for $i \geq n-1$ (since $\ell = n-2$). However, the pair $(M \times I, M_0 \times I)$ is homotopy equivalent to a CW pair with additional cells only in dimensions $\geq n-1$ (since by Lemma \ref{lem:seventhree} this holds for $(M, M_0)$), so again  all the obstructions vanish.  As a result we have a map $h_2 = (h_1', H_1') \colon M \times I \to \Bl$, which is the claimed homotopy  $\wpp_\ell(\hat f) \simeq \hat \nu_\ell$.
	
	Next we show that the inclusion $\hept{\Ml\vv M_0} \to \ker (\wpp_\ell)_*$ is an isomorphism. Suppose that $\hat f \in \ker (\wpp_\ell)_*$, so that there exists a homotopy $\hat h \colon M \times I \to \Bl$  in $\heqpt{\Bl}$ giving $\wpp_\ell \circ \hat f \simeq \hat \nu_\ell$. The obstruction to lifting the homotopy $\hat h\circ i \colon M_0 \times I \to \Bl$ to $\Ml$,  relative to $M_0 \times \bd I$, lie in $H^{i+1}(M_0 \times I, M_0 \times \bd I; \pi_i(\Fl)) \cong \bar H^{i}(M_0; \pi_i(\Fl))$, where $\Fl$ is the homotopy fibre of the map  $\wpp_{\ell} \colon \Ml \to \Bl$. But $\pi_i(\Fl) = 0$ for $i \leq n-2$ by Lemma \ref{lem:normalk-type} and by Lemma \ref{lem:seventhree} we may assume that $M_0$ is homotopy equivalent to a CW complex with $\dim M_0 \leq n-2$, so all the obstructions vanish and $\hat f\circ i \simeq \hatid_{M_0}$.
	 Therefore
 $\hept{\Ml\vv M_0}  = \ker (\wpp_\ell)_*$.

The surjectivity  in the extension follows from \cite[Theorem 3.1]{Baues:2008} and \cite{Hambleton:1988,Hambleton:2018} since $\ell  = n-2$. If $\phi\colon P_\ell(M) \to P_\ell(M)$ is a homotopy equivalence, then the condition implies that there exists  $f \colon M \to M$ a homotopy equivalence such that 
  $H_2 = (\phi\circ c_\ell \simeq c_\ell \circ f)$ (see \cite[Corollary 3.3]{Baues:2008}). If $(\phi, H_\phi) \in \heqpt{\Bl}_{c_*[M]}$, then 
  $\nu_\ell \simeq \nu_\ell \circ \phi$ via $H_\phi$, and hence   
 $(f, H_f ) \in \heqpt{\Ml}$, where $H_f \coloneqq (H_{\phi} \circ (c_{\ell} \times \id_I)) \ast (\nu_{\ell}\circ H_2)$ concatenates the two homotopies.
 In the homotopy pullback model of $\heqpt{\Bl}$, the homotopy $H_2$, together with this choice of $H_f$, gives a path $\wpp_{\ell}(\hat{f}) \simeq (\phi,H_{\phi})$. 
\end{proof}
\begin{corollary}\label{cor:im-gamma} For $n \geq 4$ and $\ell = n-2$, the subgroup $\hept{\Ml\vv M_0} \cong \Image (\gamma_0)$ is independent of the choice of a finite generating set for $\pi_1(M, x_0)$. 
\end{corollary}

\begin{remark} If we omit the normal structure conditions on  based, locally oriented homotopy self-equivalences,
we obtain an exact sequence
\eqncount
\begin{equation}\label{eq:seveneight}
1 \to \hept{M\vv M_0}  \to \hept{M} \xrightarrow{c_*} \hept{P_\ell(M)}_{c_*[M]} \to 1
\end{equation} for $n \geq 4$ and $\ell = n-2$. The arguments are similar to the ones above.   Here $\hept{P_\ell(M)}_{c_*[M]}$ denotes the subgroup of self-equivalences preserving the orientation local system $\bZ^w$ and fixing the class $c_{*}[M] \in H_n(P_{\ell}(M);\bZ^w)$, see Remark \ref{rem:foursix}.
\end{remark}

\begin{remark}
In Baues \cite[Theorem 8.14]{Baues:1996a} the group $\hept{M\vv M_0}$ with $M_0 = M - D^{2k}$ is  identified with $H_k(M;\bZ) \otimes T_k$, where $M^{2k}$ is an almost parallelizable $(k-1)$-connected $2k$-manifold. Here $T_k = {\rm Tors}(\pi_{2k+1}(S^{k+1}))$. This group is related by a calculation to the ``pinching maps'' given by elements in $\pi_{2k}(S^k)$:  see \cite[Theorem 3.5]{Baues:1996a} combined with \cite[Proposition 1.3]{Baues:1996a} which gives a surjection
 $1^{+}\colon \pi_{2k}(M_0) \twoheadrightarrow \hept{M\vv M_0}$ and $\pi_{2k}(M_0) \cong H_k(M;\bZ) \otimes \pi_{2k}(S^k)$.
 \end{remark}
 
 \subsection{Commensurability properties of $\hept{\Ml}$}
We would like to study the set of components $\pi_0(\heqpt{\Ml}) \cong \hept{\Ml}$ for manifolds with finite fundamental group. The use of the braids leads to generalizations of related work on 
$\hept{M}$ (see  \cite[Theorem 2]{Triantafillou:1995}, \cite[Theorem A]{Hambleton:2004},  \cite[Theorem D]{Bustamante:2024}), investigating its commensurability up to finite kernel with an arithmetic group. Here is a preliminary result.

\begin{lemma}\label{lem:seveneight}
For $n \geq 4$ and $\ell \geq 2$, there is an exact sequence of groups 
$$
[\Sigma (M, M_0), P_\ell(BO)]_{\gamma_\ell} \to \hept{\Ml\vv M_0} \xrightarrow{\ j_* \ }  \hept{M\vv M_0}
$$
relating $\hept{\Ml\vv M_0}$ to the usual group of homotopy self-equivalences. 
\end{lemma}
\begin{proof}
The homotopy pullback $(BO \to P_{\ell}(BO) \leftarrow M)$ in  diagram (\ref{diag:threezero}) yields a fibre sequence $\Omega P_{\ell}(BO) \to \Ml \to M \times BO$ (see \cite[Example 3.3.14]{MunsonVolic:2015}) from which we obtain an  exact sequence
$$ [(M\times I, M \times \bd I), P_\ell(BO)]_{\gamma_\ell} \to \hept{\Ml} \xrightarrow{\ j_* \ }  \hept{M}$$
of groups (generalizing \cite[Lemma 3.16]{Hambleton:2004} for $n=4$ and $\ell =2$)  where the first group is equivalent to $\pi_1(\on{Map}(M,P_{\ell}(BO)),\gamma_{\ell})$. The required exact sequence for the pair $(M, M_0)$ follows directly from the definitions. 
\end{proof}

\begin{theorem}\label{thm:seventhirteen} Let $n \geq 4$ and $\ell = n-2$. If $\pi_1(M)$ is finite,  the group $\hept{\Ml\vv M_0} $ is finite. 
\end{theorem}
\begin{proof} The proof is a reduction to the simply connected case via the universal covering map $p \colon \wM \to M$.     We choose a base point in $M_0 \cap H$, and proceed in three steps:
\begin{enumerate}
\item The kernel $\cK = \ker \left   \{ p^*\colon \hept{M\vv M_0} \to \hept{\widetilde{M} \vv \widetilde{M}_0} \right \}$ is finite;

\item 
The  group $\hept{M\vv M_0} $ is finite;
\item 
 There is an injection
$\hept{\Ml\vv M_0} \to  \hept{M\vv M_0}$.
\end{enumerate}
The proof of part (i) is based on obstruction theory as in  \cite[Chapter VI, Sections 7 \& 12]{Hu:1959} to estimate the size of $\cK$ (see \cite[\S 26]{Olum:1950} for an extension of this theory to non-simple spaces). The details of this part will be given in the proof of Lemma \ref{lem:partone} below. 

The proof of part (ii) is now a consequence of part (i) and the result for the simply-connected case,  namely that  $\hept{\widetilde{M} \vv \widetilde{M} - \mathring D^n} $ is a finite group, which is proved in \cite[Theorem 1.4]{Baues:1996a} (relying on \cite[Addendum: Theorem 2.7']{Aubry:1991}).   To apply the simply-connected result, we  use an engulfing argument (after a isotopy in $\wM$)  for the complements  of open sets $\widetilde H_0 \subset e^n \subset \widetilde H $ embedded in $\wM$. Here $e^n$ is a small $n$-ball in the interior of $\widetilde H$, and $\widetilde H_0 $ is an isotoped copy of $\widetilde H$ in the interior of $e^n$. Then we have the homomorphisms 
$$ \hept{\widetilde{M} \vv \widetilde{M}  \setminus \widetilde H_0}\to\hept{\widetilde{M} \vv \widetilde{M} \setminus e^n} \to  \hept{\widetilde{M} \vv  \widetilde{M}  \setminus \widetilde H} =  \hept{\widetilde{M} \vv \widetilde{M}_0} $$
induced by the inclusions of the complements, and the  composition of the two maps in the sequence is an isomorphism.

  For part (iii), we use the exact sequence
$$ [\Sigma (M, M_0),P_\ell(BO)]_{\gamma_\ell} \to \hept{\Ml\vv M_0} \xrightarrow{\ j_* \ }  \hept{M\vv M_0}$$
given by Lemma \ref{lem:seveneight}. Since $ n\geq 4$ and $\ell = n-2$, obstruction theory shows that the group $ [\Sigma (M, M_0), P_\ell(BO)]_{\gamma_\ell}$ is trivial: the obstructions to a null-homotopy for an element in this group lie in $H^{i-1}(M, M_0; \pi_i(P_\ell(BO)))$, and these vanish since the coefficients $\pi_i(P_\ell(BO)) = 0$ for $i > \ell = n-2$, and the pair $(M, M_0)$ is homotopy equivalent to a CW pair with relative cells only in  dimensions $\geq n-1$ (see Lemma \ref{lem:seventhree}). 
  Alternately, note that $[\Sigma(M/M_0),P_{n-2}(BO)]=\{ \ast\}$ since $M/M_0$ is $(n-2)$-connected.  Hence $\hept{\Ml\vv M_0} $ is a subgroup of the finite  group $\hept{M\vv M_0} $, and the proof is complete. 
\end{proof}

If $M$ is simply-connected, it follows from part (iii) above and \cite[Proposition 1.3]{Baues:1996a} that  $\hept{\Ml\vv M_0}$ is an abelian group. This holds whenever the map $\gamma_0$ is a group homomorphism, for example when $M$ is an oriented $4$-manifold and $\pi_1(M)$ is finite of odd order,  but perhaps not in general (see \cite[p.~169]{Hambleton:2004} for a discussion). 

\begin{question}
	Under which conditions is $\gamma_0$ a group homomorphism ? Is $\hept{\Ml\vv M_0}$ always an abelian group ?
\end{question}

We now return to explain the proof of part (i) in Theorem \ref{thm:seventhirteen}. For expositional purposes we follow the notation and presentation in  \cite[Chapter VI]{Hu:1959} in the setting of simple spaces. The generalization to non-simple spaces is provided by \cite[\S 26]{Olum:1950}. 

\begin{lemma}\label{lem:partone}
In the setting of \textup{Theorem \ref{thm:seventhirteen}}, the kernel $$\cK\coloneqq \ker \left   \{ p^*\colon \hept{M\vv M_0} \to \hept{\widetilde{M} \vv \widetilde{M}_0} \right \}$$
 induced by the universal covering $p \colon \wM \to M$  is finite. 
\end{lemma}
\begin{proof} Any homotopy class in $\hept{M\vv M_0}$ has a representative fixing $M_0$ point-wise (see Lemma \ref{lem:sevenfive}). Let
$$
\cW = \{f\colon M \to M \vv f|_{M_0} = \mu \text{\ and\ } \tilde f \simeq \id_{\wM}\  (\text{rel\ } \wM_0)\}
$$  
be the set of all possible extensions of the inclusion $\mu= i_{M_0}\colon M_0 \to M$,  with the additional property that their lifts to the universal cover satisfy $\tilde f \simeq \id_{\wM}$. Note that every $f \in \cW$ is a homotopy self-equivalence of $M$.
  
The maps in $\cW$ are divided into distinct homotopy classes (relative to $M_0$), and the method of \cite[Chapter VI, Theorem 12.2]{Hu:1959} provides a way to show that $\cK$ is finite by enumerating these classes.  For each $q \geq 1$, the relation of $q$-homotopy relative to $M_0$  (see \cite[Chapter VI, Section 7]{Hu:1959}) divides $\cW$ into disjoint $q$-\emph{homotopy classes relative to $M_0$}. Then every  $(q-1)$-homotopy class $\theta_{q-1}$ relative to $M_0$ contains a certain number of $q$-homotopy classes $\theta_q$ relative to $M_0$,  and we let  $\cW_q(\theta_{q-1})$ denote  the set of  these $q$-homotopy classes.  By 
\cite[Chapter VI, Lemma 9.2]{Hu:1959}, every $(q-1)$-homotopy class $\theta_{q-1}$ determines a subgroup
$J_{\theta_{q-1}}^q(M, M_0; \pi_q(M))$ of the cohomology group $H^q(M, M_0; \pi_q(M))$, and hence a quotient group
  $$Q_{\theta_{q-1}}^q(M, M_0; \pi_q(M)) \coloneqq H^q(M, M_0; \pi_q(M))/ J_{\theta_{q-1}}^q(M, M_0; \pi_q(M)).$$
  The set $\cW_q(\theta_{q-1})$ now determines an \emph{$\theta_{q-1}$-admissible set} $A^q_{\theta_{q-1}}$
  in $Q_{\theta_{q-1}}^q(M, M_0; \pi_q(M))$ (see \cite[p.~188]{Hu:1959}). We now state  the enumeration step (see \cite[Theorem 26.2]{Olum:1950} for the extension to non-simple spaces):
  
  \begin{nonum}[{\cite[Chapter VI, Theorem 12.2]{Hu:1959}}]
  The elements in $\cW_q(\theta_{q-1})$ are in bijection with the elements of the $\theta_{q-1}$-admissible set 
  $A^q_{\theta_{q-1}}$ in $Q_{\theta_{q-1}}^q(M, M_0; \pi_q(M))$.
  \end{nonum}

	To apply this theory, we identify $(M, M_0)$ with the CW-pair $(P, Q)$ in Lemma \ref{lem:seventhree}. In particular, we assume that up to homotopy $M_0$ is the $(n-2)$-skeleton of $M$. This implies that there is a unique $(n-2)$-homotopy class $\theta_{n-2}$ in $\cW$, namely $\theta_{n-2} \coloneqq [\id_M]$, which is the starting point of our count. 
For the  induction step, we let $n \geq q > n-2$ and assume that we have already shown that there are only finitely many  $(q-1)$-homotopy classes in $\cW$. We pick one such $(q-1)$-class $\theta_{q-1}$ and  claim that the set of $q$-classes $\cW_q(\theta_{q-1})$ it contains is finite (details are given below).  After two steps ($q=n-1$ and $q=n$),  this implies that there are only finitely many $n$-homotopy classes in $\cW$. To finish, we note that $H^{q}(M, M_0;\pi_q(M)) = 0$ for $q > n$ because $M$ has dimension $n$. This implies that the finite set of $n$-homotopy classes in $\cW$ equals the set of all homotopy classes in $\cW$, and completes the proof. We now explain the key step for $n \geq q > n-2$ in detail.

	We start by noting that the universal covering $p\colon  \wM \to M$ induces a comparison map 
 $$p^*\colon  H^{q}(M, M_0;\pi_q(M)) \cong H^{q}(M, M_0;\pi_q(\wM)) \to H^{q}(\wM, \wM_0;\pi_q(\wM)),$$
  where $q > 1$, since $n \geq 4$. Let $\tilde{\theta}_{q-1} \coloneqq[\id_{\wM}]$ denote the $(q-1)$-homotopy class relative to $\widetilde{M}_0$ in the set $\widetilde{\cW}$ of lifts to $\wM$ of elements in $\cW$. It follows from \cite[Chapter VI, Proposition 10.2, Lemma 10.5]{Hu:1959} and covering theory,  that $p^*$ sends obstruction sets to obstruction sets. This implies that $p^*$ descends to a comparison map of quotient groups 
 $$\bar p^*\colon Q_{\theta_{q-1}}^q(M, M_0; \pi_q(M)) \to Q_{\tilde{\theta}_{q-1}}^q(\wM, \wM_0; \pi_q(\wM)).$$

	Next, we claim that $J_{\tilde\theta_{q-1}}^q(\wM, \wM_0; \pi_{q}(\wM)) = \{0\}$,  for $q \leq n$. Recall that we identified $(M, M_0)$ with the CW-pair $(P, Q)$ in Lemma \ref{lem:seventhree}. Passing to the universal covers of $(M, M_0)$ and $(P, Q)$, respectively, shows that the inclusion $\wM_0 \subset \wM$ is an $(n-2)$-equivalence as well and yields a compatible homotopy equivalence $(\wM, \wM_0) \simeq (\widetilde{P}, \widetilde{Q})$ to a CW-pair with the properties in Lemma \ref{lem:seventhree}. In particular, we may assume that $\wM_0$ is the $(n-2)$-skeleton of $\wM$. Since $\tilde\theta_{q-1} = [\id_{\wM}]$ for all $q$, it follows from \cite[Chapter VI, Lemma 9.1 and Lemma 9.2]{Hu:1959} that $J_{\tilde\theta_{q-1}}^q(\wM, \wM_0; \pi_{q}(\wM))$ is the set of all $\delta^q(\id_{\wM}, \id_{\wM}; h_t)$ where $h_t$ is any $(q-1)$-homotopy relative $\wM_0$ from $\id_{\wM}$ to $\id_{\wM}$. 
	
	By \cite[Chapter VI, Theorem 8.3]{Hu:1959}, we have $\delta^q(\id_{\wM}, \id_{\wM}; h_t) = 0$ if and only if there exists an $q$-homotopy $h_t^*$ relative $\wM_0$ from $\id_{\wM}$ to $\id_{\wM}$ with the property that $h_t$ and $h_t^*$ restrict to the same $(q-2)$-homotopy. Our assumption, $n \geq q$, implies that $(q-2)$ is at most $(n-2)$. Since we identified $\wM_0 \simeq \widetilde{Q}$ with the $(n-2)$-skeleton of $\wM \simeq \widetilde{P}$ and we work relative to $\wM_0$, it follows that \emph{any} $q$-homotopy relative to $\wM_0$ from $\id_{\wM}$ to $\id_{\wM}$ is a valid choice of $h_t^*$. And, because $\tilde\theta_{q-1} = [\id_{\wM}]$ for all $q$, such an $h_t^*$ exists as long as $q \leq n$. 
	
	The fact that $J_{\tilde\theta_{q-1}}^q(\wM, \wM_0; \pi_{q}(\wM)) = \{0\}$, for $q \leq n$, implies that the induced map 
	$\ker(p^*) \to \ker(\bar p^*) $
	is \emph{surjective}.
	Since 
	$\tilde{f} \simeq \id_{\wM}$ for all $f \in \cW$ by assumption, the lift $\tilde{f}$ of any map $f \in \theta_{q-1}$ satisfies $\tilde{f} \in \tilde{\theta}_{q-1}$ and $\tilde{\theta}_{q-1} = \tilde{\theta}_{q}$. Hence, it follows by
	\cite[Chapter VI, Theorem 11.1]{Hu:1959} that the set $A^q_{\theta_{q-1}} \subseteq Q_{\theta_{q-1}}^q(M, M_0; \pi_q(M))$ is contained in  $\ker(\bar p^*)$.
	Now, $\cW_q(\theta_{q-1})$ is finite if $A^q_{\theta_{q-1}}$ is finite, so by the  surjectivity statement just established,  it is  enough to show that $\ker(p^*)$ is finite.
	
	To see this, we observe there is a transfer map $$\trf^*\colon  H^{q}(\wM, \wM_0;\pi_q(\wM)) \to H^{q}(M, M_0;\pi_q(\wM))$$
   with the property that $\trf^*\circ p^* = |\pi_1(M)|$, because $\pi_1(M)$ is finite. It follows that every element in the subgroup $\ker(p^*) \subseteq H^{q}(M, M_0;\pi_q(M))$ is  torsion annihilated by a power of $|\pi_1(M)|$.  The finiteness of $\pi_1(M)$ furthermore implies that $\pi_q(M)$ is finitely generated, and hence that $H^{q}(M, M_0;\pi_q(M))$ is finitely generated abelian. We conclude that  $\ker(p^*)$ is  finitely generated abelian and, since every element is torsion, $\ker(p^*)$ has to be finite. \end{proof}

As a consequence, we obtain the following generalization  of \cite[Theorem 2]{Triantafillou:1995}
  (compare \cite[Theorem D]{Bustamante:2024}).  Our statement applies to $\hept{\Ml}$ for normal types $\ell = n-2$, but includes a restrictive cohomological assumption for manifolds of dimensions $\geq 6$.

\begin{theorem}\label{thm:sevensixteen}
Let $M$ be a closed connected smooth  or topological $n$-manifold with finite fundamental group, where $n \geq 4$, and let $\ell = n -2$. Suppose that $H^{4i-1}(M ; \bbQ)=0$ for every $i \geq 1$ with $4i \leq \ell$. Then $\hept{\Ml}$ is commensurable up to finite kernel with an arithmetic group. 
\end{theorem}

\begin{proof}
We will apply an argument of Triantafillou \cite[Proposition 12, Corollary 13 and Theorem 14]{Triantafillou:1995}, which has been generalized by Bustamante--Krannich--Kupers \cite[Proposition 2.1 and Proposition 2.7]{Bustamante:2024}.	
	
\nr{Step 1:} Let $\pi = \pi_1(M, x_0)$ denote the fundamental group of $M$.
\begin{enumerate}
\item Since $\pi$ is finite, $\Aut(\pi)$  is a finite group and both kernels 
\begin{align*}
   \hept{M}^{\pi} \coloneqq \ker (\hept{M} \to \Aut(\pi))  \quad 
    \hept{\Ml}^{\pi} \coloneqq \ker (\hept{\Ml} \to \Aut(\pi))
\end{align*}
are finite index subgroups. 
\item Let $V = \bbQ^w \oplus \bbQ$ be the direct sum of the $w$-twisted and the trivial representation. Then $\hept{M}^\pi$ acts naturally  on $H^*(M; V)$ and on $H_*(M; V)$. 
\end{enumerate}
We consider the subgroup $\hept{M}^\pi_h$ whose elements stabilize the (co)homology classes:
\begin{enumerate}
	\item $h_*= \{[M]_{\bbQ} \} \subset H_n(M; \bbQ^w)$, and 
	\item $h^* = \{ \gamma_{\ell}^*(\bar{p}_i) \vv 0 < 4i \leq \ell \} \subset H^*(M; V)$,
	\end{enumerate}
where $\bar{p}_{i} \in H^{4i}(P_\ell(BO); \bbQ)$ is the unique class satisfying $w_{\ell}^*(\bar{p}_i)=p_i \in H^{4i}(BO; \bbQ)$, the universal i-th Pontrjagin class. Since $M$ is a finite CW-complex, \cite[Proposition 2.1]{Bustamante:2024} gives that $\hept{M}^\pi_h$ is commensurable up to finite kernel to an arithmetic group.	

\nr{Step 2:}
Consider a map $q \colon  [M, P_\ell(BO)] \to [M, P_\ell(BO)_{\bbQ}]$ induced by post composition with the rationalization $P_\ell(BO) \to P_\ell(BO)_{\bbQ}$. Since $M$ has the homotopy type of a finite CW-complex, \cite[Theorem 10.2(i)]{Sullivan:1977} shows that $q$ is finite-to-one. Similarly, we have a finite-to-one map $[\Sigma M, P_{\ell}(BO)] \to [\Sigma M, P_{\ell}(BO)_{\bbQ}]$.
	
\nr{Step 3:} 	
The rationalization $P_{\ell}(BO)_{\bbQ}$ is homotopy equivalent to $\prod_{0 < 4i \leq \ell} K(\bbQ, 4i)$. Consequently, $[M, P_{\ell}(BO)_{\bbQ}] \cong \bigoplus_{0 <4i \leq \ell} H^{4i}(M; \bbQ)$ and the class $q([\gamma_{\ell}])$ corresponds to the tuple $( \gamma_{\ell}^* (\bar{p}_i))_i$. A self-equivalence $[f] \in \hept{M}^\pi$ preserves $\gamma_{\ell}$ rationally; that is, $q([\gamma_{\ell} \circ f])= q([\gamma_{\ell}])$, if and only if $f^* \gamma_{\ell}^*(\bar{p}_i) = \gamma_{\ell}^*(\bar{p}_i)$ for every $4i \leq n-2$. It follows that every element of $\hept{M}^\pi_h$ preserves $q([\gamma_{\ell}])$. The stabilizer subgroup is defined as 
$$
\on{Stab}_{\Aut(M)^{\pi}_h}([\gamma_{\ell}]) \coloneqq \{ [f] \in \Aut(M)^{\pi}_h \ \vv \ [\gamma_{\ell} \circ f] = [\gamma_{\ell}] \}.
$$
We will show that $\Aut(\Ml)^{\pi}$ maps onto $\on{Stab}_{\Aut(M)^{\pi}_h}([\gamma_{\ell}])$ with finite kernel, and that $\on{Stab}_{\Aut(M)^{\pi}_h}([\gamma_{\ell}])$ has finite index in $\hept{M}^\pi_h$.  

Note that $H_n(M; \bbZ^w) \cong \bbZ$ implies the injectivity of $H_n(M; \bbZ^w) \to H_n(M; \bbQ^w)$. Therefore, if $[f] \in \hept{M}^\pi$ preserves the rational fundamental class, it also preserves the integral fundamental class. Moreover, $f^*(w)=w$ since $\pi_1(f)$ is the identity map. Thus, elements of $\Aut(M)^{\pi}_h$, in particular $[\gamma_{\ell}]$-stabilizers, preserve local orientation.

\nr{Step 4:} 
Let $\heqpt{M}^{\pi}_+$ be the space of based homotopy self-equivalences of $M$ that induce the identity on $\pi$ and preserve $[M]$. Consider the map
\begin{align*}
    a \colon \heqpt{M}^{\pi}_+ \to \on{Map}_{\bullet}(M, P_{\ell}(BO)), \qquad a(f)= \gamma_{\ell} \circ f.
\end{align*}
By Remark \ref{rem:threethree}, the homotopy fiber of $a$ over $\gamma_{\ell}$ is the subspace of $\heqpt{\Ml}$ whose underlying self-equivalences induce identity on $\pi$. The corresponding exact sequence gives
$$
\pi_1(\heqpt{M}^{\pi}_+, \id_M) \to \pi_1 (\on{Map}_{\bullet}(M, P_{\ell}(BO)), \gamma_{\ell}) \to \hept{\Ml}^{\pi} \xrightarrow{F} \pi_0(\heqpt{M}^{\pi}_+).
$$
Note that the image of $F$ is $\on{Stab}_{\Aut(M)^{\pi}_h}([\gamma_{\ell}])$. Next, we prove that $\pi_1 (\on{Map}_{\bullet}(M, P_{\ell}(BO)), \gamma_{\ell})$ is finite, which implies that $\ker (F)$ is finite.

The space $P_{\ell}(BO)$ inherits a group like infinite loop space structure from $BO$. This gives that $\on{Map}_{\bullet}(M, P_{\ell}(BO))$ is also group like and so $\pi_1(\on{Map}_{\bullet}(M, P_{\ell}(BO)),\gamma_{\ell})\cong [\Sigma M, P_{\ell}(BO)]_{\bullet}$ where the latter group is identified with the fundamental group of the mapping space based at the constant map. Then applying rationalization and using the hypothesis, we have 
$$
[\Sigma M, P_{\ell}(BO)_{\bbQ}]_{\bullet} \cong \bigoplus_{0 < 4i \leq \ell} \widetilde{H}^{4i}(\Sigma M; \bbQ) \cong \bigoplus_{0 < 4i \leq \ell} H^{4i-1}(M; \bbQ) = 0.
$$
From Step 2, we have a finite-to-one map $[\Sigma M, P_{\ell}(BO)] \to [\Sigma M, P_{\ell}(BO)_{\bbQ}]=\{0\}$. Therefore, the map $F \colon \hept{\Ml}^{\pi} \twoheadrightarrow \on{Stab}_{\Aut(M)^{\pi}_h}([\gamma_{\ell}])$ has finite kernel.

\nr{Step 5:} 
By Step 3, every element of $\hept{M}^\pi_h$ fixes $q([\gamma_{\ell}])$ and this group acts on the finite set $q^{-1}\big( q([\gamma_{\ell}]) \bigr)$. The stabilizer of $\gamma_{\ell}$ is precisely $\on{Stab}_{\Aut(M)^{\pi}_h}([\gamma_{\ell}])$ which is a finite index subgroup of $\hept{M}^\pi_h$ by the orbit-stabilizer theorem. Hence, $\hept{\Ml}^{\pi}$ and so $\hept{\Ml}$ are commensurable up to finite kernel with an arithmetic group. 
\end{proof}

We conclude this section with some further questions and remarks.
\subsection{Finite type of the classifying space of $\heqpt{\Ml}$}
Another result of Dror, Dwyer and Kan \cite[Theorem 3.4]{Dror:1981}  states that if $Y$ is a connected virtually nilpotent finite CW complex, then the classifying space of the topological monoid $\heqpt{Y}$ is of finite type. Moreover, each of  its homotopy groups is of finite type. 
In general, we have control of the bordism groups in the braid but not the Postnikov terms $\heqpt{\Bl}$. In order to use the braid to establish a similar result for the spaces of self-equivalences considered in this paper, one would need to extend  \cite[Theorem 3.4]{Dror:1981} to handle the space  $\heqpt{\Bl}$ via an analysis of its Postnikov terms (see
 \cite{Rutter:1992,Rutter:1992b}).
 
\begin{question}
	If $M$ has finite fundamental group, does it hold that   $B\heqpt{\Ml}$ is of finite type? 
		\end{question}
	We remark that in the paper cited above, the result deals with the classifying space of a topological monoid  $\heqpt{Y}$ which is strictly associative and unital, so the bar construction can be used to construct $B \heqpt{Y}$. In our setting $\heqpt{\Ml}$ is a group-like $H$-space and associativity holds only up to a homotopy. Therefore, by $B\heqpt{\Ml}$ we mean that we first strictify $\heqpt{\Ml}$  to a topological monoid and then construct its classifying space.
    
 \begin{remark} We expect that this question will have a positive answer for $\ell = n-2$, under the restrictive cohomological assumptions of Theorem \ref{thm:sevensixteen}, but leave this for an interested reader. 
\end{remark}

\subsection{Stability of $\hept{\Ml}$ under connected sums  $M \# (S^n \times S^n)$}
Several authors have studied homotopy self-equivalences of manifolds with non-empty boundary, particularly in connection with homological stability questions for diffeomorphism groups  (see, for example
\cite{Baues:1996a,Berglund:2013,Berglund:2020,Botvinnik:2017,Galatius:2017,Galatius:2018,Galatius:2020}). One can consider self-equivalences $\hept{M, \bd M}$ which preserve the boundary of $M$,  or variants which  restrict to the identity on the boundary.

 It is natural to ask whether the braid diagrams for a $2n$-dimensional manifold admit compatible stabilization maps under connected sum with $S^n \times S^n$.

It should be interesting to relate the information in the braid diagrams to the diffeomorphisms of $M$.
The odd-dimensional stabilization operation has been less studied so far. If $\dim M = 2n$, let $M_g\coloneqq M \# g(S^n \times S^n)$ denote the $g$-fold connected sum. There is an evident stabilization of the normal $n$-type, denoted $M_g\la n\ra$, by using the standard normal $n$-type of $S^n \times S^n$  (compare \cite[p.~721]{Kreck:1999}). This induces stabilization homomorphisms $\hept{M_g\la n\ra} \to \hept{M_{g+1}\la n\ra}$
compatible with connected sum (see \cite[Section 7.2]{Berglund:2020}). 

\begin{question} Let $M^{2n}$ be a smooth, closed $2n$-dimensional manifold. For a fixed $k > 0$, is there  a stability range for  the homology groups $H_k(\hept{M_g\la n\ra})$ depending on $g$ and $M$?
\end{question}

\begin{remark} Some results in this direction are given in \cite[Theorem 7.6]{Berglund:2020}  and \cite{Bustamante:2024}. 
The braid relates 
 $\hept{M_g\la n\ra}$
 to the ``arithmetic''
group
$\hept{B_g\la n\ra}$, where $B_g\la n\ra$ is the pullback defined by the Postnikov section $P_n(M_g)$.
The stability for  the homology groups $H_k(\hept{B_g\la n\ra})$ can perhaps be studied by extending the known techniques for sequences of arithmetic groups.
\end{remark}

\begin{acknowledgement}
In the final stages of preparation of this manuscript, the authors used ChatGPT solely to locate typos in an earlier draft and to help clarify some details in the proofs of Theorem \ref{thm:sixone} and  Theorem \ref{thm:sevensixteen}. The authors have carefully reviewed all arguments presented here and take full responsibility for the mathematical content and its correctness.
\end{acknowledgement}

\begin{thebibliography}{10}

\bibitem{Aubry:1991}
M.~Aubry and J.-M. Lemaire, \emph{Sur certaines \'{e}quivalences d'homotopies},
  Ann. Inst. Fourier (Grenoble) \textbf{41} (1991), 173--187.

\bibitem{Ayala:2025}
D.~Ayala and J.~Francis, \emph{A parametrized {P}ontryagin--{T}hom theorem},
  arxiv:2512.10274, 2025.

\bibitem{Barcus:1958}
W.~D. Barcus and M.~G. Barratt, \emph{On the homotopy classification of the
  extensions of a fixed map}, Trans. Amer. Math. Soc. \textbf{88} (1958),
  57--74.

\bibitem{Baues:1996a}
H.~J. Baues, \emph{On the group of homotopy equivalences of a manifold}, Trans.
  Amer. Math. Soc. \textbf{348} (1996), 4737--4773.

\bibitem{Baues:2008}
H.~J. Baues and B.~Bleile, \emph{Poincar\'{e} duality complexes in dimension
  four}, Algebr. Geom. Topol. \textbf{8} (2008), 2355--2389.

\bibitem{Berglund:2013}
A.~Berglund and I.~Madsen, \emph{Homological stability of diffeomorphism
  groups}, Pure Appl. Math. Q. \textbf{9} (2013), 1--48.

\bibitem{Berglund:2020}
\bysame, \emph{Rational homotopy theory of automorphisms of manifolds}, Acta
  Math. \textbf{224} (2020), 67--185.

\bibitem{Botvinnik:2017}
B.~Botvinnik and N.~Perlmutter, \emph{Stable moduli spaces of high-dimensional
  handlebodies}, J. Topol. \textbf{10} (2017), 101--163.

\bibitem{Bousfield:1972}
A.~K. Bousfield and D.~M. Kan, \emph{Homotopy limits, completions and
  localizations}, Lecture Notes in Mathematics, Vol. 304, Springer-Verlag,
  Berlin-New York, 1972.

\bibitem{Brown:1962}
M.~Brown, \emph{Locally flat imbeddings of topological manifolds}, Ann. of
  Math. (2) \textbf{75} (1962), 331--341.

\bibitem{Buoncristiano:1976}
S.~Buoncristiano, C.~P. Rourke, and B.~J. Sanderson, \emph{A geometric approach
  to homology theory}, London Mathematical Society Lecture Note Series, No. 18,
  Cambridge University Press, Cambridge-New York-Melbourne, 1976.

\bibitem{Bustamante:2024}
M.~Bustamante, M.~Krannich, and A.~Kupers, \emph{Finiteness properties of
  automorphism spaces of manifolds with finite fundamental group}, Math. Ann.
  \textbf{388} (2024), 3321--3371.

\bibitem{Crowley:2011}
D.~J. Crowley, \emph{On the mapping class groups of {$\#_r(S^p\times S^p)$} for
  {$p=3,7$}}, Math. Z. \textbf{269} (2011), 1189--1199.

\bibitem{Dror:1979}
E.~Dror, W.~G. Dwyer, and D.~M. Kan, \emph{Self-homotopy equivalences of
  {P}ostnikov conjugates}, Proc. Amer. Math. Soc. \textbf{74} (1979), 183--186.

\bibitem{Dror:1981}
\bysame, \emph{Self-homotopy equivalences of virtually nilpotent spaces},
  Comment. Math. Helv. \textbf{56} (1981), 599--614.

\bibitem{Ebert:2019}
J.~Ebert and O.~Randal-Williams, \emph{Semisimplicial spaces}, Algebr. Geom.
  Topol. \textbf{19} (2019), 2099--2150.

\bibitem{Edwards:1971}
R.~D. Edwards and R.~C. Kirby, \emph{Deformations of spaces of imbeddings},
  Ann. of Math. (2) \textbf{93} (1971), 63--88.

\bibitem{Eilenberg:1952}
S.~Eilenberg and N.~Steenrod, \emph{Foundations of algebraic topology},
  Princeton University Press, Princeton, NJ, 1952.

\bibitem{Eilenberg:1950}
S.~Eilenberg and J.~A. Zilber, \emph{Semi-simplicial complexes and singular
  homology}, Ann. of Math. (2) \textbf{51} (1950), 499--513.

\bibitem{Freedman:1990}
M.~H. Freedman and F.~Quinn, \emph{Topology of 4-manifolds}, Princeton
  University Press, Princeton, NJ, 1990.

\bibitem{Galatius:2014}
S.~Galatius and O.~Randal-Williams, \emph{Stable moduli spaces of
  high-dimensional manifolds}, Acta Math. \textbf{212} (2014), 257--377.

\bibitem{Galatius:2017}
\bysame, \emph{Homological stability for moduli spaces of high dimensional
  manifolds. {II}}, Ann. of Math. (2) \textbf{186} (2017), 127--204.

\bibitem{Galatius:2018}
\bysame, \emph{Homological stability for moduli spaces of high dimensional
  manifolds. {I}}, J. Amer. Math. Soc. \textbf{31} (2018), 215--264.

\bibitem{Galatius:2020}
\bysame, \emph{Moduli spaces of manifolds: a user's guide}, Handbook of
  homotopy theory, CRC Press/Chapman Hall Handb. Math. Ser., CRC Press, Boca
  Raton, FL \copyright 2020, pp.~443--485.

\bibitem{Hambleton:2019}
I.~Hambleton, \emph{Two remarks on {W}all's {D}2 problem}, Math. Proc.
  Cambridge Philos. Soc. \textbf{167} (2019), 361--368.

\bibitem{Hambleton:2023}
\bysame, \emph{A stability range for topological 4-manifolds}, Trans. Amer.
  Math. Soc. \textbf{376} (2023), 8769--8793.

\bibitem{Hambleton:1988}
I.~Hambleton and M.~Kreck, \emph{On the classification of topological
  {$4$}-manifolds with finite fundamental group}, Math. Ann. \textbf{280}
  (1988), 85--104.

\bibitem{Hambleton:2004}
\bysame, \emph{Homotopy self-equivalences of 4-manifolds}, Math. Z.
  \textbf{248} (2004), 147--172.

\bibitem{Hambleton:2009}
\bysame, \emph{Erratum: {H}omotopy self-equivalences of 4-manifolds}, Math. Z.
  \textbf{262} (2009), 473--474.

\bibitem{Hambleton:2018}
\bysame, \emph{On the classification of topological 4-manifolds with finite
  fundamental group: corrigendum}, Math. Ann. \textbf{372} (2018), 527--530.

\bibitem{Hatcher:2002}
A.~Hatcher, \emph{Algebraic topology}, Cambridge University Press, Cambridge,
  2002.

\bibitem{Hebestreit:2021}
F.~Hebestreit, M.~Land, W.~L\"{u}ck, and O.~Randal-Williams, \emph{A vanishing
  theorem for tautological classes of aspherical manifolds}, Geom. Topol.
  \textbf{25} (2021), 47--110.

\bibitem{Hegenbarth:2015a}
F.~Hegenbarth, M.~Pamuk, and D.~Repov\v{s}, \emph{{$s$}-cobordism
  classification of 4-manifolds through the group of homotopy
  self-equivalences}, Mediterr. J. Math. \textbf{12} (2015), 1107--1121.

\bibitem{Hu:1959}
S.-t. Hu, \emph{Homotopy theory}, Pure and Applied Mathematics, Vol. VIII,
  Academic Press, New York-London, 1959.

\bibitem{Kahn:1969}
P.~J. Kahn, \emph{Self-equivalences of {$(n-1)$}-connected {$2n$}-manifolds},
  Math. Ann. \textbf{180} (1969), 26--47.

\bibitem{Kan:1957}
D.~M. Kan, \emph{On c. s. s. complexes}, Amer. J. Math. \textbf{79} (1957),
  449--476.

\bibitem{Kirby:1977}
R.~C. Kirby and L.~C. Siebenmann, \emph{Foundational essays on topological
  manifolds, smoothings, and triangulations}, Annals of Mathematics Studies,
  No. 88, Princeton University Press, Princeton, NJ; University of Tokyo Press,
  Tokyo, 1977, With notes by John Milnor and Michael Atiyah.

\bibitem{Kister:1964}
J.~M. Kister, \emph{Microbundles are fibre bundles}, Ann. of Math. (2)
  \textbf{80} (1964), 190--199.

\bibitem{Kreck:1999}
M.~Kreck, \emph{Surgery and duality}, Ann. of Math. (2) \textbf{149} (1999),
  707--754.

\bibitem{Krylov:2003}
N.~A. Krylov, \emph{On the {J}acobi group and the mapping class group of
  {$S^3\times S^3$}}, Trans. Amer. Math. Soc. \textbf{355} (2003), 99--117.

\bibitem{Lashof:1963}
R.~Lashof, \emph{Poincar\'{e} duality and cobordism}, Trans. Amer. Math. Soc.
  \textbf{109} (1963), 257--277.

\bibitem{Laures:2014}
G.~Laures and J.~E. McClure, \emph{Multiplicative properties of {Q}uinn
  spectra}, Forum Math. \textbf{26} (2014), 1117--1185.

\bibitem{Laures:2024}
\bysame, \emph{Commutativity properties of {Q}uinn spectra}, Proc. Roy. Soc.
  Edinburgh Sect. A \textbf{154} (2024), 1848--1936.

\bibitem{May:1992}
J.~P. May, \emph{Simplicial objects in algebraic topology}, Chicago Lectures in
  Mathematics, University of Chicago Press, Chicago, IL, 1992, Reprint of the
  1967 original.

\bibitem{Milnor:1964}
J.~Milnor, \emph{Microbundles. {I}}, Topology \textbf{3} (1964), 53--80.

\bibitem{Milnor:1962}
J.~Milnor, \emph{On axiomatic homology theory}, Pacific Journal of Mathematics
  \textbf{12} (1962), 337--341.

\bibitem{MunsonVolic:2015}
B.~A. Munson and I.~Voli\'c, \emph{Cubical homotopy theory}, New Mathematical
  Monographs, vol.~25, Cambridge University Press, Cambridge, 2015.

\bibitem{Nicas:1982}
A.~J. Nicas, \emph{Induction theorems for groups of homotopy manifold
  structures}, Mem. Amer. Math. Soc. \textbf{39} (1982), vi+108.

\bibitem{Olum:1950}
P.~Olum, \emph{Obstructions to extensions and homotopies}, Ann. of Math. (2)
  \textbf{52} (1950), 1--50.

\bibitem{Pamuk:2009}
M.~Pamuk, \emph{Homotopy self-equivalences of 4-manifolds with {${\rm PD}_2$}
  fundamental group}, Topology Appl. \textbf{156} (2009), 1445--1458.

\bibitem{Pamuk:2010}
\bysame, \emph{Homotopy self-equivalences of 4-manifolds with free fundamental
  group}, Canad. J. Math. \textbf{62} (2010), 1387--1403.

\bibitem{Pamuk:2026}
\bysame, \emph{Homotopy self-equivalences of manifolds}, preprint (2026),
  arXiv:2610.02244.

\bibitem{Quinn:1995}
F.~Quinn, \emph{Assembly maps in bordism-type theories}, Novikov conjectures,
  index theorems and rigidity, {V}ol. 1 ({O}berwolfach, 1993), London Math.
  Soc. Lecture Note Ser., vol. 226, Cambridge Univ. Press, Cambridge, 1995,
  pp.~201--271.

\bibitem{Randal-Williams:2024}
O.~Randal-Williams, \emph{The family signature theorem}, Proc. Roy. Soc.
  Edinburgh Sect. A \textbf{154} (2024), 2024--2067.

\bibitem{Rourke:1970}
C.~P. Rourke and B.~J. Sanderson, \emph{On topological neighbourhoods},
  Compositio Math. \textbf{22} (1970), 387--424.

\bibitem{Rourke:1971}
\bysame, \emph{{$\Delta$}-sets. {I}. {H}omotopy theory}, Quart. J. Math. Oxford
  Ser. (2) \textbf{22} (1971), 321--338.

\bibitem{Rourke:1982}
\bysame, \emph{Introduction to piecewise-linear
  topology}, Springer Study Edition, Springer-Verlag, Berlin-New York, 1982,
  Reprint.

\bibitem{Rutter:1992}
J.~W. Rutter, \emph{The group of homotopy self-equivalences of
  non-simply-connected spaces using {P}ostnikov decompositions}, Proc. Roy.
  Soc. Edinburgh Sect. A \textbf{120} (1992), 47--60.

\bibitem{Rutter:1992b}
\bysame, \emph{The group of homotopy self-equivalences of
  non-simply-connected spaces using {P}ostnikov decompositions. {II}}, Proc.
  Roy. Soc. Edinburgh Sect. A \textbf{122} (1992), 127--135.

\bibitem{Rutter:1997}
\bysame, \emph{Spaces of homotopy self-equivalences}, Lecture Notes in
  Mathematics, vol. 1662, Springer-Verlag, Berlin, 1997, A survey.

\bibitem{Stong:1968}
R.~E. Stong, \emph{Notes on cobordism theory}, Mathematical notes, Princeton
  University Press, Princeton, N.J., 1968.

\bibitem{Sullivan:1977}
D.~Sullivan, \emph{Infinitesimal computations in topology}, Inst. Hautes
  \'{E}tudes Sci. Publ. Math. \textbf{47} (1977), 269--331.

\bibitem{Switzer:2002}
R.~M. Switzer, \emph{Algebraic topology---homotopy and homology}, Classics in
  Mathematics, Springer-Verlag, Berlin, 2002, Reprint of the 1975 original
  [Springer, New York; MR0385836 (52 \#6695)].

\bibitem{Taylor:2008}
L.~R. Taylor, \emph{Unoriented geometric functors}, Forum Math. \textbf{20}
  (2008), 457--467.

\bibitem{Triantafillou:1995}
G.~Triantafillou, \emph{Automorphisms of spaces with finite fundamental group},
  Trans. Amer. Math. Soc. \textbf{347} (1995), 3391--3403.

\bibitem{Wall:1965}
C.~T.~C. Wall, \emph{Finiteness conditions for {C}{W}-complexes}, Ann. of Math.
  (2) \textbf{81} (1965), 56--69.

\bibitem{Wall:1966}
\bysame, \emph{On the exactness of interlocking sequences}, Enseign. Math. (2)
  \textbf{12} (1966), 95--100.

\bibitem{Wall:1999}
\bysame, \emph{Surgery on compact manifolds}, second ed., Mathematical Surveys
  and Monographs, vol.~69, American Mathematical Society, Providence, RI, 1999,
  Edited and with a foreword by A. A. Ranicki.

\bibitem{Whitehead:1978}
G.~W. Whitehead, \emph{Elements of homotopy theory}, Graduate Texts in
  Mathematics, vol.~61, Springer-Verlag, New York-Berlin, 1978.

\bibitem{Whitehead:1950}
J.~H.~C. Whitehead, \emph{Simple homotopy types}, Amer. J. Math. \textbf{72}
  (1950), 1--57.

\end{thebibliography}
\providecommand{\bysame}{\leavevmode\hbox to3em{\hrulefill}\thinspace}
\providecommand{\MR}{\relax\ifhmode\unskip\space\fi MR }
\providecommand{\MRhref}[2]{%
  \href{http://www.ams.org/mathscinet-getitem?mr=#1}{#2}
}
\providecommand{\href}[2]{#2}

\end{document}